\documentclass[11pt]{article}
\usepackage[english,activeacute]{babel}
\usepackage[T1]{fontenc}
\usepackage{subfig,psfrag}
\usepackage[percent]{overpic}
\usepackage{graphicx,amsmath,amssymb,amsthm,xcolor,wasysym,bbm,enumitem}
\usepackage{upgreek}
\usepackage[active]{srcltx}
  \definecolor{refkey}{gray}{0.75}
  \colorlet{labelkey}{blue}
\usepackage{upgreek}
\usepackage{bbm}

\newcommand{\mm}{\vec{\mathrm{m}}}

\newcommand{\mmm}{{\mathrm{m}}}
\newcommand{\m}{\vec{m}}
\newcommand{\n}{\vec{n}}
\renewcommand{\b}{\vec{b}}
\newcommand{\nn}{\vec{\mathrm{n}}}
\newcommand{\bb}{\vec{\mathrm{b}}}
\newcommand{\cc}{\mathrm{c}}

\newcommand{\Z}{\mathbb{Z}}

\renewcommand{\S}{\mathbb{S}}

\newcommand{\R}{\mathbb{R}}
\newcommand{\C}{\mathbb{C}}

\newcommand{\N}{\mathbb{N}}

\newcommand{\ptl}{{\partial}}

 \providecommand{\abs}[1]{|#1 |}

 \providecommand{\norm}[1]{\|#1 \|}

\renewcommand{\Re}{\operatorname{Re}}
\renewcommand{\Im}{\operatorname{Im}}

\newcommand{\Erfc}{\operatorname{Erfc}}
\newcommand{\Erf}{\operatorname{Erf}}

\newcommand{\bqq}{\begin{equation*}}
\newcommand{\eqq}{\end{equation*}}
\newcommand{\bq}{\begin{equation}}
\newcommand{\eq}{\end{equation}}

\newtheorem{thm}{Theorem}[section]

\newtheorem{prop}[thm]{Proposition}
\newtheorem{lema}[thm]{Lemma}
\newtheorem{lemma}[thm]{Lemma}

\newtheorem{cor}[thm]{Corollary}
\newtheorem{remark}[thm]{Remark}

  \renewcommand{\S}{\mathbb{S}}

\theoremstyle{definition}

\begin{document}
\title{Self-similar solutions of the one-dimensional Landau--Lifshitz--Gilbert equation}

\author{
\renewcommand{\thefootnote}{\arabic{footnote}} Susana Guti\'errez\footnotemark[1] \ and Andr\'e de Laire\footnotemark[2]}
 \footnotetext[1]{School of Mathematics,
University of Birmingham, Edgbaston, Birmingham, B15 2TT, United
Kingdom. E-mail: {\tt s.gutierrez@bham.ac.uk} }
\footnotetext[2]{Laboratoire Paul Painlev\'e, Universit\'e Lille 1, 59655 Villeneuve d'Ascq Cedex,
France. E-mail: {\tt andre.de-laire@math.univ-lille1.fr} }

\date{}

\maketitle

\begin{abstract}
We consider the one-dimensional Landau--Lifshitz--Gilbert (LLG)
equation, a model des\-cribing the dynamics for the spin in
ferromagnetic materials. Our main aim is the analytical study of the
bi-parametric family of self-similar solutions of this model.
In the presence of damping, our construction provides a family of
global solutions of the LLG equation which are associated to a
discontinuous initial data of infinite (total) energy, and which are
smooth and have finite energy for all positive times. Special
emphasis will be given to the behaviour of this family of solutions
with respect to the Gilbert damping parameter.

We would like to emphasize that our analysis also includes the study of self-similar solutions of the Schr\"odinger map and the heat flow for harmonic maps into the $2$-sphere as special cases. In particular, the results presented here recover some of the previously known
results in the setting of the $1$d-Schr\"odinger map equation.

\end{abstract}

\tableofcontents

\section{Introduction and statement of results}
 \setcounter{equation}{0}
   \numberwithin{equation}{section}

In this work we consider the one-dimensional
Landau--Lifshitz--Gilbert equation (LLG)
\begin{equation}\label{LL}
\ptl_t \mm= \beta \mm\times \m_{ss} -\alpha \mm \times (\mm\times \mm_{ss}),
\quad s\in \R, \quad t> 0, \tag{LLG}
\end{equation}
where  $\mm=(\mmm_1, \mmm_2, \mmm_3):\R \times (0,\infty)\longrightarrow \S^2$
is the spin vector, $\beta\geq 0$, $\alpha\geq 0$, $\times$ denotes
the usual cross-product in $\R^3$, and $\mathbb{S}^2$ is the unit
sphere in $\mathbb{R}^3$.

Here we have not included the effects of anisotropy or an external
magnetic field. The first term on the right in \eqref{LL} represents
the exchange interaction, while the second one corresponds to the
Gilbert damping term and may be considered as a dissipative term in
the equation of motion. The parameters $\beta\geq 0$ and $\alpha\geq 0$  are the so-called
exchange constant and Gilbert damping coefficient, and take into
account the exchange of energy in the system and the effect of
damping on the spin chain respectively. Note that, by considering
the time-scaling  $\mm(s,t)\to \mm(s,(\alpha^2+\beta^2)^{1/2}t)$, in
what follows we will assume w.l.o.g. that
\bq\label{normalization}
  \alpha, \ \beta\in [0,1] \qquad {\hbox{and}}\qquad \alpha^2+\beta^2=1.
\eq
The Landau--Lifshitz--Gilbert equation was first derived on
phenomenological grounds by L.~Landau and E.~Lifshitz
to describe the dynamics for the magnetization or spin $\mm(s,t)$ in
ferromagnetic materials \cite{landaulifshitz,gilbert}.
The nonlinear evolution equation \eqref{LL} is related to several
physical and mathematical problems and
it has been seen to be a physically relevant model for several
magnetic materials \cite{hubert,kosevich}.  In the setting of the LLG
equation, of particular importance is to consider the effect of
dissipation on the spin \cite{steiner,daniel-lak,daniel-lak1}.

The Landau--Lifshitz family of equations includes as special cases
the well-known heat-flow for harmonic maps and the Schr\"odinger
map equation onto the $2$-sphere. Precisely, when $\beta=0$ (and therefore $\alpha=1$) the LLG
equation reduces
 to the one-dimensional {\it{heat-flow equation for harmonic maps}}
 \begin{equation}\label{HFHM}
   \partial_{t} \mm= - \mm\times(\mm\times \mm_{ss})=\mm_{ss}+|\mm_s|^2 \mm \tag{HFHM}
 \end{equation}
 (notice that $|\mm|^2=1$, and in particular $\mm\cdot
 \mm_{ss}=-|\mm_s|^2$). The opposite limiting case of the LLG equation
 (that is $\alpha=0$, i.e. no dissipation/damping and therefore
 $\beta=1$) corresponds to the {\it{Schr\"odinger map equation onto the sphere}}
 \begin{equation}\label{SM}
   \partial_{t}\mm=\mm\times \mm_{ss}. \tag{SM}
 \end{equation}
Both special cases have been objects of intense research and we refer
the interested reader to \cite{lakshmanan,guo,lin,guan} for surveys.

Of special relevance is the connection of the LLG equation with certain non-linear
Schr\"odinger equations.  This connection is established as follows:
Let us suppose that $\mm$ is the tangent vector of a curve in
$\R^3$, that is $\mm=\vec{X}_s$, for some curve $\vec{X}(s,t)\in
\R^3$ parametrized by the arc-length. It can be shown
\cite{daniel-lak} that if $\mm$ evolves under (\ref{LL}) and we
define the so-called filament function $u$ associated to
$\vec{X}(s,t)$ by
\bq\label{hasimoto}
u(s,t)=\cc(s,t)\displaystyle{e^{i\int_0^s\uptau(\sigma,t)\,d\sigma}}, \eq
in terms of the curvature $\cc$ and torsion $\uptau$ associated to the
curve, then $u$ solves the following non-local non-linear
Schr\"odinger equation with damping
\bq\label{schrodinger}
iu_t+(\beta-i\alpha)u_{ss}+\frac{u}{2}\left(\beta\abs{u}^2+2\alpha\int_0^s
\Im(\bar u u_s)-A(t)\right)=0, \eq where  $A(t)\in\mathbb{R}$
is a time-dependent function defined in terms of the curvature and
torsion and their derivatives at the point $s=0$.
The transformation (\ref{hasimoto}) was first considered in the
undamped case by Hasimoto in \cite{hasimoto}. Notice that if
$\alpha=0$, equation (\ref{schrodinger}) can be transformed into the
well-known completely integrable cubic Schr\"odinger equation.

The main purpose of this paper is the analytical study  of
self-similar solutions of the LLG equation of the form
\bq\label{m-tilde}
\mm(s,t)=\m\left(\frac{s}{\sqrt{t}}\right), \eq for some profile
$\m:\R\to \mathbb{S}^2$, with emphasis on the
behaviour of these solutions with respect to the Gilbert damping
parameter $\alpha\in[0,1]$.

For $\alpha=0$, self-similar solutions have generated considerable
interest
\cite{lakshmanan0,lakshmanan,buttke,vega-gutierrez,delahoz}. We are
not aware of any other study of such solutions for
$\alpha>0$ in the one dimensional case (see \cite{germain-rupflin} for a study of self-similar solutions of the harmonic map flow in higher dimensions). However, Lipniacki~\cite{lipniacki} has studied
self-similar solutions for a related model with nonconstant
arc-length. On the other hand, little is known analytically about the effect of
damping on the evolution of a one-dimensional spin chain. In
particular,  Lakshmanan and Daniel  obtained an explicit solitary
wave solution in \cite{daniel-lak, daniel-lak1} and demonstrated the
damping of the solution in the presence of dissipation in the system. In this setting, we would like to understand how
the dynamics of self-similar solutions to this model is affected by
the introduction of damping in the equations governing the motion of
these curves.

As will be shown in Section~\ref{sec-self-similar} self-similar solutions of
\eqref{LL} of the type \eqref{m-tilde} constitute a bi-parametric
family of solutions $\{\m_{c_0,\alpha}\}_{c_0,\alpha}$ given by
\bq\label{self-similar}
\mm_{c_0,\alpha}(s,t)=\m_{c_0,\alpha}\left(\frac{s}{\sqrt{t}}\right),\qquad
c_0>0, \qquad \alpha\in[0,1],
\eq
where $\m_{c_0,\alpha}$ is the solution of the Serret--Frenet
equations
\begin{equation}\label{serret}
\left\{
\begin{aligned}
 \m'&=c \n,\\
 \n'&=-c\m +\tau \b,\\
 \b'&=-\tau \n,
\end{aligned}
\right.
\end{equation}
with curvature and torsion given respectively by
\bq\label{c-tau}
 c_{c_0,\alpha}(s)=c_0e^{-\frac{\alpha s^2}{4}},
\quad \tau_{c_0,\alpha}(s)=\frac{\beta s}{2}, \eq and initial
conditions \bq\label{IC} {\m}_{c_0,\alpha}(0)=(1,0,0), \quad
{\n}_{c_0,\alpha}(0)=(0,1,0), \quad {\b}_{c_0,\alpha}(0)=(0,0,1).
\eq
The first result of this paper is the following:
\begin{thm}
\label{Theorem1} Let $\alpha\in[0,1]$, $c_0>0$\footnote{The case $c_0=0$ corresponds to the constant solution for \eqref{LL}, that is
 $$
   \mm_{c_0,\alpha}(s,t)=\m\left(\frac{s}{\sqrt{t}}\right)=(1,0,0), \qquad \forall \alpha\in[0,1].
 $$
} and ${\m}_{c_0,\alpha}$ be the solution of the Serret--Frenet
system \eqref{serret} with curvature and torsion given by \eqref{c-tau} and initial conditions \eqref{IC}. Define
$\mm_{c_0,\alpha}(s,t)$ by
$$
   \mm_{c_0, \alpha}(s,t) =\m_{c_0,\alpha}\left( \frac{s}{\sqrt{t}}  \right),   \qquad t>0.
$$
Then,
\begin{enumerate}[label=({\roman*}),ref={\it{({\roman*})}}]
\item\label{regular} The function $\mm_{c_0, \alpha}(\cdot,t)$ is a regular $\mathcal{C}^\infty(\R;\mathbb{S}^2)$-solution of \eqref{LL}  for $t>0$.

\item\label{converge} There exist unitary vectors $\vec{A}^{\pm}_{c_0,\alpha}=(A_{j,c_0,\alpha}^{\pm})_{j=1}^{3} \in \S^2$ such that the following pointwise convergence holds when $t$ goes to zero:
\bq\label{convergencia} \lim_{t\to
 0^+}\mm_{c_0,\alpha}(s,t)=\begin{cases}
 \vec{A}^+_{c_0,\alpha}, &\text{ if }s>0,\\[2ex]
 \vec{A}^-_{c_0,\alpha}, &\text{ if }s<0,
\end{cases}
\eq
where $\vec{A}^-_{c_0,\alpha}=(A_{1,c_0,\alpha}^+,-A_{2,c_0,\alpha}^+,-A_{3,c_0,\alpha}^+)$.
\item\label{rate} Moreover, there exists a constant $C(c_0,\alpha,p)$ such that for
all $t>0$
\bq\label{cota-m} \|\mm_{c_0,\alpha}(\cdot,t)-
\vec{A}^+_{c_0,\alpha}\chi_{(0,\infty)}(\cdot)-
\vec{A}^{-}_{c_0,\alpha}\chi_{(-\infty,0)}(\cdot)\|_{L^p(\R)}\leq
 C(c_0,\alpha,p) t^\frac{1}{2p}, \eq for all $p\in (1,\infty)$. In
addition, if $\alpha>0$, \eqref{cota-m} also holds for $p=1$.
Here, $\chi_E$ denotes the characteristic function of a set $E$.
\end{enumerate}
\end{thm}
%
The graphics in Figure~\ref{fig-tan} depict the profile
$\m_{c_0,\alpha}(s)$ for fixed $c_0=0.8$ and the values of $\alpha=0.01$, $\alpha=0.2$,
and $\alpha=0.4$. In particular it can be observed how the
convergence of $\m_{c_0,\alpha}$ to
$\vec{A}^\pm_{c_0,\alpha}$ is accelerated by the diffusion $\alpha$.
%

\begin{figure}[ht]
\vspace{1truecm}
\hspace*{-1.2cm}
\subfloat[$\alpha=0.01$]{
\begin{overpic}[scale=0.7]{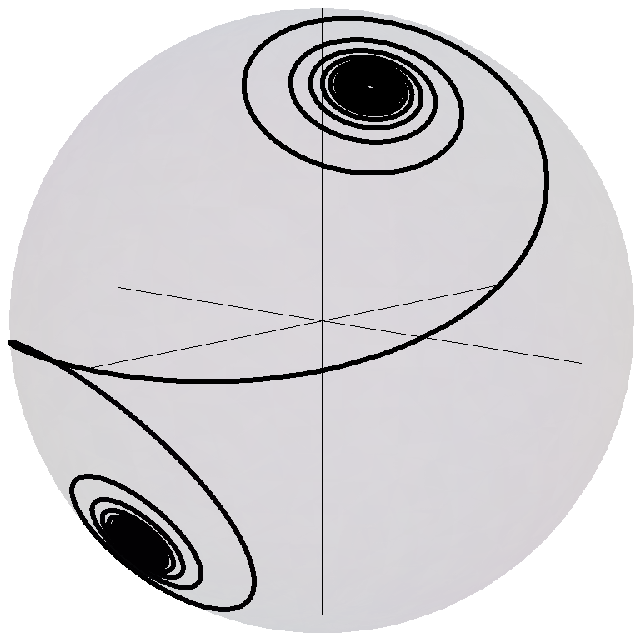}
\put(11,46){\tiny{$m_1$}}
\put(85,47){\tiny{$m_2$}}
\put(48,91){\tiny{$m_3$}}
\end{overpic}
}
\hspace*{-0.9cm}
\subfloat[$\alpha=0.2$]{
\begin{overpic}[scale=0.7]{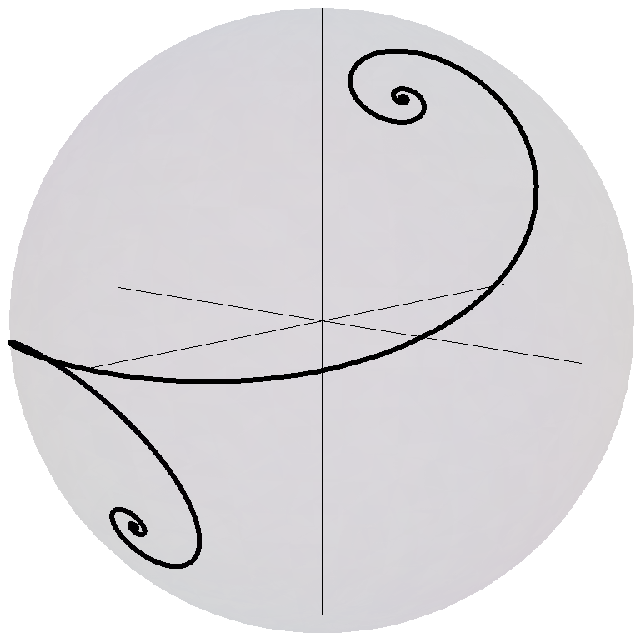}
\put(11,46){\tiny{$m_1$}}
\put(85,47){\tiny{$m_2$}}
\put(48,91){\tiny{$m_3$}}
\end{overpic}
}
\hspace*{-0.9cm}
\subfloat[$\alpha=0.4$]{
\begin{overpic}[scale=0.7]{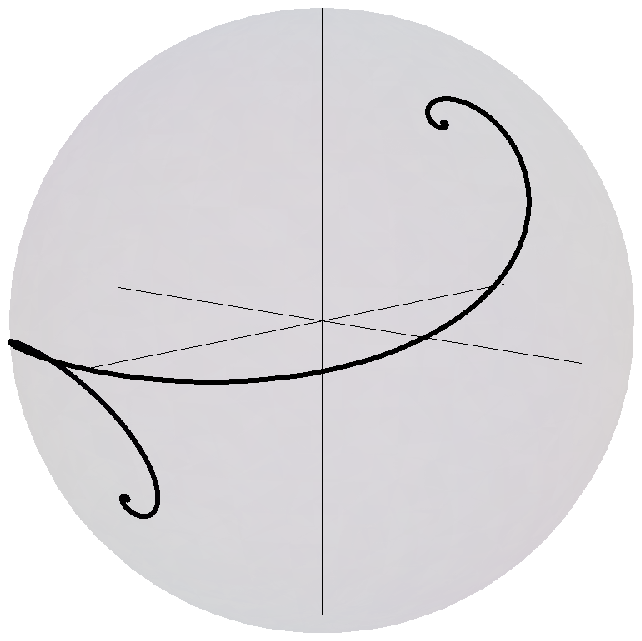}
\put(11,46){\tiny{$m_1$}}
\put(85,47){\tiny{$m_2$}}
\put(48,91){\tiny{$m_3$}}
\end{overpic}
}
\caption{The profile $\m_{c_0,\alpha}$ for $c_0=0.8$ and different values of $\alpha$.}
\label{fig-tan}
\end{figure}
Notice that the initial condition
\begin{equation} \label{data}
 \mm_{c_0,\alpha}(s,0)=\vec{A}^+_{c_0,\alpha}\chi_{(0,\infty)}(s)+\vec{A}^{-}_{c_0,\alpha}\chi_{(-\infty,0)}(s),
\end{equation}
has a jump singularity at the point $s=0$ whenever the vectors
$\vec{A}^+_{c_0,\alpha}$ and $\vec{A}^-_{c_0,\alpha}$ satisfy
$$
 \vec{A}^+_{c_0,\alpha}\neq \vec{A}^-_{c_0,\alpha}.
$$
In this situation (and we will be able to prove analytically this is the case at
least for certain ranges of the parameters $\alpha$ and $c_0$, see Proposition~\ref{jump} below),
Theorem~\ref{Theorem1} provides a bi-parametric family of global
smooth solutions of \eqref{LL} associated to a discontinuous singular  initial
data (jump-singularity).

As has been already mentioned, in the absence of damping ($\alpha=0$), singular self-similar solutions of the Schr\"odinger map equation were previously obtained in \cite{vega-gutierrez}, \cite{lakshmanan0} and  \cite{buttke}. In this framework, Theorem~\ref{Theorem1} establishes the persistence of a jump singularity for self-similar solutions in the presence of dissipation.

Some further remarks on the results stated in
Theorem~\ref{Theorem1} are in order. Firstly, from the self-similar nature of the
solutions $\mm_{c_0,\alpha}(s,t)$ and the Serret--Frenet equations
\eqref{serret}, it follows that the curvature and torsion associated
to these solutions are of the self-similar form and given by
\begin{equation} \label{cur-tor}
 \cc_{c_0,\alpha}(s,t)=\frac{c_0}{\sqrt{t}}e^{-\frac{\alpha s^2}{4t}}
 \qquad {\hbox{and}} \qquad
 \uptau_{c_0,\alpha}(s,t)= \frac{\beta s}{2\sqrt{t}}.
\end{equation}
As a consequence, the total energy $E(t)$ of the spin
$\mm_{c_0,\alpha}(s,t)$ found in Theorem~\ref{Theorem1} is
expressed as
\begin{align}
 \label{total-energy}
  E(t)={}&  \frac{1}{2}\int_{-\infty}^{\infty} |\mm_s(s,t)|^2\, ds=
  \frac{1}{2}\int_{-\infty}^{\infty} \cc_{c_0,\alpha}^2(s,t)\, ds
      \nonumber \\
={}&
  \frac{1}{2}\int_{-\infty}^{\infty} \left(  \frac{c_0}{\sqrt{t}} e^{-\frac{\alpha s^2}{4t}}\right)^2\, ds=
  c_0^2\sqrt{\frac{\pi}{\alpha t}},\qquad \alpha>0, \quad t>0.
\end{align}
It is evident from \eqref{total-energy} that the total energy of the
spin chain at the initial time $t=0$ is infinite, while the total
energy of the spin becomes finite for all positive times, showing
the dissipation of energy in the system in the presence of damping.

Secondly, it is also important to remark that in the setting of Schr\"odinger
equations,  for fixed $\alpha\in[0,1]$ and $c_0>0$, the solution
$\mm_{c_0,\alpha}(s,t)$ of (\ref{LL}) established in
Theorem~\ref{Theorem1} is associated through the Hasimoto
transformation (\ref{hasimoto}) to the filament function
\bq\label{u}
u_{c_0,\alpha}(s,t)=\frac{c_0}{\sqrt{t}}e^{(-\alpha+i\beta)\frac{s^2}{4t}},
\eq
which solves
\begin{equation}
 \label{Schrodinger}
  iu_t+(\beta-i\alpha)u_{ss}+\frac{u}{2}\left(\beta\abs{u}^2+2\alpha\int_0^s
\Im(\bar u u_s)-A(t)\right)=0,\quad {\hbox{with}}\quad A(t)=\frac{\beta c_0^2}{t}
\end{equation}
and is such that at initial time $t=0$
$$
 u_{c_0,\alpha}(s,0)=2c_0\sqrt{\pi(\alpha+i\beta)}\delta_0.
$$
Here $\delta_0$ denotes the delta distribution at the point $s=0$
and $\sqrt{z}$ denotes the square root of a complex number $z$ such
that $\Im (\sqrt{z})>0$.

Notice that the solution
$u_{c_0,\alpha}(s,t)$ is very rough at initial time, and in particular $u_{c_0,\alpha}(s,0)$ does not belong to the Sobolev class $H^s$ for any $s\geq 0$. Therefore, the standard arguments (that is a Picard iteration scheme based on Strichartz estimates and Sobolev-Bourgain spaces) cannot be applied at least not in a straightforward way to study the local well-posedness of the initial value problem for the Schr\"odinger equations (\ref{Schrodinger}). The existence of solutions of the Scr\"odinger equations (\ref{Schrodinger}) associated to an initial data proportional to a Dirac delta opens the question of developing a well-posedness theory for Schr\"odinger equations of the type considered here to include initial data of infinite energy. This question was addressed by  A.~Vargas and L.~Vega in \cite{vargas-vega} and A.~Gr\"unrock in \cite{grunrock} in the case $\alpha=0$ and when $A(t)=0$ (see also \cite{banica-vega1} for a related problem), but we are not aware of any results in this setting when $\alpha>0$ (see \cite{guo} for related well-posedness results in the case $\alpha>0$ for initial data in Sobolev spaces of positive index). Notice that when $\alpha>0$ , the solution \eqref{u} has
infinite energy at the initial time, however the energy becomes
finite for any $t>0$. Moreover, as a consequence of the
exponential decay in the space variable when $\alpha>0$,
$u_{c_0,\alpha}(t)\in H^m(\R)$, for all $t>0$ and $m\in \N$. Hence these solutions do not fit into the usual functional framework for solutions of the Schr\"odinger equations (\ref{Schrodinger}).

As already mentioned, one of the
main goals of this paper is to study both the qualitative and
quantitative effect of the damping parameter $\alpha$ and the
parameter $c_0$ on the dynamical behaviour of  the family  $\{
\mm_{c_0,\alpha} \}_{c_0,\alpha}$  of self-similar solutions of
\eqref{LL} found in Theorem~\ref{Theorem1}. Precisely, in an
attempt to fully understand the regularization of the solution at
positive times close to the initial time $t=0$, and to understand how
the presence of damping affects the dynamical behaviour of these
self-similar solutions, we aim to give answers to the following
questions:
\begin{itemize}
\item[Q1:] Can we obtain a more precise behaviour of the
solutions $\m_{c_0,\alpha}(s,t)$ at positive times $t$ close to
zero?
\item[Q2:] Can we understand the limiting vectors
$\vec{A}^{\pm}_{c_0,\alpha}$ in terms of the parameters $c_0$
and $\alpha$?
\end{itemize}
In order to address our first question, we observe that, due to the self-similar nature of these solutions (see (\ref{self-similar})), the behaviour of the family of solutions
$\mm_{c_0,\alpha}(s,t)$ at positive times close to the initial time $t=0$ is directly related to the study of the asymptotics of the associated profile $\m_{c_0,\alpha}(s)$ for large values of $s$.
 In addition, the symmetries of $\m_{c_0,\alpha}(s)$  (see Theorem~\ref{thm-conver} below) allow to reduce ourselves to obtain the behaviour of the profile  $\m_{c_0,\alpha}(s)$ for large positive values of the space variable. The precise asymptotics of the profile is given in the following theorem.
%
\begin{thm}[Asymptotics] \label{thm-conver}
 Let $\alpha\in[0,1]$, $c_0>0$ and $\{ {\m}_{c_0,\alpha}, {\n}_{c_0,\alpha}  , {\b}_{c_0,\alpha} \}$ be the solution of the Serret--Frenet
system \eqref{serret} with curvature and torsion given by \eqref{c-tau} and initial conditions \eqref{IC}.
Then,
\begin{enumerate}[label=({\roman*}),ref={\it{({\roman*})}}]
\item\label{symmetries} ({\it{Symmetries}}). The components of ${\m}_{c_0,\alpha}(s)$, ${\n}_{c_0,\alpha}(s) $ and ${\b}_{c_0,\alpha}(s)$ satisfy respectively that
 \begin{itemize}
  \item[$\bullet$] $m_{1, c_0,\alpha}(s)$ is an even function,  and $m_{j, c_0,\alpha}(s)$ is an odd function for $j\in\{2,3\}$.
  \item[$\bullet$] $n_{1, c_0,\alpha}(s)$ and $b_{1, c_0,\alpha}(s)$ are odd functions, while $n_{j, c_0,\alpha}(s)$ and  $b_{j, c_0,\alpha}(s)$ are even functions for $j\in\{2,3\}$.
 \end{itemize}
\item\label{asymp} ({\it{Asymptotics}}).  There exist an unit vector $\vec{A}^+_{c_0,\alpha}\in \mathbb{S}^2$ and $\vec{B}^+_{c_0,\alpha}\in \R^3$ such that
the following asymptotics hold for all
$s\geq s_0=4\sqrt{8+c_0^2}$:
\begin{align}
 \m_{c_0,\alpha}(s)
 =&\vec{A}^+_{c_0,\alpha}-\frac{2c_0}{s}\vec{B}^+_{c_0,\alpha}e^{-\alpha s^2/4}(\alpha \sin(\vec{\phi}(s))+\beta \cos(\vec{\phi}(s)))
  \nonumber \\
  & -\frac{2c_0^2}{s^2} \vec{A}^+_{c_0,\alpha}e^{-\alpha s^2/2} +O\left(\frac{e^{-\alpha s^2/4}}{ s^3}\right),
  \label{asym-m}\\
 \vec{n}_{c_0,\alpha}(s)
 =&\vec{B}^+_{c_0,\alpha}\sin(\vec{\phi}(s))+\frac{2c_0}{s}\vec{A}^+_{c_0,\alpha}\alpha e^{-\alpha s^2/4} +O\left(\frac{e^{-\alpha s^2/4}}{ s^2}\right),
  \label{asym-n}\\
 \vec{b}_{c_0,\alpha}(s)
 =&\vec{B}^+_{c_0,\alpha}\cos(\vec{\phi}(s))+\frac{2c_0}{s}\vec{A}^+_{c_0,\alpha}\beta e^{-\alpha s^2/4} +O\left(\frac{e^{-\alpha s^2/4}}{ s^2}\right).\label{asym-b}
\end{align}
Here, $\sin(\vec{\phi})$ and $\cos(\vec{\phi})$ are understood
acting on each of the components of $\vec{\phi}=(\phi_1,\phi_2,\phi_3),$
with
\bq\label{def-phi} \phi_{j}(s)=a_j+\beta
\int_{s_0^2/4}^{s^2/4}\sqrt{1+ c_0^2\frac{e^{-2\alpha
\sigma}}{\sigma}}\,d\sigma, \quad j\in\{1,2,3\},
\eq
for some constants $a_1,a_2,a_3\in [0,2\pi)$, and the vector $\vec{B}^{+}_{c_0,\alpha}$ is given in terms of $\vec{A}^+_{c_0,\alpha}=(A^{+}_{j,c_0,\alpha})_{j=1}^{3}$ by
$$
\vec{B}^+_{c_0,\alpha}=((1-(A_{1,c_0,\alpha}^+)^2)^{1/2},(1-(A_{2,c_0,\alpha}^+)^2)^{1/2},(1-(A_{3, c_0,\alpha}^+)^2)^{1/2}).$$
\end{enumerate}
\end{thm}
As we will see in Section~\ref{sec-self-similar}, the convergence and rate of convergence of the solutions $\mm_{c_0,\alpha}(s,t)$ of the LLG equation established in parts \ref{converge} and \ref{rate} of Theorem~\ref{Theorem1} will be obtained as a consequence of the more refined asymptotic analysis of the associated profile given in Theorem~\ref{thm-conver}.

With regard to the asymptotics of the profile established in part
\ref{asymp} of Theorem~\ref{thm-conver}, it is important to mention
the following:
\begin{enumerate}[label={({\alph*})},ref=({\alph*})]
\item The errors in the asymptotics in
Theorem~\ref{thm-conver}-\ref{asymp} depend only on $c_0$. In other words, the bounds for the errors terms are independent of $\alpha\in [0,1]$. More
precisely, we use the notation $O(f(s))$ to denote a function
for which exists a constant  $C(c_0)>0$ depending on $c_0$, but
independent on $\alpha$, such that \bq\label{cota-rmk}\left|
O\left(f(s)\right)\right|\leq C(c_0)\abs{f(s)}, \quad \text{for
all } s\geq s_0. \eq
\item\label{rem-cont-a} The terms $\vec A^+_{c_0,\alpha}$, $\vec B^+_{c_0,\alpha}$, $B_j^+\sin(a_j)$ and $B_j^+\cos(a_j)$, $j\in\{1,2,3\}$, and
the error terms
in Theorem~~\ref{thm-conver}-\ref{asymp}
depend continuously on $\alpha\in [0,1]$ (see Subsection~\ref{subsec-dependence} and Corollary~\ref{cor-asymp-bis}).
Therefore,  the asymptotics \eqref{asym-m}--\eqref{asym-b} show how the profile $\m_{c_0,\alpha}$ converges to $\m_{c_0,0}$ as $\alpha\to 0^+$
and to $\m_{c_0,1}$ as $\alpha\to 1^-$. In particular, we recover the asymptotics for $\m_{c_0,0}$ given in \cite{vega-gutierrez}.

\item We also remark that using the Serret--Frenet formulae and the
asymptotics in Theorem~~\ref{thm-conver}-\ref{asymp}, it is
straightforward to obtain the asymptotics for the derivatives of
$\mm_{c_0,\alpha}(s,t)$.

\item When $\alpha=0$ and for fixed $j\in\{1,2,3\}$, we can write
$\phi_j$ in \eqref{def-phi} as
\bqq
\phi_{j}(s)=a_j+\frac{s^2}{4}+c_0^2\ln(s)+C(c_0)+O\left(\frac{1}{s^2}\right),
\eqq
and we recover the logarithmic contribution in the oscillation previously found in \cite{vega-gutierrez}. Moreover, in this case the asymptotics in part \ref{asymp} represents an improvement of the one established in Theorem~1 in \cite{vega-gutierrez}.

When $\alpha>0$, $\phi_j$ behaves like

\bq\label{asymp-phi} \phi_{j}(s)=a_j+
\frac{\beta s^2}{4}+C(\alpha,c_0)+O\left(\frac{e^{-\alpha
s^2/2}}{\alpha s^2}\right), \eq
and there is no logarithmic correction in the oscillations in the presence of damping.

Consequently, the phase function $\vec{\phi}$ defined in (\ref{def-phi}) captures the different nature of the oscillatory character of the solutions in both the absence and the presence of damping in the system of equations.
\item
When $\alpha=1$, there exists an explicit formula for $\vec{m}_{c_0,1}$, $\vec{n}_{c_0,1}$ and $\vec{b}_{c_0,1}$, and in particular we have explicit expressions for the vectors $\vec{A}^{\pm}_{c_0,1}$  in terms of the parameter $c_0>0$ in the asymptotics given in part \ref{asymp}.
See Appendix.
\item At first glance, one might think that the term
$-2c_0^2 \vec A^+_{c_0,\alpha}e^{-\alpha s^2/2}/s^2$ in \eqref{asym-m} could be included in the error term
$O(e^{-\alpha s^2/4}/s^3)$. However, we cannot do this because
\bq\label{lambert}
\frac{e^{-\alpha s^2/2}}{s^2}>\frac{e^{-\alpha s^2/4}}{ s^3}, \qquad
\text{ for all } 2\leq s\leq \left(\frac{2}{3\alpha}\right)^{1/2}, \ \alpha\in(0,1/8],
\eq
and in our notation the big-$O$ must be independent of $\alpha$.
(The exact interval where the inequality in \eqref{lambert} holds can be determined using the so-called Lambert $W$ function.)

\item Let $\vec B_{c_0,\alpha,\sin}^+=(B_j\sin(a_j))_{j=1}^3$, $\vec B_{c_0,\alpha,\cos}^+=(B_j\cos(a_j))_{j=1}^3$. Then the orthogonality of $\m_{c_0,\alpha}$, $\n_{c_0,\alpha}$ and $\b_{c_0,\alpha}$ together with the asymptotics \eqref{asym-m}--\eqref{asym-b} yield
\begin{align*}
  \vec A^+_{c_0,\alpha}\cdot \vec B_{c_0,\alpha,\sin}^+=  \vec A^+_{c_0,\alpha}\cdot \vec B_{c_0,\alpha,\cos}^+=  \vec B_{c_0,\alpha,\sin}^+\cdot \vec B_{c_0,\alpha,\cos}^+=0,
\end{align*}
which gives relations between the phases.

\item Finally, the amplitude of the leading order term controlling the wave-like behaviour  of the solution $\m_{c_0,\alpha}(s)$ around
$\vec A^{\pm}_{c_0,\alpha}$ for values of $s$ sufficiently large is of the order $\displaystyle{c_0\, e^{-\alpha s^2/4}/{s}}$,
from which one observes how the convergence of the solution to its limiting values $\vec A^{\pm}_{c_0,\alpha}$ is accelerated in the presence of damping in the system. See Figure~\ref{fig-tan}.
\end{enumerate}
We conclude the introduction by stating the results answering the second of our questions. Precisely,
Theorems~\ref{thm-c-0} and \ref{thm-alpha1-2} below establish the dependence of the vectors $\vec{A}^{\pm}_{c_0,\alpha}$ in Theorem~\ref{Theorem1} with respect to the parameters $\alpha$ and $c_0$.
Theorem~\ref{thm-c-0} provides the behaviour of the limiting vector $\vec{A}^{+}_{c_0,\alpha}$
for a fixed value of $\alpha\in (0,1)$ and ``small'' values of $c_0>0$, while Theorem~\ref{thm-alpha1-2}
states the behaviour of $\vec{A}^{+}_{c_0, \alpha}$ for fixed
$c_0>0$ and $\alpha$ close to the limiting values $\alpha=0$ and
$\alpha=1$. Recall that $\vec{A}_{c_0,\alpha}^{-}$ is expressed in terms of the coordinates of $\vec{A}_{c_0,\alpha}^{+}$ as
\begin{equation}\label{A-}
 \vec{A}_{c_0,\alpha}^{-}=(A_{1,c_0,\alpha}^{+}, -A_{2,c_0,\alpha}^{+}, -A_{3,c_0,\alpha}^{+})
\end{equation}
(see part \ref{converge} of Theorem~\ref{Theorem1}).
\begin{thm}\label{thm-c-0} Let $\alpha\in[0,1]$, $c_0>0$, and $\vec{A}^+_{c_0,\alpha} =(A^{+}_{j,c_0,\alpha})_{j=1}^{3}$
be the unit vector  given  in Theorem~\ref{thm-conver}. Then $\vec{A}^+_{c_0,\alpha}$ is a continuous function of $c_0>0$. Moreover, if
$\alpha\in(0,1]$ the following inequalities hold true:
\begin{align}
&\abs{A_{1, c_0,\alpha}^+ - 1}\leq \frac{c_0^2\pi}{\alpha}\left(1+\frac{c_0^2\pi}{8\alpha}\right),\label{A1}
 \\[2ex]
&\left|A_{2, c_0,\alpha}^+-c_0\frac{\sqrt{\pi(1+\alpha)}}{\sqrt2}\right|\leq \frac{c_0^2\pi}{4}
  +\frac{c_0^2\pi}{\alpha\sqrt{2}}\left(1+\frac{c_0^2 \pi}{8}+c_0\frac{\sqrt{\pi(1+\alpha)}}{2\sqrt2}\right)
  +\left(\frac{c_0^2\pi}{2\sqrt{2}\alpha}\right)^2,
  \label{A2}\\[2ex]
&\left|A_{3, c_0,\alpha}^+-c_0\frac{\sqrt{\pi(1-\alpha)}}{\sqrt2}\right|\leq \frac{c_0^2\pi}{4}
  +\frac{c_0^2\pi}{\alpha\sqrt{2}}\left(1+\frac{c_0^2 \pi}{8}+c_0\frac{\sqrt{\pi(1-\alpha)}}{2\sqrt2}\right)
  +\left(\frac{c_0^2\pi}{2\sqrt{2}\alpha}\right)^2.
  \label{A3}
\end{align}
\end{thm}
The following result  provides an approximation of the behaviour of
$\vec{A}^+_{c_0,\alpha}$ for fixed $c_0>0$ and values of the Gilbert parameter close
to $0$ and $1$.
\begin{thm}\label{thm-alpha1-2} Let $c_0> 0$, $\alpha\in[0,1]$ and  $\vec{A}^{+}_{c_0,\alpha}$ be the unit vector
given in Theorem~\ref{thm-conver}. Then $\vec{A}^+_{c_0,\alpha}$ is a continuous function of $\alpha$ in $[0,1]$,
and the following inequalities hold true:
\begin{align}
 \abs{\vec{A}^+_{c_0,\alpha}-\vec{A}^+_{c_0,0}}&\leq C(c_0)\, \sqrt{\alpha}|\ln(\alpha)|, \quad \hbox{for all }\ \alpha\in(0,1/2],      \label{thm-alpha}  \\
\abs{\vec{A}^+_{c_0,\alpha}-\vec{A}^+_{c_0,1}} &\leq C(c_0)\, \sqrt{1-\alpha}, \quad \text{for all }\ \alpha\in [1/2,1].
       \label{thm-alpha2}
\end{align}
Here, $C(c_0)$ is a positive constant depending on $c_0$ but otherwise independent of $\alpha$.
\end{thm}

As a by-product of Theorems~\ref{thm-c-0} and~\ref{thm-alpha1-2}, we obtain the following proposition which asserts that the solutions
$\mm_{c_0,\alpha}(s,t)$ of the LLG equation found in Theorem~\ref{Theorem1} are indeed associated to a discontinuous initial data at least for certain ranges of $\alpha$ and $c_0$.
\begin{prop} \label{jump}
With the same notation as in Theorems~\ref{Theorem1} and \ref{thm-conver}, the following statements hold:
\begin{itemize}
 \item[{\it{(i)}}] For fixed $\alpha\in(0,1)$ there exists $c_0^{\ast}>0$ depending on $\alpha$ such that
 $$
 \vec{A}_{c_0,\alpha}^{+}\neq \vec{A}_{c_0,\alpha}^{-} \qquad \hbox{for all }\ c_0\in(0, c_0^{\ast}).
 $$
  \item[{\it{(ii)}}] For fixed $c_0>0$, there exists $\alpha^{\ast}_{0}>0$ small enough such that
 $$
 \vec{A}_{c_0,\alpha}^{+}\neq \vec{A}_{c_0,\alpha}^{-} \qquad \hbox{for all }\ \alpha\in(0, \alpha^{\ast}_{0}).
 $$
  \item[{\it{(iii)}}] For fixed $0<c_0\neq k\sqrt{\pi}$ with $k\in \N$, there exists $\alpha^{\ast}_{1}>0$ with $1-\alpha^{\ast}_{1}>0$ small enough such that
 $$
 \vec{A}_{c_0,\alpha}^{+}\neq \vec{A}_{c_0,\alpha}^{-} \qquad \hbox{for all }\ \alpha\in(\alpha^{\ast}_{1},1).
 $$
 \end{itemize}
\end{prop}
\begin{remark}\label{remk-jump}
Based on the numerical results in Section~\ref{section-numerics}, we conjecture that
$\vec{A}_{c_0,\alpha}^{+}\neq \vec{A}_{c_0,\alpha}^{-}$ for all $\alpha\in [0,1)$ and $c_0>0$.
\end{remark}

%
We would like to point out that some of our results and their proofs combine and extend several ideas previously
introduced in \cite{vega-gutierrez} and \cite{vega-gutierrez1}.
The approach we use in the proof of the main results in this paper
is based on the integration of the Serret--Frenet system of equations
via a Riccati equation, which in turn can be reduced to the study of
a second order ordinary differential equation given by
\begin{equation} \label{key-eq}
  f''(s)+\frac{s}{2}(\alpha+i\beta) f'(s)+\frac{c_0^2}{4} e^{-\frac{\alpha s^2}{2}} f(s)=0
\end{equation}
when the curvature and torsion are given by \eqref{c-tau}.

Unlike in the undamped case, in the presence of damping
no explicit solutions are known for equation
(\ref{key-eq}) and the term containing the
exponential in the equation (\ref{key-eq}) makes it difficult to use
Fourier analysis methods to study analytically the behaviour of the
solutions to this equation.
The fundamental step in the analysis of the behaviour of the solutions of (\ref{key-eq}) consists
 in introducing new auxiliary variables $z$, $h$ and  $y$ defined by
$$
 z=|f|^2, \qquad  y=\Re (\bar f f')  \qquad {\hbox{and}}\qquad h=\Im  (\bar f f')
$$
in terms of solutions $f$ of (\ref{key-eq}), and studying the system of equations satisfied by these key quantities.
As we will see later on, these variables are the ``natural'' ones in our problem, in the sense that the components of the tangent, normal and binormal vectors
can be written in terms of these quantities. It is important to emphasize that, in order to obtain error bounds in the asymptotic analysis independent of the damping parameter $\alpha$  (and hence recover the asymptotics when $\alpha=0$ and $\alpha=1$ as particular cases), it will be fundamental to exploit the cancellations due to the oscillatory character of  $z$, $y$ and $h$.

%
%
The outline of this paper is the following.  Section~\ref{sec-self-similar} is devoted to the construction of
the family of self-similar solutions $\{ \mm_{c_0, \alpha}\}_{c_0,\alpha}$ of the LLG equation.
In Section~\ref{sec-reduction} we reduce the study of the properties of this family of self-similar solutions to that of the properties of the
solutions of the complex second order complex ODE (\ref{key-eq}).
This analysis is of independent interest. Section~\ref{sec-proof-results} contains the proofs of the main results of this
paper as a consequence of those established in Section~\ref{sec-reduction}.
In Section~\ref{section-numerics} we give provide some numerical results for $\vec{A}^+_{c_0,\alpha}$,
as a function of $\alpha\in[0,1]$ and $c_0>0$, which give some inside for the scattering problem and justify Remark~\ref{remk-jump}.
Finally, we have included the study of the self-similar solutions of the LLG equation in the case $\alpha=1$ in
Appendix.

\noindent {\bf Acknowledgements.} S.~Guti\'errez and A.~de~Laire
were supported by the British project ``Singular vortex dynamics and
nonlinear Schr\"odinger equations'' (EP/J01155X/1) funded by EPSRC.
S.~Guti\'errez was also supported by the Spanish projects
MTM2011-24054 and IT641-13.

Both authors would like to thank L.~Vega for many enlightening
conversations and for his continuous support.

\section{Self-similar solutions of the LLG equation}\label{sec-self-similar}

\noindent First we derive what we will refer to as the geometric
representation of the LLG equation. To this end, let us assume that $\mm(s,t)=\vec{X}_{s}(s,t)$ for some curve $\vec{X}(s,t)$ in $\R^3$
parametrized with respect to the arc-length with curvature $\cc(s,t)$ and torsion $\uptau(s,t)$. Then, using the Serret--Frenet
 system of equations \eqref{serret}, we have
$$
\mm_{ss}=\cc_s\nn+\cc(-\cc\nn+\uptau \bb),
$$
and thus we can rewrite \eqref{LL} as
\bq\label{der-m-t-1}
\partial_{t} \mm= \beta (\cc_s\bb -\cc\uptau \nn) +\alpha ( \cc\uptau \bb +\cc_s \nn),
\eq
in terms of intrinsic quantities $\cc$, $\uptau$ and the Serret--Frenet
trihedron $\{ \mm,\nn,\bb \}$.

We are interested in self-similar solutions of \eqref{LL} of the
form
\begin{equation}
 \label{m-tilde-1}
 \mm(s,t)=\m\left( \frac{s}{\sqrt{t}} \right)
\end{equation}
for some profile $\m:\R\longrightarrow \mathbb{S}^2$. First,
notice that due to the self-similar nature of $\mm(s,t)$ in
\eqref{m-tilde-1}, from the Serret--Frenet equations \eqref{serret}
it follows that the unitary normal and binormal vectors and the
associated curvature and torsion are self-similar and given by
\begin{equation} \label{n-b-self}
 \nn(s,t)=\n\left( \frac{s}{\sqrt{t}} \right), \qquad
 \bb(s,t)=\b\left( \frac{s}{\sqrt{t}} \right),
\end{equation}
\begin{equation}\label{c-tau-self}
\cc(s,t)=\frac{1}{\sqrt{t}}\,c\left( \frac{s}{\sqrt{t}} \right)
 \quad {\hbox{and}}\quad
\uptau(s,t)=\frac{1}{\sqrt{t}}\,\tau\left( \frac{s}{\sqrt{t}} \right).
\end{equation}
Assume that $\mm(s,t)$ is a solution of the LLG equation, or equivalently of its geometric version \eqref{der-m-t-1}. Then, from \eqref{m-tilde-1}--\eqref{c-tau-self} it follows that the
Serret--Frenet trihedron $\{ \m(\cdot),  \n(\cdot), \b(\cdot)\}$ solves
\begin{equation}
 \label{der-m-t-2}
  -\frac{s}{2} c\n= \beta(c'\b-c\tau\n) +\alpha(c\tau\b+c'\n),
\end{equation}
As a consequence,
$$
 -\frac{s}{2} c=\alpha c' -\beta c\tau
 \quad {\hbox{and}}\quad
  \beta c' +\alpha c\tau=0.
$$
Thus, we obtain
\begin{equation}
 \label{c-tau-1}
  c(s)= c_0\, e^{-\frac{\alpha s^2}{4}}
  \quad {\hbox{and}}\quad
  \tau(s)= \frac{\beta s}{2},
\end{equation}
for some positive constant $c_0$ (recall that we are assuming
w.l.o.g. that $\alpha^2+\beta^2=1$). Therefore, in view of \eqref{c-tau-self}, the curvature and
torsion associated to a self-similar solution of \eqref{LL} of the form
\eqref{m-tilde-1} are given respectively by
\begin{equation}
 \label{c-tau-self-ex}
  \cc(s,t)= \frac{c_0}{\sqrt{t}} e^{-\frac{\alpha s^2}{4t}}
  \quad {\hbox{and}}\quad
  \uptau(s,t)= \frac{\beta s}{2t}, \qquad c_0>0.
\end{equation}
Notice that given $(\cc,\uptau)$ as above, for fixed time $t>0$ one can
solve the Serret--Frenet system of equations to obtain the solution
up to a rigid motion in the space which in general may depend on
$t$. As a consequence, and in order to determine the dynamics of the
spin chain, we need to find the time evolution of the trihedron $\{
\mm(s,t), \nn(s,t), \bb(s,t)\}$ at some fixed point
$s^*\in \R$. To this end, from the above expressions of the
curvature and torsion associated to $\mm(s,t)$ and evaluating
the equation \eqref{der-m-t-1} at the point $s^*=0$, we obtain that
$\mm_t(0,t)=\vec{0}$. On the other hand, differentiating the
geometric equation \eqref{der-m-t-1} with respect to $s$, and using
the Serret--Frenet equations \eqref{serret} together with  the
compatibility condition  $\mm_{st}=\mm_{ts}$, we get the
following relation for the time evolution of the normal vector
$$
 \cc\nn_t= \beta(\cc_{ss}\bb+c^2\uptau \mm-\cc\uptau^2\bb)
 +\alpha((\cc\uptau)_s\bb-\cc\cc_s\mm+\cc_s\uptau\bb).
$$
The evaluation of the above identity at $s^*=0$ together with the
expressions for the curvature and torsion in \eqref{c-tau-self-ex}
yield $\nn_t(0,t)=\vec{0}$. The above argument shows that
$$
 \mm_t(0,t)=\vec{0}, \qquad
 \nn_t(0,t)=\vec{0}\quad {\hbox{and}}\quad
 \bb_t(0,t)=(\mm\times\nn)_t(0,t)=\vec{0}.
$$
Therefore we can assume w.l.o.g. that
$$
 \mm(0,t)=(1,0,0), \qquad
 \nn(0,t)=(0,1,0) \quad {\hbox{and}}\quad
 \bb(0,t)=(0,0,1),
$$
and in particular
\begin{equation}
 \label{IC-1}
 \m(0)=\mm(0,1)=(1,0,0), \quad
 \n(0)=\nn(0,1)=(0,1,0), \quad {\hbox{and}}\quad
 \b(0)=\bb(0,1)=(0,0,1).
\end{equation}
Given $\alpha\in[0,1]$ and $c_0>0$, from the theory of ODE's, it
follows that there exists a unique $\{ \m_{c_0,\alpha}(\cdot),
\n_{c_0,\alpha}(\cdot), \b_{c_0,\alpha}(\cdot)\}\in
\left(\mathcal{C}^{\infty}(\mathbb{R}; \mathbb{S}^{2})\right)^{3}$
solution of the Serret--Frenet equations (\ref{serret}) with
curvature and torsion (\ref{c-tau-1}) and initial conditions
(\ref{IC-1}) such that
$$
 \vec{m}_{c_0,\alpha}\perp \vec{n}_{c_0,\alpha}, \quad
 \vec{m}_{c_0,\alpha}\perp \vec{b}_{c_0,\alpha}, \quad
 \vec{n}_{c_0,\alpha}\perp \vec{b}_{c_0,\alpha}
$$
and
$$
 |\vec{m}_{c_0,\alpha}|^2=|\vec{n}_{c_0,\alpha}|^2=|\vec{b}_{c_0,\alpha}|^2=1.
$$
Define $\mm_{c_0,\alpha}(s,t)$ as
\begin{equation}\label{self-similar-1}
\mm_{c_0,\alpha}(s,t)=\vec{m}_{c_0,\alpha}\left(\frac{s}{\sqrt{t}} \right).
\end{equation}
Then, $\vec{m}_{c_0,\alpha}(\cdot,t)\in
\mathcal{C}^{\infty}\left(\R;\mathbb{S}^2\right)$ for all $t>0$, and
bearing in mind both the relations in
\eqref{n-b-self}--\eqref{c-tau-self} and the fact that the vectors
$\{\vec{m}_{c_0,\alpha}(\cdot), \vec{n}_{c_0,\alpha}(\cdot),
\vec{b}_{c_0,\alpha}(\cdot)\}$ satisfy the identity
(\ref{der-m-t-2}), a straightforward calculation shows that
$\m_{c_0, \alpha}(\cdot,t)$ is a regular
$\mathcal{C}^\infty(\R;\mathbb{S}^2)$-solution of the LLG equation
for all $t>0$. Notice that the case $c_0=0$ yields the constant solution $\m_{0,\alpha}(s,t)=(1,0,0)$. Therefore in what follows we will assume that $c_0>0$.

The rest of the paper is devoted to establish analytical  properties of
the solutions $\{\mm_{c_0,\alpha}(s,t)\}_{c_0,\alpha}$  defined by
(\ref{self-similar-1}) for fixed $\alpha\in[0,1]$ and $c_0>0$. As  already mentioned, due to the self-similar nature of
these solutions,  it suffices to study the properties of the associated profile
$\vec{m}_{c_0,\alpha}(\cdot)$ or, equivalently, of the solution $\{\vec{m}_{c_0,\alpha}, \vec{n}_{c_0,\alpha}, \vec{b}_{c_0,\alpha} \}$ of the Serret--Frenet system (\ref{serret}) with curvature and torsion given
by (\ref{c-tau-1}) and initial conditions (\ref{IC-1}).
As we will continue to see, the analysis of the profile solution $\{\vec{m}_{c_0,\alpha}, \vec{n}_{c_0,\alpha}, \vec{b}_{c_0,\alpha} \}$ can be reduced to the study of the properties of the solutions of a certain second order complex differential equation.

\section{Integration of the Serret--Frenet system}
\label{sec-reduction}

\subsection{Reduction to the study of a second order ODE}\label{sub-sec-reduction}
Classical changes of variables from the differential geometry of
curves allow us to reduce  the nine equations in the Serret--Frenet
system into three complex-valued second order equations (see
\cite{darboux,struik,lamb}). Theses changes of variables are related
to stereographic projection and  this approach was also used in
\cite{vega-gutierrez}. However, their choice of  stereographic
projection has a singularity at the origin, which leads to an
indetermination of the initial conditions of some of the new
variables. For this reason, we consider in the following lemma a
stereographic projection that is compatible with the initial
conditions \eqref{IC-1}. Although the proof of the lemma below is a slight modification of that in
\cite[Subsections 2.12 and 7.3]{lamb},
 we have included its proof here both for the sake of completeness and to clarify to the unfamiliar reader how the integration of the Frenet equations can be reduced to the study of a second order differential equation.
\begin{lemma}\label{def-f}
Let $\vec{m}=(m_j(s))_{j=1}^{3}$, $\vec{n}=(n_j(s))_{j=1}^{3}$ and
$\vec{b}=(b_j(s))_{j=1}^{3}$ be a solution of the Serret--Frenet
equations \eqref{serret} with positive curvature $c$ and torsion $\tau$.
Then, for each $j\in\{1,2,3\}$ the function \bqq
f_j(s)=e^{\frac12\int_0^sc(\sigma)\eta_j(\sigma)\,d\sigma}, \quad
\text{ with }\quad \eta_j(s)=\frac{(n_j(s)+i b_j(s))}{1+m_j(s)},
\eqq solves the equation \bq\label{eq-f-j}
f_j''(s)+\left(i\tau(s)-\frac{c'(s)}{c(s)}\right)f_j'(s)+\frac{c^2(s)}{4}f_j(s)=0,
\eq with initial conditions \bqq f_j(0)=1,\quad
f'_j(0)=\frac{c(0)(n_j(0)+i b_j(0))}{2(1+m_j(0))}. \eqq Moreover,
the coordinates of  $\vec{m}$, $\vec{n}$ and $\vec{b}$ are given in
terms of $f_j$ and $f'_{j}$ by \bq\label{inverse}
m_j(s)=2\left(1+\frac4{c(s)^2}\left|\frac{f_j'(s)}{f_j(s)}\right|^2\right)^{-1}-1,
\quad
n_j(s)+ib_j(s)=\frac{4f'_j(s)}{c(s)f_j(s)}\left(1+\frac4{c(s)^2}\left|\frac{f_j'(s)}{f_j(s)}\right|^2\right)^{-1}.
\eq
The above relations are valid at least as long as $m_j>-1$ and $\abs{f_j}>0$.
\end{lemma}
\begin{proof}
For simplicity, we omit the index $j$. The proof relies on several
transformations that are rather standard in the study of curves.
First we define the complex function
\begin{equation}
\label{N}
N=(n+ib)e^{i\int_0^s \tau(\sigma)\,d\sigma}.
\end{equation}
Then $N'=i\tau N+(n'+ib')e^{i\int_0^s \tau(\sigma)\,d\sigma}$. On
the other hand, the  Serret--Frenet equations imply that
$$n'+ib'=-cm-i\tau Ne^{-i\int_0^s \tau(\sigma)\,d\sigma}.$$
Therefore, setting
$$\psi=c e^{i\int_0^s \tau(\sigma)\,d\sigma},$$
we get \bq\label{der-N} N'=-\psi m. \eq Using again the
Serret--Frenet equations, we also obtain \bq\label{der-m}
m'=\frac12(\overline \psi N+\psi \overline N). \eq Let us consider
now the  auxiliary function
\begin{equation}
\label{phi}
\varphi=
\frac{N}{1+m}.
\end{equation}
Differentiating and using \eqref{der-N}, \eqref{der-m} and
\eqref{phi}
\begin{align*}
  \varphi'
  &=\frac{N'}{1+m}-\frac{Nm'}{(1+m)^2}
  \nonumber \\
  &=\frac{N'}{1+m}-\frac{\varphi m'}{1+m}
  \nonumber \\
  &=-\frac{\varphi^2 \overline \psi}{2}-\frac{\psi}{2(1+m)}(2m+\varphi \overline N).
\end{align*}
Noticing that we can recast the relation  $m^2+n^2+b^2=1$ as
$N\overline N=(1-m)(1+m)$ and recalling the definition of $\varphi $
in \eqref{phi}, we have $\varphi \overline N=1-m$, so that
\begin{equation} \label{eq-phi}
\varphi' +\frac{\varphi^2 \overline \psi}{2}+\frac{\psi}{2}=0.
\end{equation}
Finally, define the stereographic projection of $(m,n,b)$ by
\begin{equation}
 \label{eta}
\eta=\frac{n+ib}{1+m}.
\end{equation}
Observe that  from  the definitions of $N$ and $\varphi$,
respectively in \eqref{N} and \eqref{phi}, we can rewrite $\eta$  as
$$
  \eta=\varphi e^{-i\int_{0}^{s} \tau(\sigma)\, d\sigma},
$$ and from \eqref{eq-phi} it follows that $\eta$ solves the Riccati equation
\begin{equation}
 \label{eq-eta}
 \eta'+i\tau \eta+\frac{c}{2}(\eta^2+1)=0,
\end{equation}
(recall that $\psi=ce^{i\int_{0}^{s}\tau(\sigma)\, d\sigma}$).
Finally, setting
\begin{equation}
 \label{f}
 f(s)=e^{\frac12\int_0^sc(\sigma)\eta(\sigma)\,d\sigma},
\end{equation}
we get \bq\label{der-eta} \eta=\frac{2f'}{c f} \eq and
equation \eqref{eq-f-j} follows from \eqref{eq-eta}. The initial
conditions are an immediate consequence of the definition of $\eta$ and
$f$ in \eqref{eta} and \eqref{f}.

A straightforward calculation shows that the inverse transformation
of the stereographic projection is \bqq
m=\frac{1-\abs{\eta}^2}{1+\abs{\eta}^2}, \quad
n=\frac{2\Re\eta}{1+\abs{\eta}^2}, \quad  b=\frac{2\Im
\eta}{1+\abs{\eta}^2}, \eqq so that we obtain \eqref{inverse} using
\eqref{der-eta} and the above identities.
\end{proof}
Going back to our problem, Lemma~\ref{def-f} reduces the analysis of the solution
$\{\vec{m},\vec{n},\vec{b}\}$ of the
Serret--Frenet system (\ref{serret}) with curvature and torsion given
by (\ref{c-tau-1}) and initial conditions (\ref{IC-1}) to the study
of the second order differential equation
\begin{equation}
\label{eq-f0}
f''(s)+\frac{s}2(\alpha+i\beta)f'(s)+\frac{c_0^2}{4} e^{-\alpha s^2/2}f(s)=0,
\end{equation}
with three initial conditions:
For $(m_1,n_1,b_1)=(1,0,0)$ the associated initial condition for $f_1$ is
\bq\label{ic1} f_1(0)=1, \quad f'_1(0)=0, \eq
for $(m_2,n_2,b_2)=(0,1,0)$  is
\bq\label{ic2} f_2(0)=1, \quad f'_2(0)=\frac{c_0}{2}, \eq
and for $(m_3,n_3,b_3)=(0,0,1)$  is
\bq\label{ic3} f_3(0)=1, \quad  f'_3(0)=\frac{ic_0}{2}. \eq
It is important to notice that, by multiplying \eqref{eq-f0} by
$\bar f'$ and taking the real part, it is easy to see
that
$$
 \frac{d\ }{ds}\left[ \frac{1}{2}\left( e^{\frac{\alpha s^2}{2}}|f'|^2+\frac{c_0^2}{4}|f|^2 \right)  \right]=0.
 $$
Thus,
\begin{equation}
 \label{energy-1}
 E(s):\, = \frac{1}{2}\left( e^{\frac{\alpha s^2}{2}}|f'|^2+\frac{c_0^2}{4}|f|^2 \right) =E_0,
 \qquad \forall \, s\in \mathbb{R},
\end{equation}
with $E_0$ a constant defined by the value of $E(s)$ at some point
$s_0\in\mathbb{R}$. The conservation of the energy $E(s)$ allows us
to simplify the expressions of $m_j$, $n_j$ and $b_j$ for $j\in\{1,2,3\}$ in the formulae \eqref{inverse} in  terms of the
solution $f_j$ to \eqref{eq-f0} associated to the initial conditions
\eqref{ic1}--\eqref{ic3}.

Indeed, on the one hand notice that the energies associated to the
initial conditions \eqref{ic1}--\eqref{ic3} are respectively
\begin{equation}
 \label{energy-i}
   E_{0,1}=\frac{c_0^2}{8}, \qquad E_{0,2}=\frac{c_0^2}{4} \qquad {\hbox{and}}
   \qquad E_{0,3}=\frac{c_0^2}{4}.
\end{equation}
On the other hand, from \eqref{energy-1}, it follows that
$$
\left( 1+\frac{4}{c_0^2 e^{- \frac{\alpha s^2}{2}}} \frac{|f'_j|^2(s)}{|f_j|^2(s)}   \right)^{-1}=
\frac{c_0^2}{8E_{0,j}} |f_j|^2(s), \qquad j\in\{1,2,3\}.
$$
Therefore, from \eqref{energy-i}, the above identity and formulae
\eqref{inverse} in Lemma~\ref{def-f}, we conclude that
\begin{align}
m_1(s)&=2\abs{f_1(s)}^2-1, \qquad
n_1(s)+ib_1(s)=\frac{4}{c_0}e^{\alpha s^2/4}\bar f_1(s) f'_1(s),
 \label{m-1}\\
m_j(s)&=\abs{f_j(s)}^2-1,  \qquad
n_j(s)+ib_j(s)=\frac{2}{c_0}e^{\alpha s^2/4}\bar f_j(s) f'_j(s), \quad j\in\{2,3\}.
 \label{m-2}
\end{align}
The above identities give the expressions of the tangent, normal and
binormal vectors in terms of the solutions $\{f_j\}_{j=1}^{3}$ of
the second order differential equation (\ref{eq-f0}) associated to
the initial conditions (\ref{ic1})--(\ref{ic3}).

By Lemma~\ref{def-f}, the formulae \eqref{m-1} and \eqref{m-2}  are valid as long as ${m_j}>-1$, which is equivalent to the condition $\abs{f_j}\neq 0$.
As shown in Appendix, for $\alpha=1$ there is $\tilde s>0$ such that ${m_j(\tilde s)}=-1$ and then  \eqref{m-1} and \eqref{m-2} are (a priori) valid just in a bounded interval.
However, the trihedron  $\{\m,\n,\b\}$ is defined globally and $f_j$ can also be extended globally as the solution of the linear equation \eqref{eq-f0}.
Then, it is simple to verify that the functions given by the l.h.s. of formulae \eqref{m-1} and \eqref{m-2} satisfy the Serret--Frenet system and hence, by the uniqueness
of the solution, the formulae \eqref{m-1} and \eqref{m-2}  are valid for all $s\in \R$.

\subsection{The second-order equation. Asymptotics}\label{second}

In this section we study the properties of the complex-valued
equation
\bq\label{eq-f}
f''(s)+\frac{s}2(\alpha+i\beta)f'(s)+\frac{c_0^2}{4}f(s) e^{-\alpha
s^2/2}=0,
\eq
for fixed $c_0>0$, $\alpha \in[0,1)$, $\beta>0$ such that $\alpha^2+\beta^2=1$.
We begin noticing that in the case $\alpha=0$, the solution can be
written explicitly in terms of parabolic cylinder functions or
confluent hypergeometric functions (see \cite{abram}). Another
analytical approach using Fourier analysis techniques has been taken
in \cite{vega-gutierrez}, leading to the asymptotics
\bq\label{f-alpha-0}
f(s)=C_1e^{i(c_0^2/2)\ln(s)}+C_2\frac{e^{-is^2/4}}{s}e^{-i(c_0^2/2)\ln(s)}+O(1/s^2),
\eq
as $s\to\infty$, where the constants $C_1$, $C_2$ and $O(1/s^2)$
depend on the initial conditions and $c_0$.

For $\alpha=1$, equation \eqref{eq-f} can be also solved explicitly
and the solution is given by \bqq
f(s)=\frac{2f'(0)}{c_0}\sin\left(\frac{c_0}{2}\int_0^s
e^{-\sigma^2/4}\,d\sigma\right)+ f(0)\cos\left(\frac{c_0}{2}\int_0^s
e^{-\sigma^2/4}\,d\sigma\right).
\eqq
In the case $\alpha\in(0,1)$, one cannot compute the solutions of
\eqref{eq-f} in terms of known functions and we will follow a more
analytical analysis. In contrast with the situation when $\alpha=0$,
it is far from evident to use Fourier analysis to study (\ref{eq-f})
when $\alpha>0$.

For the rest of this section we will assume that $\alpha\in[0,1)$.
In addition, we will also assume that $s>0$ and we
will develop the asymptotic analysis necessary to establish part
\ref{asymp} of Theorem~\ref{thm-conver}.
At this point, it is important to  recall the expressions given in
(\ref{m-1})--(\ref{m-2}) for the coordinates of the tangent, normal
and binormal vectors  associated to our family of solutions of the
LLG equation in terms $f$. Bearing this
in mind, we observe that the study of the asymptotic behaviour of
these vectors are dictated by the asymptotic behaviour of the
variables
\begin{equation}
 \label{def-z}
 z=|f|^2, \quad y=\Re(\bar f f'),
 \quad {\hbox{and}}\quad h=\Im(\bar f f')
\end{equation}
associated to the solution $f$ of (\ref{eq-f}).

As explained in the remark (a) after Theorem~\ref{thm-conver}, we need to work with remainder terms
that are independent of $\alpha$. To this aim, we proceed in two steps: first we found uniform estimates for $\alpha\in[0,1/2]$
in Propositions~\ref{cotas} and~\ref{prop-asymp}, then we treat the case $\alpha\in[1/2,1)$ in Lemma~\ref{lema-sin-beta}.
In Subsection~\ref{subsec-dependence} we provide some continuity results that allows us to take $\alpha\to1^-$ and give the full statement
in Corollary~\ref{cor-asymp-bis}. Finally, notice that these asymptotics lead to the asymptotics for the original equation \eqref{eq-f}
(see Remark~\ref{rem-asymp}).

We begin our analysis by establishing the following:

\begin{prop}\label{cotas}
Let $c_0>0$,  $\alpha\in[0,1)$, $\beta>0$ such that
$\alpha^2+\beta^2=1$, and $f$ be a solution of (\ref{eq-f}).
Define $z$, $y$ and $h$ as $z=|f|^2$ and $y+ih=\bar f f'$. Then
\begin{enumerate}
\item[(i)] There exists $E_0\geq 0$ such that the identity
\begin{equation*}
 \frac{1}{2}\left( e^{\alpha\frac{s^2}{2}}|f'|^2+\frac{c_0^2}{4}|f|^2 \right) =E_0
\end{equation*}
holds true for all $s\in \R$. In particular, $f$, $f'$, $z$, $y$
and $h$ are bounded functions. Moreover, for all $s\in \R$
\begin{align}
&\abs{f(s)}\leq \frac{\sqrt{8E_0}}{c_0}, \quad \abs{f'(s)}\leq \sqrt{2E_0}\,  e^{-\alpha s^2/4},  \label{est-f} \\
&\abs{z(s)}\leq \frac{8E_0}{c_0^2}\qquad {\hbox{and}}\qquad \abs{h(s)}+\abs{y(s)}\leq \frac{8E_0}{c_0} e^{-\alpha s^2/4}.  \label{est-y}
\end{align}
\item[(ii)] The limit
$$
z_\infty :\,=\lim_{s\to \infty}z(s)
$$
exists.
\item[(iii)]
Let $\gamma:=2E_0-c_0^2z_{\infty}/2$ and $s_0=4\sqrt{8+c_0^2}$.
  For all $s\geq s_0$, we have
 \bq\label{z-z-inf-s3} z(s)-z_\infty=-\frac{4}{s}(\alpha y+\beta
 h)-\frac{4\gamma}{s^2}e^{-\alpha s^2/2}+R_0(s),
\eq where \begin{equation}\label{cota-R-0bis} \abs{R_0(s)}\leq
C(E_0,c_0)\frac{e^{-\alpha s^2/4}}{s^3}.
\end{equation}
\end{enumerate}
\end{prop}
\begin{proof}
Part {\it{(i)}} is just the conservation of energy proved in \eqref{energy-1}.
Next, using the conservation law in part {\it{(i)}}, we obtain that
the variables $\{z,y,h\}$ solve the first-order real system
\begin{align}
 z'&=2y,\label{z}\\
 y'&=\beta\frac{s}{2} h-\alpha\frac{s}{2}y+e^{-\alpha s^2/2}\left(2E_0-\frac{c_0^2}{2} z\right),\label{y}\\
 h'&=-\beta \frac{s}{2} y-\alpha\frac{s}{2}h\label{h}.
\end{align}
To show {\it{(ii)}}, plugging  \eqref{z} into \eqref{h} and
integrating from $0$ to some $s>0$  we obtain \bq\label{z-int}
z(s)-\frac1{s}\int_0^s z(\sigma)\,d\sigma =-\frac{4}{\beta
s}\left(h(s)-h(0)+\frac{\alpha}{2}\int_0^s \sigma h(\sigma)\,d\sigma
\right). \eq Also, using the above identity, \bq\label{der-int}
\frac{d}{ds}\left(\frac1s \int_0^s z(\sigma)\,d\sigma
\right)=-\frac{4}{\beta s^2}\left(h(s)-h(0)+\frac{\alpha}{2}\int_0^s
\sigma h(\sigma)\,d\sigma \right). \eq Now, since from part
{\it{(i)}} $|h(s)|\leq \frac{8E_0}{c_0}\, e^{-\alpha s^2/4}$, both
$h$ and $\alpha\int_{0}^{s}\sigma h(\sigma)\, d\sigma$ are bounded
functions, thus from (\ref{der-int}) it follows that the limit of
$\frac1s \int_0^s z$ exists, as $s\to\infty$. Hence \eqref{z-int}
and previous observations conclude that the limit
$z_\infty:=\lim_{s\to \infty}z(s)$ exists and furthermore
\begin{equation}\label{equal-lim}
z_\infty:=\lim_{s\to \infty}z(s)=\lim_{s\to \infty}\frac1{s}\int_0^s
z(\sigma).
\end{equation}

We continue to prove {\it{(iii)}}. Integrating \eqref{der-int}
between $s>0$ and $+\infty$ and using integration by parts, we
obtain
\begin{equation}
 \label{z-int-1}
  z_{\infty}-\frac{1}{s}\int_{0}^{s} z(\sigma)\, d\sigma=
  -\frac{4}{\beta}\int_{s}^{\infty} \frac{h(\sigma)}{\sigma^2}\, d\sigma+
  \frac{4}{\beta} \frac{h(0)}{s}-
  \frac{2\alpha}{\beta}\left[
  \frac{1}{s}\int_{0}^{s}\sigma h(\sigma)\, d\sigma +\int_{s}^{\infty} h(\sigma)\, d\sigma
  \right].
\end{equation}
From (\ref{z-int}) and (\ref{z-int-1}), we get \bq\label{z-z-infty}
z(s)-z_\infty=-\frac{4}{\beta}\frac{h(s)}{s}+\frac{2\alpha}{\beta}\int_s^\infty
h(\sigma)\,d\sigma+\frac{4}{\beta}\int_s^\infty
\frac{h(\sigma)}{\sigma^2}. \eq
In order to compute the integrals in \eqref{z-z-infty},  using
\eqref{z} and \eqref{y}, we write \bqq h=\frac{2}{\beta
}\left(\frac{y'}{s}+\frac{\alpha}{4}z'-\frac{2E_0}{s} e^{-\alpha
s^2/2}+\frac{c_0^2}{2s}ze^{-\alpha s^2/2}\right). \eqq Then,
integrating by parts and using the bound for $y$ in (\ref{est-y}),
\begin{equation}
 \label{int-h}
 \int_s^\infty h(\sigma)=\frac{2}{\beta}\left(-\frac{y}{s}+\int_s^\infty\frac{y}{\sigma^2}
 +\frac{\alpha}{4}(z_\infty-z)
 -2E_0\int_s^\infty \frac{e^{-\alpha \sigma^2/2}}{\sigma}+\frac{c_0^2}{2}\int_s^\infty
 \frac{z}{\sigma}e^{-\alpha \sigma^2/2}
\right).
\end{equation}
Also, from (\ref{z}) and (\ref{z-z-infty}), we obtain
\begin{equation}
 \label{int-h2}
\int_s^\infty \frac{h(\sigma)}{\sigma^2}=\frac{2}{\beta }\left(
\int_s^\infty \frac{y'}{\sigma^3}+
\frac{\alpha}{2}\int_s^\infty \frac{y}{\sigma^2}
-2E_0\int_s^\infty \frac{e^{-\alpha \sigma^2/2}}{\sigma^3}+
\frac{c_0^2}{2} \int_s^\infty \frac{z}{\sigma^3}e^{-\alpha \sigma^2/2}
\right).
\end{equation}
Multiplying  \eqref{z-z-infty} by $\beta^2$, using \eqref{int-h},
\eqref{int-h2} and the identity
$$
 \alpha \int_s^\infty
 \frac{e^{-\alpha \sigma^2/2}}{\sigma^n}=\frac{e^{-\alpha
 s^2/2}}{s^{n+1}}-(n+1)\int_s^\infty \frac{e^{-\alpha
 \sigma^2/2}}{\sigma^{n+2}}, \quad \text{for all }\alpha\geq 0, \
 n\geq 1,
$$
we conclude that
\begin{align}\nonumber
 (\alpha^2+\beta^2)(z-z_\infty)
 =&
 -\frac{4}{s}(\alpha y+\beta
 h)-\frac{8E_0}{s^2}e^{-\alpha s^2/2}
     \\
 &+
 8\alpha \int_s^\infty\frac{y}{\sigma^2}+
 8\int_s^\infty\frac{y'}{\sigma^3}+2c^2_0\int_s^\infty e^{-\alpha
 \sigma^2/2}z\left(\frac{\alpha}\sigma+\frac2{\sigma^3}\right).
  \label{extra-1}
\end{align}
Finally, using (\ref{z}) and the boundedness of $z$ and $y$, an
integration by parts argument shows that
\begin{equation}\label{extra-2}
 8\alpha\int_{s}^{\infty} \frac{y}{\sigma^2}+8\int_{s}^{\infty}\frac{y'}{\sigma^3}=
 -4\alpha\frac{z}{s^2}-8\frac{y}{s^3}
 -12\frac{z}{s^4}+8\int_{s}^{\infty}z\left(\frac{\alpha}{\sigma^3}-\frac{6}{\sigma^5}   \right).
\end{equation}
Bearing in mind that $\alpha^2+\beta^2=1$, from (\ref{extra-1}) and
(\ref{extra-2}), we obtain the following identity
\bq
\begin{split}\label{z-z-inf}
z-z_\infty=&-\frac{4}{s}(\alpha y+\beta
h)-\frac{8E_0}{s^2}e^{-\alpha s^2/2} -4\alpha
\frac{z}{s^2}-8\frac{y}{s^3}
-12\frac{z}{s^4}+8\int_s^\infty z\left(\frac{\alpha }{\sigma^3}+\frac{6}{\sigma^5}\right)\,d\sigma \\
&+2c^2_0\int_s^\infty e^{-\alpha
\sigma^2/2}z\left(\frac{\alpha}\sigma+\frac2{\sigma^3}\right)\,d\sigma,
\end{split}
\eq for all $s>0$. In order to prove {\it{(iii)}}, we first write
$z=z-z_\infty+z_\infty$ and observe that
\begin{align*}
8\alpha \int_s^\infty \frac{z }{\sigma^3}&=8\alpha \int_s^\infty \frac{z-z_\infty }{\sigma^3}+\frac{4\alpha z_\infty}{s^2},\\
 \int_s^\infty \frac{z }{\sigma^5}&= \int_s^\infty \frac{z-z_\infty }{\sigma^5}+\frac{ z_\infty}{4s^4} \qquad {\hbox{and}}\\
\int_s^\infty e^{-\alpha \sigma^2/2}z\left(\frac{\alpha}\sigma+\frac2{\sigma^3}\right)&=
\int_s^\infty e^{-\alpha \sigma^2/2}(z-z_\infty)\left(\frac{\alpha}\sigma+\frac2{\sigma^3}\right)+\frac{z_\infty}{s^2} e^{-\alpha s^2/2}.
\end{align*}
Therefore, we can recast \eqref{z-z-inf}  as
\eqref{z-z-inf-s3} with \bq\label{R-0}
\begin{split}
R_0(s)= &-\frac{4\alpha (z-z_\infty)}{s^2}-\frac{8y}{s^3}
-\frac{12(z-z_\infty)}{s^4}+8\int_s^\infty (z-z_\infty) \left(\frac{\alpha}{\sigma^3}+\frac{6}{\sigma^5}\right)\,d\sigma \\
&+2c^2_0\int_s^\infty e^{-\alpha
\sigma^2/2}(z-z_\infty)\left(\frac{\alpha}\sigma+\frac2{\sigma^3}\right)\,d\sigma.
\end{split}
\eq

Let us take $s_0\geq 1$ to be fixed in what follows. For $t\geq
s_0$, we denote $\norm{\cdot}_t$ the norm of $L^\infty([t,\infty))$.
From the definition of $R_0$ in \eqref{R-0} and the elementary
inequalities \bq\label{int-conv} \alpha \int_s^\infty
\frac{e^{-\alpha \sigma^2/2}}{\sigma^n}\leq \frac{e^{-\alpha s^2/2}
}{s^{n+1}}, \quad \text{ for all }\alpha\geq 0, \quad n\geq 1, \eq
and \bq\label{int-conv2}
 \int_s^\infty \frac{e^{-\alpha \sigma^2/2}}{\sigma^n}\leq \frac{e^{-\alpha s^2/2} }{(n-1)s^{n-1}},
 \quad \text{ for all } \alpha\geq 0, \quad n>1,
\eq we obtain
$$
\norm{R_0}_t\leq \frac{8\norm{y}_t}{t^3}+\frac{4}{t^2}\left(8+c_0^2e^{-\alpha t^2/2}
\right)\norm{z-z_\infty}_t.
$$
Hence, choosing $s_0=4\sqrt{8+c_0^2}$, so that
$\frac{4}{t^2}\left(8+c_0^2e^{-\alpha t^2/2}  \right)\leq 1/2$, from
\eqref{est-y} and \eqref{z-z-inf-s3} we conclude that there exists a
constant $C(E_0,c_0)>0$ such that \bqq \norm{z-z_\infty}_t\leq
\frac{C(E_0,c_0)}{t}e^{-\alpha t^2/4}, \quad \text{ for all
}\alpha\in[0,1)\quad {\hbox{and}}\quad \ t \geq s_0, \eqq which
implies that \bq\label{cota} \abs{z(s)-z_\infty}\leq
\frac{C(E_0,c_0)}{s}e^{-\alpha s^2/4}, \quad \text{ for all
}\alpha\in[0,1), \quad s \geq s_0. \eq

Finally, plugging \eqref{est-y} and \eqref{cota}  into \eqref{R-0}
and bearing in mind the inequalities \eqref{int-conv} and
\eqref{int-conv2}, we deduce that
\begin{equation}
 \label{cota-R0-bis}
  |R_{0}(s)|\leq C(E_{0}, c_0)\, \frac{e^{-\alpha s^2/4}}{s^3},
  \qquad \forall\,  s\geq s_0=4\sqrt{8+c_0^2},
\end{equation}
and the proof of {\it{(iii)}} is completed.
\end{proof}
%
Formula \eqref{z-z-inf-s3} in Proposition~\ref{cotas} gives $z$ in
terms of $y$ and $h$. Therefore, we can reduce our analysis to that
of the variables $y$ and $h$ or, in other words, to that of the
system \eqref{z}--\eqref{h}.
 In fact, a first attempt could be to
define $w=y+ih$, so that from \eqref{y} and \eqref{h}, we have that
$w$ solves
\bq\label{eq-w} \left(
we^{(\alpha+i\beta)s^2/4}\right)'=e^{(-\alpha+i\beta) s^2/4}
\left(\gamma-\frac{c_0^2}{2} (z-z_\infty)\right). \eq
From \eqref{cota} in Proposition~\ref{cotas} and (\ref{eq-w}), we see
that the limit
$w_{*}=\lim_{s\rightarrow \infty}w(s)e^{(\alpha+i\beta)s^2/4}$
 exists (at least when $\alpha\neq 0$), and integrating (\ref{eq-w})
from some $s>0$ to $\infty$ we find that
$$
w(s)=e^{-(\alpha+i\beta)s^2/4}\left(w_{*}-\int_{s}^{\infty}
e^{(-\alpha+i\beta)\sigma^2/4}\left(\gamma-\frac{c_0^2}{2}
(z-z_\infty)\right)d\sigma\right).
$$
In order to obtain an asymptotic expansion, we need to estimate
$\int_{s}^\infty e^{(-\alpha+i\beta)\sigma^2/4}(z-z_\infty)$, for
$s$ large. This can be achieved using \eqref{cota},
\bq\label{cota-mala} \left| \int_{s}^\infty
e^{(-\alpha+i\beta)\sigma^2/4}(z-z_\infty)\,d\sigma\right| \leq
C(E_0,c_0)\int_{s}^\infty
\frac{e^{-\alpha\sigma^2/2}}{\sigma}\,d\sigma \eq
and the asymptotic expansion
\bqq \int_{s}^\infty \frac{e^{-\alpha\sigma^2/2}}{\sigma}\,d\sigma=
e^{-\alpha s^2/2}\left(\frac1{\alpha s^2} -\frac2{\alpha^2s^4}
+\frac{8}{\alpha^3s^6} +\cdots\right). \eqq However this estimate
diverges as $\alpha\to 0$. The problem is that the bound used in
obtaining \eqref{cota-mala}  does not take into account the
cancellations due to the oscillations. Therefore, and in order to obtain
the asymptotic behaviour of $z$, $y$ and $h$ valid for all
$\alpha\in[0,1)$, we need a more refined analysis.
In the next proposition we study the system \eqref{z}--\eqref{h},
where we consider the cancellations due the oscillations (see
Lemma~\ref{int-osc} below). The following result provides estimates that are valid for $s\geq s_1$, for some $s_1$
independent of $\alpha$, if $\alpha$ is small.
\begin{prop}\label{prop-asymp} With the same notation and
terminology as in Proposition~\ref{cotas},  let
$$s_1=\max\left\{4\sqrt{8+c_0^2}, 2c_0\left(\frac{1}{\beta}-1\right)^{1/2}\right\}.$$ Then for all $s\geq s_1$,
\begin{align}
y(s)&=be^{-\alpha s^2/4}\sin(\phi(s_1;s))-\frac{2\alpha \gamma}{s} e^{-\alpha s^2/2}+O\left(\frac{e^{-\alpha s^2/2}}{\beta^2 s^2}\right),\label{asym-y-0}\\
h(s)&=be^{-\alpha s^2/4}\cos(\phi(s_1;s))-\frac{2\beta \gamma}{s} e^{-\alpha s^2/2}+O\left(\frac{e^{-\alpha s^2/2}}{\beta^2 s^2}\right),\label{asym-h-0}
\end{align}
where
\bqq \phi(s_1;s)=a+\beta \int_{s_1^2/4}^{s^2/4}\sqrt{1+
c_0^2\frac{e^{-2\alpha t}}{t}}\,dt,
\eqq
$a\in[0,2\pi)$ is a real constant, and  $b$ is a positive constant given by
\bq\label{b}
b^2=\left(2E_0-\frac{c_0^2}{4}z_{\infty}\right)z_{\infty}.
\eq
\end{prop}
%
%
\begin{proof}
First, notice that plugging the expression for $z(s)-z_\infty$ in
(\ref{z-z-inf-s3}) into (\ref{y}), the system (\ref{y})--(\ref{h})
for the variables $y$ and $h$ rewrites equivalently as
\begin{align}
 y'&=\frac{s}{2} (\beta h-\alpha y)+ \frac{2c_0^2}{s}e^{-\alpha s^2/2} (\beta h +\alpha y)
 +\gamma e^{-\alpha s^2/2} +R_1(s),
 \label{y-1}\\
 h'&=- \frac{s}{2} (\beta y+\alpha h),
 \label{h-1}
\end{align}
where
\begin{equation}
 \label{R-1}
 R_1(s)=-\frac{c_0^2}{2} e^{-\alpha s^2/2}R_0(s)
 +\frac{2c_0^2\gamma e^{-\alpha s^2}}{s^2},
\end{equation}
and $R_{0}$ is given by (\ref{R-0}).

Introducing the new variables, \bq\label{def-u-v} u(t)=e^{\alpha
t}y(2 \sqrt{t}), \quad v(t)=e^{\alpha t}h(2 \sqrt{t}), \eq we recast
\eqref{y-1}--\eqref{h-1} as \bq\label{system}
\begin{pmatrix}
u\\
v
\end{pmatrix}'=\begin{pmatrix}
\alpha K & \beta (1+K)\\
-\beta & 0
\end{pmatrix}
\begin{pmatrix}
u\\
v
\end{pmatrix}
+\begin{pmatrix}
F\\
0
\end{pmatrix},
\eq with \bqq K=\frac{c_0^2  e^{-2\alpha t}}{t}, \quad F=\gamma
\frac{e^{-\alpha t}}{\sqrt t}+\frac{e^{-\alpha t}}{\sqrt
t}R_1(2\sqrt t), \eqq where $R_1$ is the function defined in
(\ref{R-1}). In this way, we can regard \eqref{system} as a
non-autonomous system. It is straightforward to check that the
matrix
\bqq A=\begin{pmatrix}
\alpha K & \beta (1+K)\\
-\beta & 0
\end{pmatrix}\eqq
is diagonalizable, i.e. $A=PDP^{-1}$, with \bqq D=\begin{pmatrix}
\lambda_+ & 0\\
0 & \lambda_-
\end{pmatrix}, \quad
P=\begin{pmatrix}
-\frac{\alpha K}{2\beta}-i{\Delta}^{1/2} & -\frac{\alpha K}{2\beta}+i{\Delta}^{1/2}\\
1 & 1
\end{pmatrix},
\eqq
\begin{equation}
 \label{delta}
\lambda_\pm=\frac{\alpha K}{2}\pm i\beta{\Delta}^{1/2},
\qquad {\hbox{and}}\qquad
\Delta=1+K-\frac{\alpha^2 K^2}{4\beta^2}.
\end{equation}
At this point we remark that the condition $t\geq t_1$, with
$t_1:=s_1^2/4$ and $s_1\geq 2c_0(\frac{1}{\beta}-1)^{1/2}$, implies
that \bq\label{cota-K-beta} 0< K\left(\frac1{\beta}-1\right)\leq 1,
\quad \forall\, t\geq t_1, \eq so that
\bq\label{cota-Delta}
\Delta=1+K-\frac{(1-\beta^2)}{4\beta^2}K^2=\left(1+\frac{K}{2}+\frac{K}{2\beta}\right)
\left(1+\frac{K}2\left(1-\frac1\beta\right)\right)\geq \frac12, \quad \forall\, t\geq t_1. \eq
Thus, defining
\begin{equation}
 \label{w-def}
 w=(w_1,w_2)=P^{-1}(u,v),
\end{equation}
we get \bq\label{sys2}
\left(e^{-\int_{t_1}^tD}w\right)'=e^{-\int_{t_1}^tD}
\left((P^{-1})'P w+P^{-1} \tilde F \right), \eq with $\tilde
F=(F,0)$. From the definition of $w$ and taking into account that
$u$ and $v$ are real functions, we have that $w_1=\bar w_2$ and
therefore the study of \eqref{sys2} reduces to the analysis of the
equation: \bq\label{sys3} \left(e^{-\int_{t_1}^t\lambda_+
}w_1\right)'=e^{-\int_{t_1}^t\lambda_+} G(t),\eq with
$$
G(t)=i\frac{\alpha
K'}{4\beta\Delta^{1/2}}(w_1+\bar
w_1)-\frac{\Delta'}{4\Delta}(w_1-\bar
w_1)+i\frac{F}{2\Delta^{1/2}}.
$$
From \eqref{sys3} we have \bq\label{w1}
w_1(t)=e^{\int_{t_1}^t\lambda_+}\left(
w_1(t_1)+w_\infty-\int_t^\infty e^{-
\int_{t_1}^\tau\lambda_+}G(\tau)\, d \tau \right), \eq with
$$
 w_\infty=\int_{t_1}^\infty e^{- \int_{t_1}^\tau\lambda_+}G(\tau).
$$
Since
\begin{equation}
 \label{w1-def}
 w_1=\frac{iu}{2\Delta^{1/2}}+\frac{v}2+\frac{i\alpha K
v}{4\beta \Delta^{1/2}},
\end{equation}
we recast $G$ as $G=i(G_1+G_2+G_3)$ with \bqq G_1=\frac{\alpha
K'v}{4\beta
\Delta^{1/2}}-\frac{\Delta'}{4\Delta^{3/2}}\left(u+\frac{\alpha K
v}{2\beta}\right), \quad G_2=\frac{\gamma e^{-\alpha t}}{2
t^{1/2}\Delta^{1/2}} \quad {\hbox{and}}\quad G_3=\frac{ e^{-\alpha
t}}{2 t^{1/2}\Delta^{1/2}}R_1(2t^{1/2}). \eqq Now, from the
definition of $K$ and $\Delta$, we have
$$
\begin{array}{ll}
K'=\displaystyle{-K\left(2\alpha+\frac1{t}\right)}, \quad &
K''=\displaystyle{K\left( \left( 2\alpha +\frac{1}{t} \right)^2+\frac{1}{t^2}\right)},
    \\[2ex]
\Delta'=\displaystyle{K'\left(1-\frac{\alpha^2 K}{2\beta^2}\right)} &
{\hbox{and}}\quad
\Delta''=\displaystyle{K\left(  \left(2\alpha+\frac{1}{t}\right)^2
+\frac{1}{t^2}\right)\left( 1-\frac{\alpha^2 K}{2\beta^2} \right)
-\frac{\alpha^2}{2 \beta^2}K^2\left( 2\alpha+\frac{1}{t} \right)^2}.
\end{array}
$$
Also, since $s_{1}=\max\{4\sqrt{8+c_0^2}, 2c_0(1/\beta-1)^{1/2}
\}$, for all $t\geq t_1=s_{1}^2/4$, we have in particular that
$t\geq 8+c_0^2$ and $t\geq c_0^2(1/\beta-1)$, hence
\begin{equation}
 \label{key}
 \frac{c_0^2}{t\beta}=\frac{c_0^2}{t}\left( \frac{1}{\beta}-1\right)+\frac{c_0^2}{t}\leq 2
\end{equation}
and
\begin{equation}
 \label{cota-K}
\left| 1-\frac{\alpha^2 K}{4\beta^2}\right|\leq
1+\frac{1}{4\beta}\left( \frac{c_0^2}{t\beta}  \right)\leq
\frac{2}{\beta}.
\end{equation}
Therefore
\bq\label{cota-K-prima}
 \abs{K'}\leq c_0^2e^{-2\alpha t}\left(\frac{2\alpha
 }{t}+\frac1{t^2}\right),
 \quad \quad
 \abs{\Delta'}\leq \frac{2c_0^2e^{-2\alpha t} }{\beta}\left(\frac{2\alpha}{t}+\frac1{t^2}\right)
\eq and
\begin{equation}
 \label{cota-K-prima-2}
 |\Delta''|\leq \frac{24c_0^2}{\beta} e^{-2\alpha t}\left( \frac{\alpha}{t}+\frac{1}{t^2}  \right).
\end{equation}
From Proposition~\ref{cotas}, $u$ and $v$ are bounded in terms of
the energy. Thus, from the definition of $G_1$ and the estimates
(\ref{cota-K-beta}), (\ref{cota-Delta}) and (\ref{cota-K-prima}), we
obtain
$$
\abs{G_1(t)}\leq \frac{C(E_0,c_0)e^{-2\alpha
t}}{\beta^2}\left(\frac{\alpha }{t}+\frac1{t^2}\right).
$$
Since \bq\label{cota-e} \left|e^{\pm
\int_{t_1}^\tau\lambda_+}\right|\leq 2, \eq we conclude that
\bq\label{cota-G1} \left|\int_t^\infty  e^{-
\int_{t_1}^\tau\lambda_+}G_1(\tau) \, d \tau \right|\leq
\frac{C(E_0,c_0)}{\beta^2} \int_t^\infty e^{-2\alpha
\tau}\left(\frac{\alpha }{\tau}+\frac1{\tau^2}\right)\leq
\frac{C(E_0,c_0)e^{-2\alpha t}}{\beta^2 t}. \eq Here we have used
the  inequality \bq\label{ipp2}
 \alpha\int_t^\infty \frac{ e^{-2\alpha \sigma}}{\sigma^n}\,d\sigma\leq \frac{e^{-2\alpha t}}{2 t^n}, \quad n\geq
 1,
\eq which follows by integrating by parts.

In order to handle the terms involving $G_2$ and $G_3$, we need to
take advantage of the oscillatory character of the involved
integrals, which is exploited in Lemma \ref{int-osc}.
From (\ref{cota-Delta}), (\ref{cota-K-prima}) and
(\ref{cota-K-prima-2}), straightforward calculations show that the
function defined by $f=\gamma/(2t^{1/2}\Delta^{1/2})$ satisfies the
hypothesis in part {\it{(ii)}} of Lemma~\ref{int-osc} with $a=1/2$
and $L=C(E_0, c_0)/\beta$. Thus invoking this lemma with
$f=\gamma/(2t^{1/2}\Delta^{1/2})$ and noticing that
$$
   \frac{1}{\Delta^{1/2}}=1+\left(\frac{1}{\Delta^{1/2}}  -1\right)
$$
and that
 \bqq \left| \frac{1}{\Delta^{1/2}}-1\right| = \left|
\frac{1-\Delta}{\Delta^{1/2}(\Delta^{1/2}+1)}\right| \leq
\frac{|K|\left| 1-\frac{\alpha^2K}{4\beta^2}
\right|}{|\Delta^{1/2}(\Delta^{1/2}+1)|} \leq \frac{2\sqrt{2}
c_0^2}{\beta t}, \eqq where we have used \eqref{cota-Delta} and
\eqref{cota-K}, we conclude that \bq\label{cota-G2} \int_t^\infty
e^{- \int_{t_1}^\tau\lambda_+}G_2(\tau) \, d \tau
=\frac{\gamma}{2(\alpha +i \beta)t^{1/2}}{e^{-\int_{t_1}^t\lambda_+
}e^{-\alpha t}}+R_2(t), \eq with \bqq \abs{R_2(t)}\leq
\frac{C(E_0,c_0)e^{-\alpha t}}{\beta^2 t^{3/2}}. \eqq
%
For $G_3$, we first write explicitly (recall the definition of $R_1$
in (\ref{R-1}))\bq\label{G-3} G_3(t)=-\frac{c_0^2 R_0(2\sqrt
t)e^{-3\alpha t}}{4t^{1/2}\Delta^{1/2}}+\frac{c_0^2\gamma
e^{-5\alpha t}}{4t^{3/2}\Delta^{1/2}}:\, =G_{3,1}(t)+ G_{3,2}(t).
\eq Using \eqref{cota-R0-bis} and (\ref{cota-Delta}), we see that
$\abs{G_{3,1}(t)}\leq C(E_0, c_0) e^{-4\alpha t}/t^2$, so that we
can treat this term as we did for $G_1$ to obtain
\begin{equation}
 \label{cota-G31}
 \left| \int_{t}^{\infty} e^{-\int_{t_1}^{\tau}\lambda_{+}} G_{3,1}(\tau)\, d\tau
 \right|\leq \frac{C(E_0, c_0) e^{-4\alpha t}}{t}.
\end{equation}
For the second term, using (\ref{cota-Delta}), (\ref{cota-K-prima})
and (\ref{key}), it is easy to see that the function $f$ defined by
$f=(c_0^2\gamma)/(4t^{3/2}\Delta^{1/2})$ satisfies
$$
 \abs{f(t)}\leq\frac{C(E_0, c_0)}{t^{3/2}}
 \quad {\hbox{and}}\quad
 \abs{f'(t)}\leq C(E_0, c_0)\left( \frac{\alpha}{t^{3/2}}+\frac{1}{t^{5/2}} \right),
$$
as a consequence, invoking part {\it{(i)}} of Lemma~\ref{int-osc},
we obtain
\begin{equation}
 \label{cota-G32}
 \left| \int_{t}^{\infty} e^{-\int_{t_1}^{\tau}\lambda_{+}} G_{3,2}(\tau)\, d\tau
 \right|\leq \frac{C(E_0, c_0) e^{-5\alpha t}}{\beta {t}^{3/2}}.
\end{equation}
From (\ref{w1}), (\ref{cota-e}), (\ref{cota-G1}), (\ref{cota-G31})
and (\ref{cota-G32}), we deduce that
\begin{equation}
 \label{w1-bis}
 w_1(t)=e^{\int_{t_1}^t\lambda_+}\left( w_1(t_1)+w_\infty\right)
-\frac{\gamma(\beta+i\alpha)}{2t^{1/2}}e^{-\alpha t}+R_3(t)
\quad {\hbox{with}}\quad
\abs{R_3(t)}\leq \frac{C(E_0,c_0)e^{-\alpha t}}{\beta^2 t}.
\end{equation}
%
Now we {\it{claim}} that
\begin{equation}
 \label{E-bound}
  e^{\int_{t_1}^t\lambda_+}=C_{\alpha,c_0}e^{ i\beta I(t)}+H(t),
 \quad \text{ with } \quad I(t)=\int_{t_1}^{t}\sqrt{1+K(\sigma)}\,d \sigma, \quad
 \abs{H(t)}\leq \frac{3c_0^2e^{-2\alpha t}}{t},
\end{equation}
and
$$
C_{\alpha,c_0}=\exp\left(\frac{\alpha}{2} \int_{t_1}^\infty
K\, d\sigma\right) \exp\left( -i \frac{\alpha^2}{4\beta}
\int_{t_1}^\infty
\frac{K^2}{\Delta^{1/2}+(1+K)^{1/2}}\,d\sigma\right).
$$
Indeed, recall that $\lambda_+=\frac{\alpha K}{2}+
i\beta{\Delta}^{1/2}$ so that
\begin{equation}
 \label{H0}
  e^{\int_{t_1}^{t} \lambda_{+}}=e^{\alpha\int_{t_1}^{t} \frac{K}{2}}\,
  e^{i\beta \int_{t_1}^{t} \Delta^{1/2}}.
\end{equation}
First, we notice that \bqq \alpha \int_{t_1}^t \frac{K}{2}=c_0^2
\alpha\int_{t_1}^\infty \frac{ e^{-2\alpha \sigma}}{2\sigma}-c_0^2
\alpha\int_t^\infty \frac{ e^{-2\alpha \sigma}}{2\sigma}, \eqq where
both integrals are finite in view of \eqref{ipp2}.
Moreover, by combining with the fact that  $\abs{ 1-e^{-x}}\leq x$, for $x\geq 0$, we can write $$  \exp\left(-c_0^2
\alpha\int_t^\infty \frac{ e^{-2\alpha
\sigma}}{2\sigma}\right)=1+H_{1}(t), $$ with
\begin{equation}
 \label{H1}
 \abs{H_{1}(t)}\leq \frac{c_0^2e^{-2\alpha t}}{4t}, \quad
\text{ for all } t\geq c_0^2/4.
\end{equation}
The above argument shows that \begin{equation}
\label{H2}e^{\alpha\int_{t_1}^t
\frac{K}{2}}=e^{\alpha\int_{t_1}^\infty \frac{K}{2}}(1+H_{1}(t)),
\end{equation}
with $H_1(t)$ satisfying (\ref{H1}).

For the second term of the eigenvalue, using the definition of
$\Delta$ in (\ref{delta}), we write
\begin{align*}
i\beta \int_{t_1}^t \Delta^{1/2}
&=i\beta \int_{t_1}^t \left(\Delta^{1/2}- \sqrt{1+K}\right)+i\beta \int_{t_1}^t \sqrt{1+K}\\
&=-i\frac{\alpha^2 }{4\beta}\int_{t_1}^t \frac{K^2}{\Delta^{1/2}+(1+K)^{1/2}}+i\beta\int_{t_1}^t \sqrt{1+K}.
\end{align*}
Proceeding as before and using that $\abs{1-e^{ix}}\leq \abs{x}$,
for $x\in \R$, and that \bqq \alpha\int_{t}^\infty
\frac{K^2}{\Delta^{1/2}+\sqrt{1+K}}\leq \alpha\int_{t}^{\infty}
K^2(\sigma)\, d\sigma =\alpha c_0^4 \int_{t}^\infty
\frac{e^{-4\alpha \tau}}{\tau^2}\leq \frac{c_0^4e^{-4\alpha
t}}{4t^2}, \eqq we conclude that
\begin{equation}
 \label{H3}
 e^{ i \beta \int_{t_1}^t \Delta^{1/2}}=
 e^{i\beta I(t)}e^{ -i \frac{\alpha^2}{4\beta}
\int_{t_1}^\infty \frac{K^2}{\Delta^{1/2}+(1+K)^{1/2}}}(1+H_{2}),
\end{equation}
with
$$
 \abs{H_{2}(t)}\leq \frac{c_0^4e^{-4\alpha
t}}{16\beta t^2}\leq \frac{c_0^2 e^{-4\alpha t}}{8t},
$$
bearing in mind (\ref{key}). Therefore, from \eqref{H0}, \eqref{H2}
and \eqref{H3},
$$
e^{\int_{t_1}^{t}\lambda_{+}}= C_{\alpha,c_0} e^{i\beta I(t)}(1+H_1(t))(1+H_2(t)).
$$
The claim follows from the above identity, the bounds for $H_1$ and
$H_2$, and the fact that $C_{\alpha, c_0}$ satisfies that
$|C_{\alpha,c_0}|=|e^{\int_{t_1}^{\infty} \lambda_{+}}|\leq 2$ (see
(\ref{cota-e})).
From (\ref{w1-bis}), the claim and writing
\begin{equation}
\label{cont}
 C_{\alpha,c_0}(w_1(t_1)+w_{\infty})= (be^{ia})/2
\end{equation}
for some real constants $a$ and $b$ such that $b\geq 0$ and
$a\in[0,2\pi)$, it follows that
\begin{equation}
 \label{w_1-asymp}
 w_1(t)=\frac{b}{2} e^{i(\beta I(t)+a)}-\frac{\gamma(\beta+i\alpha)}{2t^{1/2}}+R_{w_1}(t)
 \quad {\hbox{with}}\quad
 \abs{R_{w_1}(t)}\leq \frac{C(E_0,c_0)e^{-\alpha t}}{\beta^2 t}.
\end{equation}
The above bound for $R_{w_1}(t)$ easily follows from the bounds for
$R_3(t)$ and $H(t)$ in (\ref{w1-bis}) and (\ref{E-bound})
respectively, and the fact that
\begin{equation}
 \label{w1-bound}
 \abs{w_1(t)}\leq C(E_0, c_0), \qquad \forall\,  t\geq t_1.
\end{equation}
This last inequality is a consequence of
(\ref{def-u-v}),
(\ref{cota-Delta}), (\ref{w1-def}), (\ref{key})
and the bounds
for $y$ and $h$ established in (\ref{est-y}) in
Proposition~\ref{cotas}.

Going back to the definition of $w$ in (\ref{w-def}), we have
$(u,v)=P(w_1, w_2)$, that is
\begin{equation}
\label{u-v-asymp}
\begin{array}{l}
 u=\displaystyle{-\frac{\alpha K}{2\beta}(w_1+\bar w_1)-i{\Delta^{1/2}}(w_1-\bar w_1)=2\Im (w_1)+R_4(t)},
   \\[2ex]
 v= (w_1+\bar w_1)=2\Re(w_1),
\end{array}
\end{equation}
with
\begin{align*}
 \abs{R_4(t)}=&\left| \frac{-\alpha K}{\beta}\Re(w_1) +2(\Delta^{1/2}-1)\Im(w_1)   \right|
   \\
 &
 \leq \frac{K}{\beta}\abs{\Re(w_1)}+2\frac{\abs{\Delta-1}}{\Delta^{1/2}+1}\abs{\Im(w_1)}
     \\
 &
 \leq \frac{2c_0^2e^{-2\alpha t}}{\beta t}(\abs{\Re(w_1)}+\abs{\Im(w_1)})
 \leq \frac{C(E_0, c_0)e^{-2\alpha t}}{\beta t},
\end{align*}
where we have used (\ref{cota-Delta}), (\ref{cota-K}), and
(\ref{w1-bound}). From (\ref{w_1-asymp}) and (\ref{u-v-asymp}), we obtain
\begin{align*}
u(t)&=b\sin(\beta I(t)+a)-\frac{\alpha \gamma}{t^{1/2}} e^{-\alpha t}+R_5(t),\\
v(t)&=b\cos(\beta I(t)+a)-\frac{\beta \gamma}{t^{1/2}} e^{-\alpha t}+R_6(t),
\end{align*}
with
$$
 \abs{R_5(t)}+\abs{R_6(t)}\leq C(E_0,c_0)e^{-\alpha t}/(\beta^2 t).
$$
The asymptotics for $y$ and $h$ given in \eqref{asym-y-0} and
\eqref{asym-h-0}  are a direct consequence of \eqref{def-u-v} and
the above identities and bounds.

Finally, we compute the value of $b$. In fact, from \eqref{asym-y-0}
and \eqref{asym-h-0} \bqq \lim_{s\to \infty}(y^2(s)+h^2(s))e^{\alpha
s^2/2}=b^2. \eqq On the other hand, since $y+ih=\bar f f'$ and using
the conservation of energy (\ref{energy-1}) \bqq
\left(y^2(s)+h^2(s)\right)e^{\alpha s^2/2}=|y+ih|^2(s) e^{\alpha
s^2/2}= \abs{f'}^2\abs{f}^2e^{\alpha
s^2/2}=(2E_0-\frac{c_0^2}{4}\abs{f}^2)\abs{f}^2, \eqq so that,
taking the limit as $s\to \infty$ and recalling that $z=|f|^2$,
\eqref{b} follows.
\end{proof}
\begin{remark}\label{rem-continuity}
%
%
%
From the definitions of $b$ in \eqref{b}, and $be^{ia}$ in \eqref{cont} (in terms of $C_{\alpha,c_0}$, $w_1(t_1)$ and $w_{\infty}$ in \eqref{cont}),
it is simple to verify that $b$ and  $be^{i a }$ depend continuously on $\alpha \in [0,1)$,
provided that $z_\infty$ is a continuous function of $\alpha$. In Subsection~\ref{subsec-dependence} we will prove
that $z_{\infty}$ depends continuously on $\alpha$, for $\alpha \in [0,1]$, and establish the continuous dependence of the
 constants $b$ and $be^{ia}$ with respect to the parameter $\alpha$ in Lemma~\ref{lema-cont-a} above.
\end{remark}

In the proof of Proposition~\ref{prop-asymp}, we have used the
following key lemma that establishes the control of
certain integrals by exploiting their oscillatory character.


\begin{lemma}\label{int-osc} With the same notation as in the proof
of Proposition~\ref{cotas}.
\begin{itemize}
\item[\it{(i)}] Let $f\in C^1((t_1,\infty))$ such that
\bqq \abs{f(t)}\leq L/t^{a} \quad {\hbox{and}} \quad
\abs{f'(t)}\leq
L\left(\frac{\alpha}{t^a}+\frac1{t^{a+1}}\right), \eqq
 for some
constants $L$, $a>0$. Then, for all $t\geq t_1$ and $l\geq 1$
$$
\int_{t}^\infty e^{-\int_{t_1}^\tau\lambda_{+} }
e^{-l\alpha\tau}f(\tau)\,d\tau= \frac{1}{(\alpha +i
\beta)}{e^{-\int_{t_1}^t\lambda_{+} }e^{-l\alpha t}f(t)}+F(t),
$$
with \bq\label{F-1} \abs{F(t)}\leq \frac{ C(l,a,c_0)Le^{-l\alpha
t}}{\beta t^{a}}. \eq
\item[{\it{(ii)}}] If in addition $f\in C^2((t_1,\infty))$,
\bq\label{segunda-der} \abs{f'(t)}\leq L/t^{a+1} \quad
{\hbox{and}}\quad \abs{f''(t)}\leq L\left(
\frac{\alpha}{t^{a+1}}+\frac{1}{t^{a+2}}\right), \eq then
\begin{equation}
\label{F-2}
\abs{F(t)}\leq \frac{ C(l,a,c_0)Le^{-l\alpha t}}{\beta t^{a+1}}.
\end{equation}
\end{itemize}
Here $C(l,a,c_0)$ is a positive constant depending only on $l$, $a$ and $c_0$.
\end{lemma}
\begin{proof}
Define $\lambda=\lambda_{+}$. Recall (see proof of
Proposition~\ref{cotas}) that
$$
 \lambda_{+} =\frac{\alpha K}{2}+i\beta\Delta^{1/2} \quad {\hbox{and}}\quad
 \Delta=1+K-\frac{\alpha^2 K^2}{4\beta^2},\quad {\hbox{with}}\quad K=c_0^2 \frac{e^{-2\alpha t}}{t}.
$$
Setting $R_\lambda=1/\lambda-1/(i\beta)$ and integrating by parts,
we obtain
\begin{align*}
\left(1+\frac{l\alpha}{i\beta}\right)\int_{t}^\infty
e^{-\int_{t_1}^\tau\lambda }  e^{-l\alpha\tau}f(\tau)\,d\tau
={}&{e^{-\int_{t_1}^t\lambda }e^{-l\alpha t}f(t)} \left( \frac{1}{i\beta}+R_\lambda \right)   \\
& +\int_{t}^\infty e^{-\int_{t_1}^\tau\lambda } e^{-l\alpha\tau}\left(-l\alpha f R_\lambda+\frac{f'}{\lambda}-\frac{f\lambda'}{\lambda^2}\right)d\tau,
\end{align*}
or, equivalently,
$$
  \int_{t}^\infty e^{-\int_{t_1}^\tau\lambda }  e^{-\alpha\tau}f(\tau)\,d\tau =
  \frac{1}{l\alpha+i\beta} e^{-\int_{t_1}^t\lambda }  e^{-\alpha t}f(t)
  + F(t),
$$
with
$$
 F(t)= \frac{i\beta}{l\alpha+i\beta}\left(
  e^{-\int_{t_1}^t\lambda }  e^{-l\alpha t}R_\lambda f +
  \int_{t}^\infty e^{-\int_{t_1}^\tau\lambda } e^{-l\alpha
  \tau}\left(-l\alpha f
  R_\lambda+\frac{f'}{\lambda}-\frac{f\lambda'}{\lambda^2}\right)d\tau
  \right).
$$
Using (\ref{cota-Delta}), \eqref{key} and (\ref{cota-K-prima}), it is easy
to check that for all $t\geq t_1$
\begin{equation}
  \label{ast-2}
  \abs{\lambda}\geq \frac{\beta}{\sqrt{2}}
  \quad {\hbox{and}}\quad
  \abs{\lambda '}\leq 3c_0^2\left( \frac{2\alpha}{t}+\frac{1}{t^2} \right).
\end{equation}
On the other hand,
$$
  |R_\lambda|=\left|  \frac{i\beta-\lambda}{i\beta \lambda} \right|\leq
  \frac{\sqrt{2}}{\beta^2} \left(\beta |1-\Delta^{1/2}|+\frac{\alpha K}{2}\right),
$$
with, using the definition of $\Delta$ in (\ref{cota-Delta}) and
(\ref{key}),
$$
   \frac{\alpha K}{2}\leq \frac{c_0^2}{2t}\quad {\hbox{and}}\quad
   \abs{1-\Delta^{1/2}}=\frac{\abs{1-\Delta}}{1+\Delta^{1/2}}\leq \abs{1-\Delta}
   \leq \frac{c_0^2}{t}+ \frac{c_0^2}{4 \beta t}\,\left( \frac{c_0^2}{\beta t}\right)
   \leq \frac{2c_0^2}{\beta t}.
$$
Previous lines show that
\begin{equation}
 \label{ast-4}
  \abs{R_\lambda}\leq \frac{10 c_0^2}{\beta^2 t}.
\end{equation}
The estimate (\ref{F-1}) easily follows from the bounds
(\ref{cota-e}), (\ref{ipp2}), (\ref{ast-2}), (\ref{ast-4}) and the
hypotheses on $f$.
To obtain part {\it{(ii)}} we only need to improve the estimate for
the term
$$
 \int_{t}^\infty e^{-\int_{t_1}^\tau\lambda }  e^{-l\alpha \tau}\frac{f'}{\lambda}\,d\tau
$$
in the above argument. In particular, it suffices to prove that
$$
  \left|\int_{t}^\infty e^{-\int_{t_1}^\tau\lambda }  e^{-l\alpha \tau}\frac{f'}{\lambda}\right|
  \leq C(l,c_0,a ) \frac{L e^{-{l\alpha t}}}{\beta^2 t^{a +1}}.
$$
Now, consider the function $g=f'/\lambda$. Notice that from
(\ref{key}), (\ref{ast-2}) and the hypotheses on $f$ in
(\ref{segunda-der}), we have
$$
  \abs{g(t)}\leq \frac{\sqrt{2} L}{\beta t^{a +1}}
$$
and
\begin{align*}
  \abs{g'(t)}
  \leq&
  \frac{\sqrt{2}}{\beta} L\left( \frac{\alpha}{t^{a+1}}+\frac{1}{t^{a + 2}} \right)
  + \frac{6L}{\beta} \left(\frac{c_0^2}{\beta t}\right) \left(
  \frac{2\alpha}{t^{a +1}} +\frac{1}{t^{a +2}} \right)
     \\
  \leq &\frac{14 L}{\beta} \left( \frac{2\alpha}{t^{a +1}} +\frac{1}{t^{a +2}}  \right).
\end{align*}
Therefore, from part {\it{(i)}}, we obtain
$$
\left|\int_{t}^\infty e^{-\int_{t_1}^\tau\lambda }  e^{-l\alpha \tau}\frac{f'}{\lambda}\right|
  \leq C(l,c_0,a ) L e^{-{l\alpha t}} \left(
  \frac{1}{\beta t^{a +1}} + \frac{1}{\beta^2 t^{a +1}} \right)
  \leq \frac{C(l,c_0,a)Le^{-l\alpha t}}{\beta^2 t^{a+1}},
$$
as desired.
\end{proof}

We remark that if $\alpha\in[0,1/2]$, the asymptotics in Proposition~\ref{prop-asymp} are uniform in $\alpha$.
Indeed,
$$\max_{\alpha\in[0,1/2]}\left\{4\sqrt{8+c_0^2}, 2c_0\left(\frac{1}{\beta}-1\right)^{1/2}\right\}=4\sqrt{8+c_0^2}=s_0.$$
Therefore in this situation we can omit the dependence on $s_1$ in the function $\phi(s_1;s)$, because the asymptotics are valid with
\begin{equation}\label{phi2}
\phi(s):=\phi(s_0;s)=a+\beta \int_{s_0^2/4}^{s^2/4}\sqrt{1+c_0^2\frac{e^{-2\alpha t}}{t}}\,dt.
\end{equation}

We continue to show that the factor $1/\beta^2$ in the big-$O$ in
formulae \eqref{asym-y-0} and \eqref{asym-h-0} are due to the method
used and this factor can be avoided if $\alpha$ is far from zero.
More precisely, we have the following:
\begin{lemma}\label{lema-sin-beta} Let $\alpha\in[1/2,1)$. With the same notation as in
Propositions~\ref{cotas} and~\ref{prop-asymp}, we have the following asymptotics: for all
$s\geq s_0$,
\begin{align}
y(s)&=be^{-\alpha s^2/4}\sin(\phi(s))-\frac{2\alpha \gamma}{s} e^{-\alpha s^2/2}+O\left(\frac{e^{-\alpha s^2/2}}{s^2}\right),\label{asym-y-2}\\
h(s)&=be^{-\alpha s^2/4}\cos(\phi(s))-\frac{2\beta \gamma}{s} e^{-\alpha s^2/2}+O\left(\frac{e^{-\alpha s^2/2}}{s^2}\right).\label{asym-h-2}
\end{align}
Here, the function $\phi$ is defined by \eqref{phi2} and the bounds controlling the error terms depend on $c_0$, and the energy $E_0$, and are independent of $\alpha\in[1/2,1)$
\end{lemma}
\begin{proof}
\noindent Let $\alpha\in [1/2,1)$ and define $w=y+ih$. From
Proposition~\ref{prop-asymp} and \eqref{asymp-phi}, we have that for
all $\alpha \in [1/2,1)$
\bq
\label{w-tilde-a}
\lim_{s\to\infty}we^{(\alpha+i\beta)s^2/4}=b i e^{-i\tilde a},
\eq
where $\tilde a:=a+C(\alpha,c_0)$,  $a$ and $b$ are the constants defined in Proposition~\ref{prop-asymp}
and  $C(\alpha,c_0)$ is the constant in \eqref{asymp-phi}. Then, since $w$ satisfies
\bq
\label{der-w}
\left(
we^{(\alpha+i\beta)s^2/4}\right)'=e^{(-\alpha+i\beta) s^2/4}
\left(\gamma-\frac{c_0^2}{2} (z-z_\infty)\right),
\eq
integrating the above identity between $s$ and infinity,
\bqq
 we^{(\alpha+i\beta)s^2/4}=ib  e^{-i\tilde a}-
\int_s^\infty e^{(-\alpha+i\beta) \sigma^2/4}
\left(\gamma-\frac{c_0^2}{2} (z-z_\infty)\right)d\sigma.
\eqq

Now, integrating by parts and using \eqref{int-conv} (recall that
$1\leq 2\alpha$), we see that
$$
\int_s^\infty e^{(-\alpha+i\beta) \sigma ^2/4}\,d\sigma= 2(\alpha+i\beta) \frac{e^{(-\alpha+i\beta)
s^2/4}}{s}+O\left(\frac{e^{-\alpha s^2/4}}{s^3}\right), \qquad \forall \, s\geq s_0.
$$
Next, notice that from (\ref{cota}) in Proposition~\ref{cotas}, we
also obtain
$$
 \int_s^\infty e^{(-\alpha+i\beta) \sigma^2/4}
(z-z_\infty)\,d\sigma=O\left(\frac{e^{-\alpha s^2/2}}{s^2}\right),
\quad \forall\, s\geq s_0.
$$
The above argument shows that for all $s\geq s_0$ \bq\label{B3}
w(s)=ibe^{-\alpha s^2/4}e^{-i(\tilde a+\beta
s^2/4)}-\frac{2(\alpha+i\beta) \gamma}{s} e^{-\alpha
s^2/2}+O\left(\frac{e^{-\alpha s^2/2}}{s^2}\right). \eq The
asymptotics for $y$ and $h$ in the statement of the lemma easily
follow from (\ref{B3}) bearing in mind that $w=y+ih$ and recalling
that the function
$\phi$ behaves like (\ref{asymp-phi}) when
$\alpha>0$.
\end{proof}
%

In the following corollary we summarize the asymptotics for $z$, $y$
and $h$ obtained in this section. Precisely, as a consequence of
Proposition~\ref{cotas}-{\it{(iii)}}, Proposition~\ref{prop-asymp}
and Lemma~\ref{lema-sin-beta}, we have the following:
\begin{cor}\label{cor-asymp}
Let $\alpha\in[0,1)$. With the same notation as before,  for all
$s\geq s_0=4\sqrt{8+c_0^2}$,
\begin{align}
y(s)&=be^{-\alpha s^2/4}\sin(\phi(s))-\frac{2\alpha \gamma}{s} e^{-\alpha s^2/2}+O\left(\frac{e^{-\alpha s^2/2}}{s^2}\right),\label{asym-y}\\
h(s)&=be^{-\alpha s^2/4}\cos(\phi(s))-\frac{2\beta \gamma}{s} e^{-\alpha s^2/2}+O\left(\frac{e^{-\alpha s^2/2}}{s^2}\right),\label{asym-h}\\
  z(s)&=z_{\infty}-\frac{4b}{s}e^{-\alpha s^2/4}(\alpha \sin(\phi(s))+\beta \cos(\phi(s)))+\frac{4\gamma e^{-\alpha s^2/2} }{s^2}
  +O\left(\frac{e^{-\alpha s^2/4}}{s^3}\right),
  \label{asym-z}
  \end{align}
where
$$
\phi(s)=a+\beta \int_{s_0^2/4}^{s^2/4}\sqrt{1+c_0^2\frac{e^{-2\alpha t}}{t}}\,dt,
$$
for some constant $a\in[0,2\pi)$,
$$
 b=z_\infty^{1/2}\left( 2E_0-\frac{c_0^2}{4}z_\infty \right)^{1/2}, \qquad \gamma = 2E_0-\frac{c_0^2}{2}z_\infty
\qquad {\hbox{and}}\qquad z_\infty=\lim_{s\rightarrow\infty} z(s).
$$
Here, the bounds controlling the error terms depend on $c_0$ and the
energy $E_0$, and are independent of $\alpha\in [0,1)$.
\end{cor}
\begin{remark} In the case when $s<0$, the same arguments
to the ones leading to the asymptotics in the above corollary will
lead to an analogous asymptotic behaviour for the variables $z$, $h$
and $y$ for $s<0$. As mentioned at the beginning of
Subsection~\ref{second}, here we have reduced ourselves to the case
of $s>0$ when establishing the asymptotic behaviour of the latter
quantities due to the parity of the solution we will be applying
these results to.
\end{remark}
\begin{remark}\label{rem-asymp}
The asymptotics in Corollary~\ref{cor-asymp} lead to the asymptotics for the solutions $f$ of the equation \eqref{eq-f}, at least if
$\abs{f}_{\infty}:=z_\infty^{1/2}$ is strictly positive. Indeed, this implies that there exists $s^*\geq s_0$
such that $f(s)\neq 0$ for all $s\geq s^*$. Then
writing $f$ in its polar form $f=\rho\exp(i\theta)$,
we have $\rho^2 \theta'=\Im(\bar f f')$. Hence, using  \eqref{def-z}, we obtain $\rho=z^{1/2}$
and $\theta'=h/z$. Therefore, for all $s\geq s^*$,
\bq\label{phase-f}
\theta(s)-\theta(s^*)=\int_{s^*}^s \frac{h(\sigma)}{z(\sigma)}\,d\sigma.
\eq
Hence, using the asymptotics for $z$ and $h$ in Corollary~\ref{cor-asymp}, we can obtain the asymptotics for $f$.
In the case that $\alpha\in(0,1]$, we can also show that the phase converges. Indeed,
the asymptotics in Corollary~\ref{cor-asymp} yield that the integral in \eqref{phase-f} converges as $s\to \infty$ for $\alpha>0$,
and we conclude that there exists a constant $\theta_\infty\in \R$ such that
$$f(s)=z(s)^{1/2} \exp\left(
i\theta_\infty-i\int_{s}^\infty\frac{h(\sigma)}{z(\sigma)}\,d\sigma
\right),\quad \text{for all }\ s\geq s^*.$$
The asymptotics for $f$ is obtained by plugging the asymptotics in Corollary~\ref{cor-asymp} into the above expression.
\end{remark}
%
\subsection{The second-order equation. Dependence on the parameters}\label{subsec-dependence}
The aim of this subsection is to study the dependence of the $f$, $z$, $y$ and $h$
on the parameters $c_0>0$ and $\alpha\in[0,1]$. This will allow us to pass to the limit $\alpha\to 1^-$ in the asymptotics in Corollary~\ref{cor-asymp}
and will give us the elements for the proofs of Theorems~\ref{thm-c-0} and \ref{thm-alpha1-2}.

\subsubsection{Dependence on $\alpha$}
We will denote by $f(s,\alpha)$ the solution of
\eqref{eq-f} with some initial conditions $f(0,\alpha)$,
$f'(0,\alpha)$ that are independent of $\alpha$.
Indeed, we are interested  in initial conditions that depend only on $c_0$ (see \eqref{ic1}--\eqref{ic3}).
Moreover, in view of \eqref{energy-i}, we assume that the energy $E_0$ in \eqref{energy-1} is a function of $c_0$.
In order to simplify the notation, we denote with a subindex $\alpha$ the
derivative with respect to $\alpha$ and by $'$ the
derivative with respect to $s$. Analogously to Subsection~\ref{second}, we define
\bq\label{def-z-alpha} z(s,\alpha)=\abs{f(s,\alpha)}^2, \quad
y(s,\alpha)=\Re(\bar f(s,\alpha)f'(s,\alpha)),\quad
h(s,\alpha)=\Im(\bar f(s,\alpha)f'(s,\alpha))
\eq
and
\bqq
z_{\infty}(\alpha)=\lim_{s\to\infty}\abs{f(s,\alpha)}^2.
\eqq
Observe that in Proposition~\ref{cotas}-$(ii)$, we proved the existence of $z_\infty(\alpha)$, for $\alpha \in [0,1)$.
For $\alpha\in(0,1]$,  the estimates in \eqref{est-y} hold true and hence $z(s,\alpha)$ is a bounded function whose derivative
decays exponentially. Therefore, it admits a limit at infinity for all $\alpha\in[0,1]$ and  $z_\infty(1)$ is well-defined.

The next lemma provides
estimates for $z_\alpha$, $h_\alpha$ and $y_\alpha$.

\begin{lema}\label{lemma-z-alpha}
Let $\alpha\in(0,1)$. There exists a constant $C(c_0)$, depending on
$c_0$ but not on $\alpha$, such that for all $s\geq 0$,
\begin{align}
\left|z_\alpha(s,\alpha)\right|\leq C(c_0)\min\left\{\frac{s^2}{\sqrt{1-\alpha}}+s^3,  \frac{s^2}{\sqrt{\alpha(1-\alpha)}}, \frac1{\alpha^2\sqrt{1-\alpha}} \right\},\label{z-alpha}\\
\left|y_\alpha(s,\alpha)\right|+\left|h_\alpha(s,\alpha)\right|\leq C(c_0)e^{-\alpha s^2/4}\min\left\{\frac{s^2}{\sqrt{1-\alpha}}+s^3,  \frac{s^2}{\sqrt{\alpha(1-\alpha)}} \right\}\label{h-alpha}.
\end{align}
\end{lema}
\begin{proof}
Differentiating \eqref{eq-f0} with respect to $\alpha$,
\bq\label{f-alpha}
f_\alpha''+\frac{s}2(\alpha+i\beta)f_\alpha'+\frac{c_0^2}{4}f_\alpha
e^{-\alpha s^2/2}=g, \eq where
 \bqq
g(s,\alpha)=-\left(1-i\frac{\alpha}{\beta}\right)\frac{s}{2}f'+\frac{c_0^2s^2}{8}f e^{-\alpha s^2/2}. \eqq
 Also, since the initial conditions do not
depend on $\alpha$, \bq\label{ic-alpha}
f_\alpha(0,\alpha)=f'_\alpha(0,\alpha)=0. \eq Using the estimates in
\eqref{est-f} and that $\alpha^2+\beta^2=1$, we obtain
\bq\label{est-g} \abs{g}\leq C(c_0)\left(\frac{s}{\beta}e^{-\alpha
s^2/4}+s^2e^{-\alpha s^2/2}\right), \quad \text{ for all } s\geq 0.
\eq Multiplying \eqref{f-alpha} by $\bar f_\alpha'$ and taking
real part, we have \bq\label{conser-alpha}
\frac12\left(\abs{f_\alpha'}^2\right)'+\frac{\alpha
s}{2}\abs{f_\alpha'}^2+\frac{c_0^2}{8}\left(\abs{f_\alpha}^2\right)'e^{-\alpha
s^2/2}=\Re(g \bar f_\alpha'). \eq Multiplying
\eqref{conser-alpha} by $2e^{\alpha s^2/2}$ and integrating, taking
into account \eqref{ic-alpha}, \bq\label{conser-alpha2}
\abs{f_\alpha'}^2 e^{\alpha
s^2/2}+\frac{c_0^2}{4}\abs{f_\alpha}^2=2\int_0^s e^{\alpha
\sigma^2/2}\Re(g \bar f_\alpha')\, d\sigma. \eq Let us define
the real-valued function $\eta=\abs{f_\alpha'} e^{\alpha s^2/4}$.
Then \eqref{conser-alpha2} yields \bqq \eta^2(s)\leq 2\int_0^s
e^{\alpha \sigma^2/4}\abs{g}\eta\,d\sigma, \quad \text{ for all
}s\geq 0. \eqq Thus, by the Gronwall inequality (see e.g.
\cite[Lemma A.5]{brezis-mon}),
\bq\label{gronwall} \eta(s)\leq
\int_0^s e^{\alpha \sigma^2/4}\abs{g},d\sigma, \quad \text{ for all
}s\geq 0.
\eq
From \eqref{est-g}, \eqref{conser-alpha2} and
\eqref{gronwall}, we conclude that
\begin{eqnarray*}
(\abs{f_\alpha'} e^{\alpha s^2/4}+\frac{c_0}{2}\abs{f_\alpha})^2
&\leq&
2(|f_\alpha|^2e^{\alpha s^2/2}+\frac{c_0^2}{4}|f_\alpha|^2)
   \\
&\leq&
4\int_{0}^{s} e^{\alpha \sigma^2/4} |g| \eta\, d\sigma
\leq 4\left( \sup_{\sigma\in[0,s]} \eta(\sigma)\right)\left( \int_{0}^{s} e^{\alpha \sigma^2/4} |g|\, d\sigma  \right)
   \\
&\leq& \left( \int_{0}^{s} e^{\alpha \sigma^2/4} |g|\, d\sigma  \right)^2.
\end{eqnarray*}
Thus, using (\ref{est-g}), from the above inequality it follows
\bq\label{proof-f-alpha}
\abs{f_\alpha'} e^{\alpha s^2/4}+\frac{c_0}{2}\abs{f_\alpha}\leq
C(c_0)\int_0^s\left(\frac{\sigma}{\beta} +\sigma^2e^{-\alpha
\sigma^2/4}\right)\,d\sigma, \quad \text{ for all } s\geq 0. \eq
In particular, for all $s\geq 0$, \bq\label{est-f-alpha}
\begin{split}
\abs{f_\alpha(s)}&\leq C(c_0)\min\left\{\frac{s^2}{\sqrt{1-\alpha}}+s^3,  \frac{s^2}{\sqrt{\alpha(1-\alpha)}}\right\},\\
\abs{f_\alpha'(s)}&\leq  C(c_0)e^{-\alpha s^2/4}
\min\left\{\frac{s^2}{\sqrt{1-\alpha}}+s^3,
\frac{s^2}{\sqrt{\alpha(1-\alpha)}}\right\},
\end{split}
\eq
where we have used that
\bqq
\int_0^s\sigma^2e^{-\alpha
\sigma^2/4}\,d\sigma\leq  s^2\int_0^s e^{-\alpha
\sigma^2/4}\,d\sigma\leq s^2\sqrt{\pi/\alpha}.
\eqq
Notice that from \eqref{ic-alpha} and \eqref{est-f-alpha},
\bqq
\abs{f_\alpha(s)}\leq
\int_0^s\abs{f_\alpha'}\,d\sigma\leq
\frac{C(c_0)}{\sqrt{\alpha(1-\alpha)}}\int_0^s \sigma^2e^{-\alpha
\sigma^2/4}\,d\sigma,
\eqq
and
\begin{equation}
\label{integral}
\int_{0}^{\infty}\sigma^2 e^{-\alpha \sigma^2/4}\, d\sigma=\frac{2\sqrt{\pi}}{\alpha^{3/2}},
\end{equation}
so that
\bq\label{est-f-alpha2}
\abs{f_\alpha(s)}\leq \frac{C(c_0)}{\alpha^2\sqrt{1-\alpha}}.
\eq
On the other hand, differentiating the relations in \eqref{def-z-alpha} with
respect to $\alpha$,
\bq\label{proof-z-alpha}
\abs{z_\alpha}\leq
2\abs{f_\alpha}\abs{f}, \quad \abs{y_\alpha+ih_\alpha}\leq
\abs{f_\alpha}\abs{f'}+\abs{f}\abs{f'_\alpha}.
\eq
By putting together \eqref{est-f}, \eqref{est-f-alpha}, \eqref{est-f-alpha2}
and \eqref{proof-z-alpha}, we obtain \eqref{z-alpha} and \eqref{h-alpha}.
\end{proof}
\begin{lema}\label{z-cont}
The function $z_{\infty}$ is continuous in $(0,1]$. More
precisely, there exists a constant $C(c_0)$  depending on $c_0$ but
not on $\alpha$, such that
\bq\label{cota-z-cont}
\abs{z_{\infty}(\alpha_2)-z_\infty(\alpha_1)}\leq
\frac{C(c_0)}{L(\alpha_2,\alpha_1)}\abs{\alpha_2-\alpha_1}, \quad
\text{ for all }\alpha_1,\alpha_2 \in (0,1],
\eq
where
$$
L(\alpha_2,\alpha_1):=\alpha_1^2\alpha_2^{3/2}\left( \alpha_1^{3/2}\sqrt{1-\alpha_2}+\alpha_2^{3/2}\sqrt{1-\alpha_1}\right).
$$
In particular,
\bq\label{cota-z-cont2}
\abs{z_{\infty}(1)-z_\infty(\alpha)}\leq C(c_0)\sqrt{1-\alpha},
\quad \text{ for all }\alpha\in [1/2,1].
\eq
\end{lema}
\begin{proof}
Let $\alpha_1,\alpha_2\in(0,1]$, $\alpha_1<\alpha_2$. By classical
results from the ODE theory, the functions $y(s,\alpha)$,
$h(s,\alpha)$ and $z(s,\alpha)$ are smooth in $\R\times [0,1)$ and
continuous in $\R\times [0,1]$ (see e.g.
\cite{coddington,hartman}). Hence,  integrating \eqref{z} with respect to  $s$, we
deduce that
\bq\label{proof-z-cont}
z_\infty(\alpha_2)-z_\infty(\alpha_1)=2\int_0^\infty
(y(s,\alpha_2)-y(s,\alpha_1))\,ds= 2\int_0^\infty
\int_{\alpha_1}^{\alpha_2} \frac{dy}{d\mu}(s,\mu)\,d\mu\,ds.
\eq
To estimate the last integral, we use \eqref{h-alpha}
\bq\label{proof-z-cont2}
\int_{\alpha_1}^{\alpha_2} \abs{\frac{dy}{d\mu}(s,\mu)}\,d\mu \leq
C(c_0)\frac{s^2}{\sqrt{\alpha_1}}\int_{\alpha_1}^{\alpha_2}\frac{e^{-\mu
s^2/4}}{\sqrt{1-\mu}}\,d\mu. \eq
Now, integrating by parts,
\bqq
\int_{\alpha_1}^{\alpha_2}\frac{e^{-\mu
s^2/4}}{\sqrt{1-\mu}}\,d\mu=2\left(
\sqrt{1-\alpha_1}e^{-\alpha_1 s^2/4}-\sqrt{1-\alpha_2}e^{-\alpha_2
s^2/4}\right)
-\frac{s^2}{2}\int_{\alpha_1}^{\alpha_2}\sqrt{1-\mu}e^{-\mu
s^2/4}\,d\mu.
\eqq
Therefore, by combining with \eqref{proof-z-cont} and \eqref{proof-z-cont2},
\bqq
\abs{z_\infty(\alpha_2)-z_\infty(\alpha_1)}\leq
\frac{C(c_0)}{\sqrt{\alpha_1}}\left(\sqrt{1-\alpha_1}\int_0^\infty
s^2e^{-\alpha_1 s^2/4}\,d s-\sqrt{1-\alpha_2}\int_0^\infty
s^2e^{-\alpha_2 s^2/4}\,d s\right),
\eqq
and bearing in mind (\ref{integral}), we conclude that
\bqq
\abs{z_\infty(\alpha_2)-z_\infty(\alpha_1)}\leq
\frac{C(c_0)}{\sqrt{\alpha_1}}\left(\frac{\sqrt{1-\alpha_1}}{\alpha_1^{3/2}}-\frac{\sqrt{1-\alpha_2}}{\alpha_2^{3/2}}\right),
\eqq
which, after some algebraic manipulations and using that $\alpha_1$, $\alpha_2\in(0,1]$, leads to  \eqref{cota-z-cont}.
\end{proof}

The estimate for $z_{\infty}$ near zero is more involved and it is based in an improvement of the estimate for the derivative
of $z_{\infty}$.

\begin{lemma}\label{z-cont-en-0}
The function $z_{\infty}$ is continuous in $[0,1]$. Moreover,
there exists a cons\-tant \mbox{$C(c_0)>0$}, depending on $c_0$ but
not on $\alpha$ such that for all $\alpha\in (0,1/2]$,
\bq\label{ln-alpha}
\abs{z_{\infty}(\alpha)-z_\infty(0)}\leq C(c_0)
\sqrt{\alpha}\abs{\ln(\alpha)}.
\eq
\end{lemma}
\begin{proof}
As in the proof of Lemma~\ref{z-cont}, we recall that the
functions $y(s,\alpha)$, $h(s,\alpha)$ and $z(s,\alpha)$ are smooth
in any compact subset of $\R\times [0,1)$. From now on we will use
the identity \eqref{z-z-inf} fixing $s=1$. We can verify that
the two integral terms in \eqref{z-z-inf} are continuous functions at
 $\alpha=0$, which proves that $z_{\infty}$ is
continuous in $0$. In view of Lemma~\ref{z-cont}, we conclude that  $z_{\infty}$ is continuous in $[0,1]$.

Now we claim that
\bq\label{claim}
\left|\frac{dz_{\infty}}{d\alpha}(\alpha)\right|\leq C(c_0)
\frac{\abs{\ln(\alpha)}}{\sqrt{\alpha}}, \quad \text{for all }
\alpha\in (0,1/2].
\eq
In fact, once \eqref{claim} is proved, we can compute
\bqq
\abs{z_{\infty}(\alpha)-z_\infty(0)}=\left|\int_0^{\alpha}
\frac{dz_{\infty}}{d\mu}(\mu)d\mu\right|\leq
C(c_0)\int_{0}^{\alpha}
\frac{\abs{\ln(\mu)}}{\sqrt{\mu}}\,d\mu=2C(c_0)\sqrt{\alpha}(\abs{\ln(\alpha)}+2),
\eqq
which implies \eqref{ln-alpha}.

It remains to prove the claim. Differentiating \eqref{z-z-inf}
(recall that $s=1$) with respect to $\alpha$, and using that
$y(1,\cdot)$, $h(1,\cdot )$ and $z(1,\cdot)$ are continuous
differentiable in $[0,1/2]$, we deduce that there exists a constant
$C(c_0)>0$ such that
\bq\label{der-z}
\left|\frac{dz_{\infty}}{d\alpha}(\alpha)\right|\leq
C(c_0)+8\abs{I_1(\alpha)}+2c_0^2\abs{I_2(\alpha)},
\eq
with
\bq\label{I1}
I_1(\alpha)=\int_1^\infty\frac{z}{\sigma^3}+\alpha
\int_1^\infty\frac{z_{\alpha}}{\sigma^3}+6\int_1^\infty\frac{z_{\alpha}}{\sigma^5}
\eq
and
\bq\label{I2}
I_2(\alpha)=-\frac{\alpha}2 \int_1^\infty
e^{-\alpha \sigma^2/2}z \sigma
+\alpha \int_1^\infty e^{-\alpha \sigma^2/2}\frac{z_\alpha}\sigma
+2\int_1^\infty e^{-\alpha \sigma^2/2} \frac{z_{\alpha}}{\sigma^3}.
\eq

By \eqref{est-y} and \eqref{z-alpha}, $z$ is uniformly bounded and
$z_\alpha$ grows at most as a cubic polynomial, so that the first
and the  last integral in the r.h.s.\@ of \eqref{I1} are bounded
independently of $\alpha\in[0,1/2]$. In addition,   \eqref{z-alpha}
also implies that
\bq
\abs{z_\alpha}=\abs{z_\alpha}^{1/2}\abs{z_\alpha}^{1/2}
\leq C(c_0)(s^3)^{1/2}\left( \frac{1}{\alpha^2}\right)^{1/2}=
C(c_0)\frac{s^{3/2}}{\alpha},
\eq
which shows that the remaining integral in (\ref{I1}) is bounded.

Thus, the above argument shows that
\bq\label{bound-I-1}
\abs{I_1(\alpha)}\leq C(c_0)\qquad {\hbox{for all}}\quad \alpha\in[0,1/2].
\eq
The same arguments also yield that the first two integrals in the r.h.s. of \eqref{I2} are
bounded by $C(c_0)\alpha^{-1/2}.$ Using once more that
$\abs{z_\alpha}\leq C(c_0)s^2\alpha^{-1/2}$, we obtain the following bounds for the remaining two integrals in (\ref{I2})
\bqq
\left|\alpha\int_{1}^{s} e^{-\alpha\sigma^2/2}\frac{z_{\alpha}}{\sigma}\right|\leq \frac{C(c_0)}{\sqrt{\alpha}}\int_{1}^{\infty} \alpha \sigma e^{-\alpha\sigma^2/2}\, d\sigma=
\frac{C(c_0)}{\sqrt{\alpha}}e^{-\alpha/2}\leq\frac{C(c_0)}{\sqrt{\alpha}}
\eqq
and
\bqq
 \left| 2\int_{1}^{\infty} e^{-\alpha\sigma^/2} \frac{z_{\alpha}}{\sigma^3}\, d\sigma   \right|\leq
 \frac{C(c_0)}{\sqrt{\alpha}}\int_{1}^{\infty} \frac{e^{-\alpha\sigma^2/2}}{\sigma}\, d\sigma \leq
 C(c_0)\frac{\abs{\ln(\alpha)}}{\sqrt{\alpha}}.
\eqq
In conclusion, we have proved that
$$\abs{I_2(\alpha)}\leq C(c_0)\frac{\abs{\ln(\alpha)}}{\sqrt{\alpha}},$$
which combined with \eqref{der-z} and \eqref{bound-I-1}, completes
the proof of claim.
\end{proof}

We end this section showing that the previous continuity results allow us to ``pass to the limit'' $\alpha\to 1^-$ in Corollary~\ref{cor-asymp}.
Using the notation $b(\alpha)=b$ and $a(\alpha)=a$ for the constants defined for $\alpha\in[0,1)$ in Proposition~\ref{prop-asymp} in Subsection~\ref{second}, we have

\begin{lemma}\label{lema-cont-a}
The value $b(\alpha)$ is a continuous function of $\alpha\in[0,1]$ and the value
$b(\alpha) e^{ia(\alpha)}$ is continuous function of $\alpha\in[0,1)$
that can be continuously extended  to $[0,1]$.
The function $a(\alpha)$ has a (possible discontinuous) extension for  $\alpha\in[0,1]$ such that $a(\alpha)\in[0,2\pi)$.
\end{lemma}
\begin{proof}
By Lemma~\ref{z-cont-en-0}, we have the continuity of $z_{\infty}$ in [0,1]. Therefore, in view of Remark~\ref{rem-continuity}, the function $be^{ia}$
is a continuous function of $\alpha\in[0,1)$ and by \eqref{b} $b$ is actually well-defined and continuous in $\alpha\in[0,1]$.

%

It only remains to prove that the limit
\bq\label{limite}
L:=\lim_{\alpha \to 1^-}b(\alpha) e^{ia(\alpha)}
\eq
exists. If $b(1)=0$, it is immediate that $L=0$
and we can give any arbitrary value in $[0,2\pi)$ to $a(1)$.
Let us suppose that $b(1)>0$. Integrating \eqref{der-w}, we get
\bqq
w(s)e^{(\alpha+i\beta)s^2/4}=w(s_0)e^{(\alpha+i\beta)s_0^2/4}+
\int_{s_0}^s e^{(-\alpha+i\beta) \sigma^2/4}
\left(\gamma-\frac{c_0^2}{2} (z-z_\infty)\right)d\sigma,
\eqq
and this relation is valid for any $\alpha\in(0,1]$.
Let $\alpha\in(0,1)$. In view of \eqref{w-tilde-a}, letting $s\to\infty$, we have
\bq\label{proof-integral}
ibe^{i(a+C(\alpha,c_0))}=w(s_0)e^{(\alpha+i\beta)s_0^2/4}+
\int_{s_0}^\infty e^{(-\alpha+i\beta) \sigma^2/4}
\left(\gamma-\frac{c_0^2}{2} (z-z_\infty)\right)d\sigma,
\eq
where $C(\alpha,c_0)$ is the constant in \eqref{asymp-phi}.
Notice that the r.h.s. of \eqref{proof-integral} is well-defined for any $\alpha\in(0,1]$
and by the arguments given in the proof of Lemma~\ref{z-cont} and the
dominated convergence theorem, the r.h.s. is also continuous for any $\alpha\in(0,1]$.
Therefore, the limit $L$ in \eqref{limite} exists and  is given by the r.h.s. of \eqref{proof-integral} evaluated in $\alpha=1$
and divided by $ie^{iC(1,c_0)}$. Moreover,
$$\lim_{\alpha \to 1^-}e^{ia(\alpha)}=\frac{L}{b(1)},$$
so that by the compactness of the the unit circle in $\C$, there exists $\theta\in [0,2\pi)$ such that $e^{i\theta}=L/b(1)$
and we can extend $a$ by defining $a(1)=\theta$.
\end{proof}

The following result summarizes an improvement of Corollary~\ref{cor-asymp}
to include the case $\alpha=1$ and the continuous dependence of the constants appearing in the asymptotics on $\alpha$. Precisely, we have the following:
\begin{cor}\label{cor-asymp-bis}
 Let $\alpha\in[0,1]$,  $\beta\geq 0$ with $\alpha^2+\beta^2=1$ and  $c_0>0$. Then,
 \begin{itemize}
 \item[{\it{(i)}}] The asymptotics in Corollary~\ref{cor-asymp} holds true for all $\alpha\in[0,1]$.
 \item[{\it{(ii)}}] Moreover, the values $b$ and $be^{ia}$  are continuous functions of $\alpha\in[0,1]$
 and each term in the asymptotics for $z$, $y$ and $h$ in Corollary~\ref{cor-asymp} depends continuously on $\alpha\in[0,1]$.
 \item[{\it{(iii)}}] In addition, the bounds controlling the error terms depend on $c_0$ and are independent of $\alpha\in [0,1]$.
 \end{itemize}
\end{cor}
\begin{proof}
Let $s\geq s_0$ fixed. As noticed in the proof of Lemma~\ref{z-cont}, the functions $y(s,\alpha)$, $h(s,\alpha)$, $z(s,\alpha)$ are continuous in $\alpha=1$.
In addition, by Lemma~\ref{lema-cont-a} $be^{ia}$ is continuous in $\alpha=1$, using the definition of $\phi$, it is immediate that $b\sin(\phi(s))$  and
$b\cos(\phi(s))$ are continuous in $\alpha=1$. Therefore the big-$O$ terms in \eqref{asym-y}, \eqref{asym-h} and \eqref{asym-z} are also
 are continuous in $\alpha=1$. The proof of the corollary follows by letting $\alpha\rightarrow 1^{-}$  in \eqref{asym-y}, \eqref{asym-h}
 and \eqref{asym-z}.
\end{proof}
\subsubsection{Dependence on $c_0$}
In this subsection, we study the dependence of $z_\infty$ as a function of $c_0$, for a fixed value of $\alpha$.
To this aim, we need to take into account the initial conditions given in \eqref{ic1}--\eqref{ic3}. More generally,
let us assume that $f$ is a solution of \eqref{eq-f} with initial conditions $f(0)$ and $f'(0)$ that depend smoothly on $c_0$,
for any $c_0>0$, and that $E_0>0$ is the associated energy defined in \eqref{energy-1}. To keep our notation simple, we omit the
parameter $c_0$ in the functions $f$ and $z_\infty$. Under these assumptions, we have

\begin{prop}\label{prop-c-0}
Let $\alpha\in[0,1]$ and $c_0>0$. Then $z_\infty$ is a continuous function of $c_0\in (0,\infty)$. Moreover if $\alpha\in(0,1]$, the following estimate hold
\bq\label{asymp-z-c}
\left|z_\infty-\left|f(0)+\frac{f'(0)\sqrt{\pi}}{\sqrt{\alpha+i\beta}}\right|^2\right|\leq
\frac{\sqrt{2 E_0}
c_0\pi}{\alpha}\left|f(0)+\frac{f'(0)\sqrt{\pi}}{\sqrt{\alpha+i\beta}}\right|
+\left(\frac{\sqrt{2 E_0} c_0\pi}{2\alpha}\right)^2.
\eq
\end{prop}

\begin{proof}
Since we are assuming that the initial conditions $f(0)$ and $f'(0)$ depend smoothly on $c_0$,
by classical results from the ODE theory, the functions $f$, $y$, $h$ and $z$ are smooth with respect to $s$ and $c_0$.
From \eqref{z-z-inf} with $s=1$, we have that $z_\infty$ can be written in terms of continuous functions of $c_0$
(the continuity of the integral terms follows from the dominated convergence theorem), so that $z_\infty$ depends continuously
on $c_0$.

To prove \eqref{asymp-z-c}, we multiply \eqref{eq-f} by $e^{(\alpha+i\beta)s^2/4}$, so that
\bqq
(f'e^{(\alpha+i\beta)s^2/4})'=-\frac{c_0^2}{4}f(s) e^{(-\alpha+i\beta) s^2/4}.
\eqq
Hence,  integrating twice, we have
\bq
\label{int-f}
f(s)=f(0)+G(s)+F(s),
\eq
with
\bqq
 G(s)=f'(0)\int_0^se^{-(\alpha+i\beta)\sigma^2/4}\,d\sigma \ \text{ and } \ F(s)=-\frac{c_0^2}4\int_0^se^{-(\alpha+i\beta)\sigma^2/4}\int_0^\sigma e^{(-\alpha+i\beta)\tau^2/4}f(\tau)\,d\tau\, d\sigma.
 \eqq
Since by Proposition~\ref{cotas} $\abs{f(s)}\leq
\frac{2\sqrt{2E_0}}{c_0}$, we obtain
\bq\label{cota-F}
\abs{F(s)}\leq \frac{\sqrt{2E_0}c_0}{2} \int_0^se^{-\alpha\sigma^2/4}\int_0^\sigma e^{-\alpha\tau^2/4}\,d\tau\, d\sigma\leq \frac{\sqrt{2E_0}c_0}{2}\cdot \frac{\pi}{\alpha}.
\eq
Using \eqref{int-f} and the identity,
 \bqq
\abs{z_1+z_2}^2=\abs{z_1}^2+2\Re(\bar z_1 z_2)+\abs{z_2}^2, \quad z_1,z_2\in\C,
 \eqq
we conclude that $z(s)=\abs{f(s)}^2$ satisfies
\bqq
z(s)=\abs{f(0)+G(s)}^2+2\Re(\bar F(s)(f(0)+G(s)))+\abs{F(s)}^2.
\eqq
Therefore, for all $s\geq 0$,
\bqq
\abs{z(s)-\abs{f(0)+G(s)}^2}\leq 2\abs{F(s)}\abs{f(0)+G(s)}+\abs{F(s)}^2.
\eqq
Hence we can use the bound \eqref{cota-F} and then let $s\to\infty$.
Noticing that
\bqq
\lim_{s\to \infty}G(s)=f'(0)\int_0^\infty
e^{-(\alpha+i\beta)\sigma^2/4}\,d\sigma=f'(0)\frac{\sqrt{\pi}}{\sqrt{\alpha+i\beta}},
\eqq
the estimate \eqref{asymp-z-c} follows.
\end{proof}

\section{Proof of the main results}\label{sec-proof-results}
In Section~\ref{sec-reduction} we have performed a careful analysis of the equation \eqref{eq-f0}, taking also into consideration the initial
conditions \eqref{ic1}--\eqref{ic3}. Therefore, the proofs of our main theorem consist mainly in coming back to the original variables
using the identities \eqref{m-1} and \eqref{m-2}. For the sake of completeness, we provide the details in the following proofs.
\begin{proof}[{\textbf{Proof of Theorem~\ref{thm-conver}}}]

Let $\alpha\in[0,1]$, $c_0>0$ and $\{\m_{c_0,\alpha}(\cdot), \n_{c_0,\alpha}(\cdot), \b_{c_0,\alpha}(\cdot)\}$ be
the unique $\mathcal{C}^{\infty}(\R; \mathbb{S}^2)$-solution of the Serret--Frenet equations \eqref{serret} with curvature and torsion \eqref{c-tau-1} and initial conditions \eqref{IC-1}. In order to simplify the notation, in the rest of the proof we drop the subindexes $c_0$ and $\alpha$ and simply write $\{\m(\cdot), \n(\cdot), \b(\cdot)\}$ for $\{\m_{c_0,\alpha}(\cdot), \n_{c_0,\alpha}(\cdot), \b_{c_0,\alpha}(\cdot)\}$.

First observe that if we define
$\{\vec{M}, \vec{N}, \vec{B}\}$ in terms of $\{ \m, \n, \b\}$ by
\begin{align*}
 \vec{M}(s)&=(m(-s), -m(-s),- m(-s)),    \\
 \vec{N}(s)&=(-n(-s), n(-s), n(-s)),    \\
 \vec{B}(s)&=(-b(-s), b(-s), b(-s)),\qquad s\in \R,
\end{align*}
then $\{\vec{M}, \vec{N}, \vec{B}\}$ is also a
solution of the Serret system (\ref{serret}) with curvature and
torsion (\ref{c-tau-1}). Notice also that
$$
 \{\vec{M}(0), \vec{N}(0), \vec{B}(0)\}=
 \{\vec{m}(0), \vec{n}(0), \vec{b}(0)\}.
$$
Therefore, from the uniqueness of the solution we conclude that
$$
 \vec{M}(s)= \m(s),\qquad  \vec{N}(s)=\n(s) \qquad {\hbox{and}}\qquad  \vec{B}(s)=\b(s), \qquad \forall \, s\in\R.
$$
This proves part \ref{symmetries} of Theorem~\ref{thm-conver}.

Second, in Section~\ref{sec-reduction} we have seen that one can write the components of the Frenet trihedron $\{\m, \n, \b\}$ as
\begin{align}
m_1(s)&=2\abs{f_1(s)}^2-1, \qquad
n_1(s)+ib_1(s)=\frac{4}{c_0}e^{\alpha s^2/4}\bar f_1(s) f'_1(s),
 \label{mnb-1}\\
m_j(s)&=\abs{f_j(s)}^2-1,  \qquad
n_j(s)+ib_j(s)=\frac{2}{c_0}e^{\alpha s^2/4}\bar f_j(s) f'_j(s), \quad j\in\{2,3\},
 \label{mnb-2}
\end{align}
with $f_j$ solution of the second order ODE \eqref{eq-f0} with initial conditions \eqref{ic1}-\eqref{ic3} respectively, and associated initial energies (see \eqref{energy-i})
\begin{equation}\label{energies}
 E_{0,1}=\frac{c_0^2}{8} \qquad {\hbox{and}}\qquad
 E_{j,1}=\frac{c_0^2}{8}, \qquad {\hbox{for}} \qquad j\in\{2,3\}.
\end{equation}
Notice that the identities \eqref{mnb-1}--\eqref{mnb-2} rewrite equivalently as

\begin{equation}\label{m-j}
\left\lbrace
\begin{aligned}
&m_{1,c_0,\alpha}=2z_1-1, &&n_{1,c_0,\alpha}=\frac{4}{c_0}e^{\alpha s^2/4}\, y_1,
\quad b_{1,c_0,\alpha}=\frac{4}{c_0}e^{\alpha s^2/4}\,h_1,
   \\
&m_{j,c_0,\alpha}=z_j-1, &&n_{j,c_0,\alpha}=\frac{2}{c_0}e^{\alpha s^2/4}\,y_j,
\quad b_{j,c_0,\alpha}=\frac{2}{c_0}e^{\alpha s^2/4}\,h_j, \quad j\in\{2,3\},
\end{aligned}\right.
\end{equation}
in terms of the quantities $\{z_j, y_j, h_j\}$ defined by
$$
z_j=|f_j|^2, \quad y_j=\Re(\bar f_j f_j') \quad {\hbox{and}}\quad h_j=\Im(\bar f_j f_j').
$$
Denote by  $z_{j,\infty}$, $a_j$, $b_j$, $\gamma_j$ and $\phi_j$ the constants and function appearing in the asymptotics of $\{y_j,h_j,z_j\}$ proved in Section~\ref{sec-reduction} in Corollary~\ref{cor-asymp-bis}.

Taking the limit as $s\rightarrow +\infty$ in \eqref{mnb-1}--\eqref{mnb-2}, and since $\abs{\m(s)}=1$, we obtain that there exists
 $\vec A^{+}=(A^{+}_{j})_{j=1}^{3}\in \mathbb{S}^2$ with
\begin{equation}
 \label{A+}
 A_{1}^{+}=2z_{1,\infty}-1, \qquad
 A_{j}^{+}=z_{j,\infty}-1, \qquad {\hbox{for }} j\in\{2,3\}.
\end{equation}
The asymptotics stated in part \ref{asymp} of Theorem~\ref{thm-conver} easily follows from formulae \eqref{mnb-1}--\eqref{mnb-2}
and the asymptotics for $\{z_j, y_j, h_j  \}$ established in Corollary~\ref{cor-asymp-bis}.
%
%
Indeed, it suffices to observe that from the formulae for $b_j$ and $\gamma_j$ in terms of the initial energies $E_{0,j}$ and $z_{j,\infty}$ given in Corollary~\ref{cor-asymp-bis}, \eqref{energies} and \eqref{A+} we obtain
\begin{gather}
b_1^2=\frac{c_0^2}{16}(1-(A_1^+)^2), \quad  b_2^2=\frac{c_0^2}{4}(1-(A_2^+)^2), \quad  b_3^2=\frac{c_0^2}{4}(1-(A_3^+)^2),\label{b-i}\\
\gamma_1=-\frac{c_0^2}{4}A_1^+, \quad  \gamma_2=-\frac{c_0^2}{2}A_2^+,
\quad  \gamma_3=-\frac{c_0^2}{2}A_3^+.\label{gamma-i}
\end{gather}
Substituting these constants in \eqref{asym-y}, \eqref{asym-h} and
\eqref{asym-z} in Corollary~\ref{cor-asymp-bis}, we obtain \eqref{asym-m}, \eqref{asym-n} and
\eqref{asym-b}. This completes the proof of Theorem~\ref{thm-conver}-\ref{asymp}.
\end{proof}
\begin{proof}[{\textbf{Proof of Theorem~\ref{Theorem1}}}]

Let $\alpha\in[0,1]$, and $c_0>0$. As before, dropping the subindexes, we will denote by $\{\m, \n,\b\}$ the unique solution of the Serret--Frenet equations \eqref{serret} with curvature and torsion \eqref{c-tau-1} and initial conditions \eqref{IC-1}. Define
\begin{equation}
 \label{1-1}
  \mm(s,t)=\m\left(\frac{s}{\sqrt{t}}   \right).
\end{equation}
As has been already mentioned (see Section~\ref{sec-self-similar}), part \ref{regular} of Theorem~\ref{Theorem1} follows from the fact that the
 triplet $\{\m,\n,\b\}$ is a regular-$(\mathcal{C}^{\infty}(\R; \mathbb{S}^2))^3$ solution of \eqref{serret}-\eqref{c-tau-1}-\eqref{IC-1} and satisfies the equation
$$
-\frac{s}{2} c\n= \beta(c'\b-c\tau\n) +\alpha(c\tau\b+c'\n).
$$
Next, from the parity of the components of the profile $\m(\cdot)$ and the asymptotics established in parts \ref{symmetries} and \ref{asymp}
 in Theorem~\ref{thm-conver}, it is immediate to prove
  the pointwise convergence \eqref{convergencia}. In addition, $\vec{A}^{-}=(A^+_1, -A^+_2, -A^+_3)$ in terms of the components of the vector $\vec{A}^{+}=(A^+_j)_{j=1}^{3}$.

Now, using the symmetries of $\m(\cdot)$, the change of variables $\eta=s/\sqrt{t}$ gives us
\bq\label{dem-est-int}
\|\mm(\cdot,t)- \vec A^+\chi_{(0,\infty)}(\cdot)-\vec A^{-}\chi_{(-\infty,0)}(\cdot)\|_{L^p(\R)}=
\sum_{j=1}^3\left(2t^{1/2} \int_0^\infty \abs{m_j(\eta)- A_j^+}^p\,d\eta\right)^{1/p}.
\eq
Therefore, it only remains to prove that the last integral is finite. To this end, let $s_0=4\sqrt{8+c_0^2}$.
On the one hand, notice that since $\m$ and $\vec A^+$ are unitary vectors,
\bq\label{dem-est-int2}
 \int_0^{s_0} \abs{m_j(s)- A_j}^p\,ds\leq 2^ps_0.
\eq
On the other hand, from the asymptotics for $\m(\cdot)$ in  \eqref{asym-m}, \eqref{cota-rmk}, and the fact that the
 vectors $\vec A^+$ and $\vec B^+$ satisfy $\abs{\vec{A}^{+}}^2=1$ and $\abs{\vec{B^+}}^2=2$, we obtain
\begin{align}\nonumber
\left(\int_{s_0}^\infty \abs{m_j(s)-A_j^+}^p\,ds\right)^{1/p}\leq &
 2\sqrt{2}c_0(\alpha+\beta)\left(\int_{s_0}^{\infty}\frac{e^{-\alpha s^2p/4}}{s^p}\right)^{1/p}+ 2c_0^2
\left(\int_{s_0}^{\infty}\frac{e^{-\alpha s^2p/2}}{s^{2p}}\right)^{1/p}\\
& +
C(c_0)\left(\int_{s_0}^{\infty}\frac{e^{-\alpha s^2p/4}}{s^{3p}}\right)^{1/p}.
\label{est-int}
\end{align}
 Since the r.h.s.\@ of \eqref{est-int} is finite for all $p\in (1,\infty)$ if $\alpha \in [0,1]$, and for all $p\in [1,\infty)$ if $\alpha\in(0,1]$,  inequality \eqref{cota-m} follows
from \eqref{dem-est-int}, \eqref{dem-est-int2} and \eqref{est-int}. This completes the proof of Theorem~\ref{Theorem1}.
\end{proof}
\begin{proof}[\textbf{Proof of Theorem~\ref{thm-c-0}}]
The proof is a consequence of Proposition~\ref{prop-c-0}. In fact, recall
the relations \eqref{A+} and \eqref{energy-i}, that is
$$
 A_{1}^{+}=2z_{1,\infty}-1, \qquad {\hbox{and}}\qquad
 A_{j}^{+}=z_{j,\infty}-1, \qquad {\hbox{for }} j\in\{2,3\},
$$
and
$$
   E_{0,1}=\frac{c_0^2}{8}, \qquad E_{0,j}=\frac{c_0^2}{4} , \qquad {\hbox{for }} j\in\{2,3\},
$$
Thus the continuity of $\vec A^+_{c_0,\alpha}$ with respect to $c_0$, follows from the continuity of $z_\infty$ in Proposition~\ref{prop-c-0}.

Using the initial conditions \eqref{ic1}--\eqref{ic3}, the values for the energies $E_{0,j}$ for $j\in\{1,2,3\}$,  and the identity
\bqq
\frac{\sqrt{\pi}}{\sqrt{\alpha+i\beta}}=\frac{\sqrt{\pi}}{\sqrt2}\left(\sqrt{1+\alpha}-i\sqrt{1-\alpha}\right),
\eqq
we now compute
\begin{equation}
\label{values}
\left|f_j(0)+\frac{f'_j(0)\sqrt{\pi}}{\sqrt{\alpha+i\beta}}\right|^2=\begin{cases}
1,  & \text{ if }j=1,\\
1+\frac{c_0^2\pi}{4}+\frac{c_0\sqrt{\pi}}{\sqrt2}\sqrt{1+\alpha}, & \text{ if }j=2,\\
1+\frac{c_0^2\pi}{4}+\frac{c_0\sqrt{\pi}}{\sqrt2}\sqrt{1-\alpha}, & \text{ if }j=3.\\
\end{cases}
\end{equation}
Then, substituting the values \eqref{values} in \eqref{asymp-z-c} and using the above relations together with the
inequality $\sqrt{1+x}\leq 1+x/2$ for $x\geq 0$, we obtain the estimates \eqref{A1}--\eqref{A3}.
\end{proof}
\begin{proof}[\textbf{Proof of Theorem~\ref{thm-alpha1-2}}] Recall that the components of
$\vec A^+_{c_0,\alpha}$ are given explicitly in \eqref{A+} in terms of the functions $z_{j,\infty}$, for $j\in\{1,2,3\}$.
The continuity on $[0,1]$ of $A^{+}_{j,c_0, \alpha}$ as a function of $\alpha$ for $j\in\{ 1,2,3 \}$ follows from that of
 $z_{j,\infty}$ established in Lemma~\ref{z-cont-en-0}. Notice also that the estimates
\eqref{thm-alpha} and \eqref{thm-alpha2}  are an immediate consequence of  \eqref{ln-alpha} in Lemma~\ref{z-cont-en-0} and \eqref{cota-z-cont2} in Lemma~\ref{z-cont}, respectively.
\end{proof}
Before giving the proof of Proposition~\ref{jump}, we recall that  when $\alpha=0$ or $\alpha=1$, the vector
$\vec{A}^+_{c_0,\alpha}=(A_{j,c_0,\alpha})_{j=1}^{3}$ is determined
explicitly in terms of the parameter $c_0$ (see
\cite{vega-gutierrez} for the case $\alpha=0$ and Appendix for
the case $\alpha=1$). Precisely,
\begin{align}
A_{1,c_0,0}&=e^{-\frac{\pi c_0^2}2},\label{A1-alpha-0}\\
A_{2,c_0,0}&= 1-\frac{e^{-\frac{\pi c_0^2}{4}}}{8\pi}\sinh(\pi c_0^2/2)\abs{c_0\Gamma(ic_0^2/4)
+2e^{i\pi/4}\Gamma(1/2+ic_0^2/4)}^2,\label{A2-alpha-0}\\
A_{3,c_0,0}&= 1-\frac{e^{-\frac{\pi c_0^2}{4}}}{8\pi}\sinh(\pi c_0^2/2)\abs{c_0\Gamma(ic_0^2/4)
              -2e^{-i\pi/4}\Gamma(1/2+ic_0^2/4)}^2 \label{A3-alpha-0}
\end{align}
and
\begin{equation}\label{A-alpha-1}
 \vec{A}^+_{c_0,1}=(\cos(c_0\sqrt\pi), \sin(c_0\sqrt\pi),0 ).
\end{equation}
%
\begin{proof}[\textbf{Proof of Proposition~\ref{jump}}] Recall that  (see~Theorem~\ref{Theorem1})
\begin{equation}\label{A-bis}
 \vec{A}^{-}_{c_0,\alpha}=(A^{+}_{1,c_0,\alpha}, -A^{+}_{2,c_0,\alpha}, -A^{+}_{3,c_0,\alpha}),
\end{equation}
with $A^{+}_{j,c_0,\alpha}$ the components of $\vec{A}^{+}_{c_0,\alpha}$. Therefore $\vec{A}^{+}_{c_0,\alpha}\neq \vec{A}^{-}_{c_0,\alpha}$ iff $A^{+}_{1,c_0,\alpha}\neq 1$ or $-1$.

Parts {\it{(ii)}} and {\it{(iii)}} follow from the continuity of ${A}^{+}_{1, c_0,\alpha}$ in $[0,1]$ established
  in Theorem~\ref{thm-alpha1-2} bearing in mind that, from the expressions for ${A}^{+}_{1, c_0,0}$ in
\eqref{A1-alpha-0} and  ${A}^{+}_{1, c_0,1}$   in \eqref{A-alpha-1}, we have that ${A}^{+}_{1, c_0,0}\neq \pm 1$ for all $c_0>0$ and
 ${A}^{+}_{1, c_0,1}\neq \pm 1$ if $c_0\neq k\sqrt{\pi}$ with $k\in\mathbb{N}$.

In order to proof part {\it{(i)}}, we will argue by contradiction. Assume that for some  $\alpha\in (0,1)$,
 there exists a sequence $\{c_{0,n}\}_{n\in\N}$ such that $c_{0,n}>0$, $c_{0,n}\longrightarrow 0$ as $n\rightarrow \infty$ and $\vec{A}^+_{c_{0,n},\alpha}=\vec{A}^-_{c_{0,n}\alpha}$. Hence from \eqref{A-bis} the second and third component of $\vec{A}^+_{c_{0,n},\alpha}$ are zero. Thus the estimate \eqref{A2} in Theorem~\ref{thm-c-0}
yields
$$
c_{0,n}\frac{\sqrt{\pi(1+\alpha)}}{\sqrt2}
\leq \frac{c_{0,n}^2\pi}{4}
  +\frac{c_{0,n}^2\pi}{\alpha\sqrt{2}}\left(1+\frac{c_{0,n}^2 \pi}{8}+c_{0,n}\frac{\sqrt{\pi(1+\alpha)}}{2\sqrt2}\right)
  +\left(\frac{c_{0,n}^2\pi}{2\sqrt{2}\alpha}\right)^2.
$$
Dividing by $c_{0,n}>0$ and letting $c_{0,n}\to 0$ as $n\rightarrow\infty$, the contradiction follows.

\end{proof}
%
%
%
\section{Some numerical results}\label{section-numerics}
%

As has been already pointed out, only in the cases $\alpha=0$ and $\alpha=1$ we have an explicit formula for $\vec A_{c_0,\alpha}^+$ (see \eqref{A1-alpha-0}--\eqref{A-alpha-1}).
Theorems~\ref{thm-c-0} and~\ref{thm-alpha1-2} give information about the behaviour of $\vec A_{c_0,\alpha}^{+}$ for small values of $c_0$ for a fixed valued of $\alpha$, and  for values of $\alpha$ near to 0 or 1 for a fixed valued of $c_0$. The aim of this section is to give some numerical results
that allow us to understand the map $(\alpha,c_0) \in[0,1]\times (0,\infty)\mapsto \vec A_{c_0,\alpha}^{\pm} \in \S^2$.
For a fixed value of $\alpha$, we will discuss first the injectivity and surjectivity (in some appropriate sense) of the map $c_0\mapsto  \vec{A}_{c_0,\alpha}^{\pm}$
and second the behaviour of $ \vec{A}_{c_0,\alpha}^+$ as $c_0\to\infty$.


For fixed $\alpha$, define $\theta_{c_0,\alpha}$ to be the angle between the unit vectors $\vec{A}^{+}_{c_0,\alpha}$ and $-\vec{A}^{-}_{c_0,\alpha}$ associated to the family of solutions  $\mm_{c_0,\alpha}(s,t)$ established in Theorem~\ref{Theorem1}, that is $\theta_{c_0,\alpha}$ such that
\begin{equation}\label{theta}
 \cos(\theta_{c_0,\alpha})=1-2(A^+_{1,c_0,\alpha})^2.
\end{equation}
It is pertinent to ask whether $\theta_{c_0,\alpha}$ may attain any value in the interval $[0,\pi]$ by varying the parameter $c_0>0$.

In Figure~\ref{fig-angulo} we plot the function $\theta_{c_0,\alpha}$ associated to the family of solutions $\mm_{c_0,\alpha}(s,t)$ established in Theorem~\ref{Theorem1} for $\alpha=0$, $\alpha=0.4$ and $\alpha=1$, as a function of $c_0>0$. The curves $\theta_{c_0,0}$ and $\theta_{c_0,1}$ are exact since we have
explicit formulae for $A_{1,c_0,\alpha}^{+}$ when $\alpha=0$ and $\alpha=1$ (see \eqref{A1-alpha-0} and \eqref{A-alpha-1}).
 We deduce that in the case $\alpha=0$, there is a bijective relation between $c_0>0$ and the angles in $(0,\pi)$.
In the case $\alpha=1$, there are infinite values of $c_0>0$ that allow to reach any angle in $[0,\pi]$. If $\alpha\in(0,1)$, numerical simulations show that there exists $\theta^*_{\alpha}\in (0,\pi)$ such that the angles in $(\theta^{*}_{\alpha},\pi)$ are reached by a unique value of $c_0$, but for angles in $[0,\theta^{*}_{\alpha}]$
there are at least two values of $c_0>0$ that produce them (See $\theta_{c_0,0.4}$ in Figure~\ref{fig-angulo}).

\begin{figure}[ht]
\begin{center}
{
\begin{overpic}[trim=0 0 130 0,clip,scale=0.7%
]{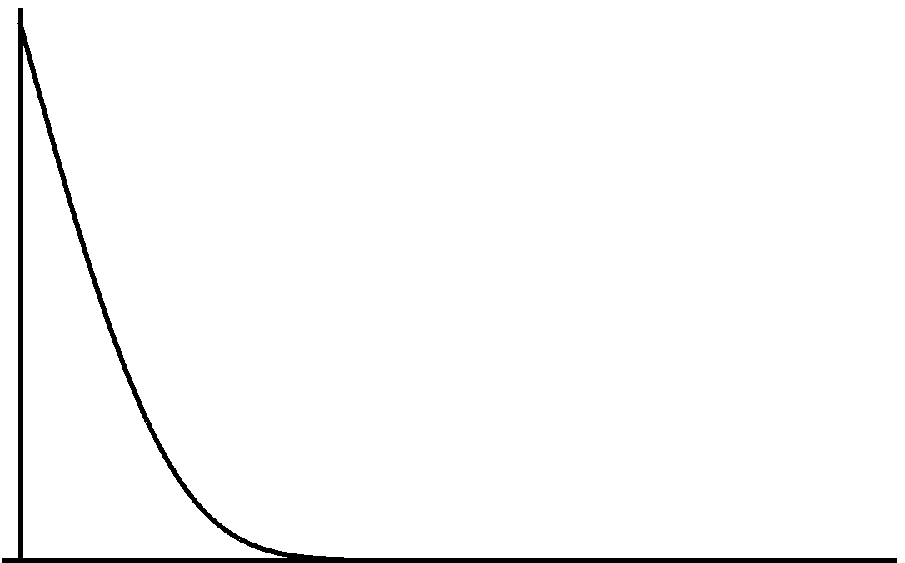}
\put(6,102){\small{$\theta_{c_0,0}$}}
\put(-3,92){\small{$\pi$}}
\put(35,-6){\small{$c_0$}}
\end{overpic}
}
\hspace{-0.3cm}
{
\begin{overpic}[trim=0 0 20 0,clip,scale=0.7]{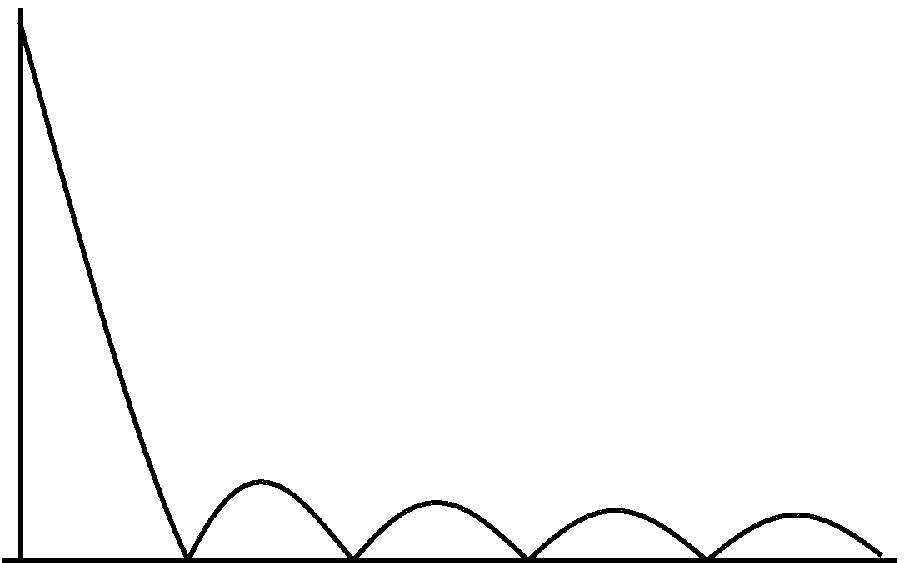}
\put(4,68){\small{$\theta_{c_0,0.4}$}}
\put(-2,62){\small{$\pi$}}
\put(45,-5){\small{$c_0$}}
\end{overpic}
}\hspace*{-0.1cm}
{
\begin{overpic}[trim=0 0 30 0,clip,scale=0.7]{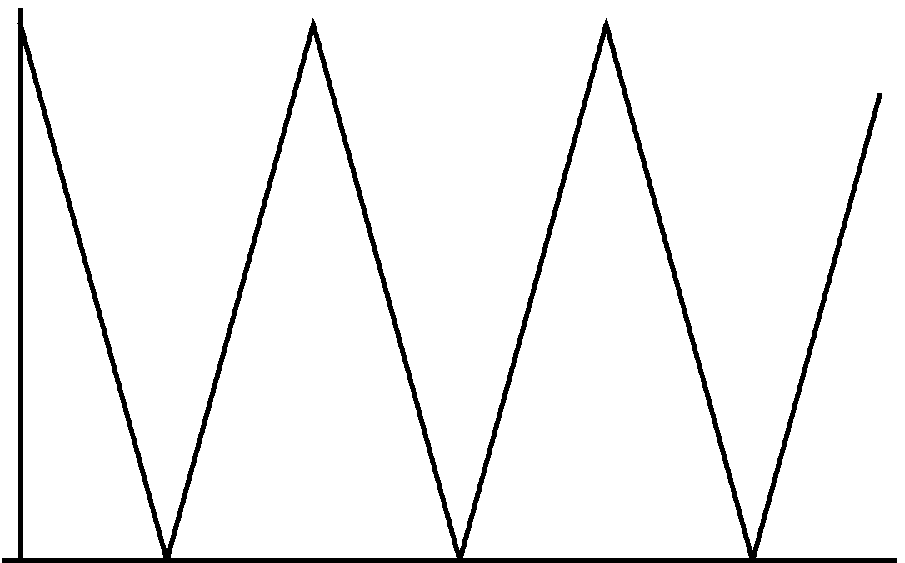}
\put(4,71){\small{$\theta_{c_0,1}$}}
\put(-2,64){\small{$\pi$}}
\put(45,-5){\small{$c_0$}}
\end{overpic}
}
\end{center}
\caption{The angles $\theta_{c_0,\alpha}$ as a function of $c_0$ for $\alpha=0$, $\alpha=0.4$ and $\alpha=1$.}
\label{fig-angulo}
\end{figure}

These numerical results suggest that, due to the invariance of \eqref{LL} under rotations\footnote{In fact, using that
$$
 (M\vec a)\times (M\vec b)=(\det M)M^{-T}(\vec a\times \vec b), \quad \text{ for all }M\in \mathcal{M}_{3,3}(\R), \ \vec a,\vec b\in \R^3,
$$
it is easy to verify that if $\mm(s,t)$ is a solution of \eqref{LL}  with initial condition $\mm^0$, then $\mm_R:=R\mm$ is a solution of \eqref{LL}
with initial condition $\mm^0_R:=R\mm^0$, for any $R\in SO(3)$.
}, for a fixed  $\alpha\in [0,1)$ one can solve the following inverse problem: Given any distinct vectors $\vec A^+,\vec A^-\in \S^2$ there exists $c_0>0$ such that
the associated solution $\mm_{c_0,\alpha}(s,t)$ given by Theorem~\ref{Theorem1} (possibly multiplied by a rotation matrix) provides a solution of \eqref{LL} with initial condition
\bq\label{initial-condition}
 \mm(\cdot,0)=\vec{A}^+\chi_{(0,\infty)}(\cdot)+\vec{A}^{-}\chi_{(-\infty,0)}(\cdot).
\eq
Note that in the case $\alpha=1$ the restriction  $\vec A^+\neq \vec A^-$ can be dropped.

In addition, Figure~\ref{fig-angulo} suggests that  $\vec{A}_{c_0,\alpha}^{+}\neq \vec{A}_{c_0,\alpha}^{-}$  for fixed $\alpha\in[0,1)$ and $c_0>0$. Indeed, notice that $\vec{A}_{c_0,\alpha}^{+}\neq \vec{A}_{c_0,\alpha}^{-}$ if and only if $A_1\neq \pm 1$ or equivalently $\cos \theta_{c_0,\alpha}\neq -1$, that is $\theta_{c_0,\alpha}\neq \pi$, which is true if $\alpha\in [0,1)$ for any $c_0>0$ (See Figure~\ref{fig-angulo}). Notice also that when $\alpha=1$, then the value $\pi$ is attained by different values of $c_0$.

The next natural question is the injectivity of the application $c_0\longrightarrow \theta_{c_0,\alpha}$, for fixed $\alpha$. Precisely, can we generate the same angle using different values of $c_0$?
In the case $\alpha=0$, the plot of $\theta_{c_0,0}$ in Figure~\ref{fig-angulo} shows that the value of $c_0$ is unique, in fact one has following formula  $\sin{(\theta_{c_0,0}/2)}=A_{1,c_0,0}=e^{-\frac{c_0^2}{2}\pi}$ (see~\cite{vega-gutierrez}). In the case $\alpha=1$,
we have $\sin{(\theta_{c_0,1}/2)}= A_{1,c_0,1}=\cos(c_0\sqrt{\pi)}$, moreover
\bq\label{periodico}
\vec A_{c_0,1}^+=\vec A_{c_0+2k\sqrt{\pi},1}^+, \quad\text{ for any }k\in \Z.
\eq
As before, if $\alpha\in(0,1)$ we do not have an analytic answer
and we have to rely on numerical simulations. However, it is difficult to test the uniqueness of $c_0$ numerically. Using the command
{\verb FindRoot } in Mathematica, we have found such values. For instance, for $\alpha=0.4$, we obtain that
$c_0\approx 2.1749$ and  $c_0\approx 6.6263$ give the same value of $\vec A_{c_0,0.4}^+$. The respective profiles $\m_{c_0,0.4}(\cdot)$ are shown
in Figure~\ref{profiles-m}. This multiplicity of solutions suggests that the Cauchy problem for \eqref{LL}
with initial condition \eqref{initial-condition} is ill-posed, at least for certain values of $c_0$. This interesting problem will be
studied in a forthcoming paper.

\begin{figure}[ht]
\begin{center}
 \subfloat[$\m_{c_0,0.4}(\cdot)$, with $c_0\approx 2.1749$]
{\begin{overpic}[scale=0.7]{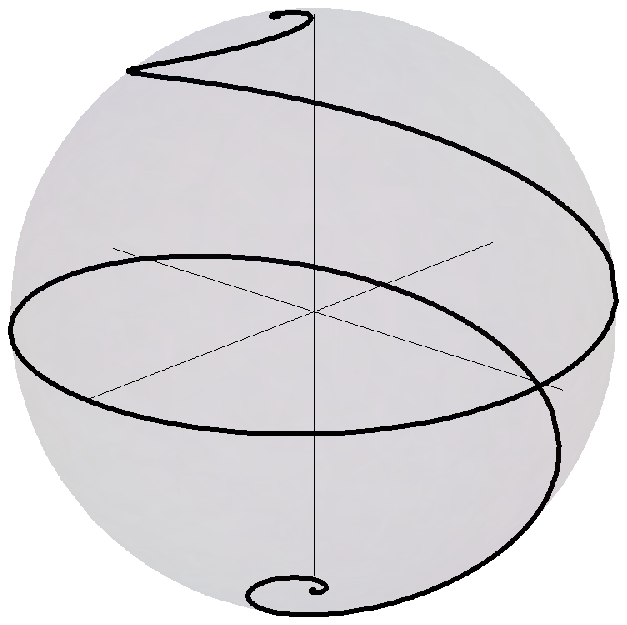}
\put(11,46){\tiny{$m_1$}}
\put(85,47){\tiny{$m_2$}}
\put(48,92){\tiny{$m_3$}}
\end{overpic}
}
 \subfloat[$\m_{c_0,0.4}(\cdot)$, with $c_0\approx 6.6263$]
{
\begin{overpic}[scale=0.7]{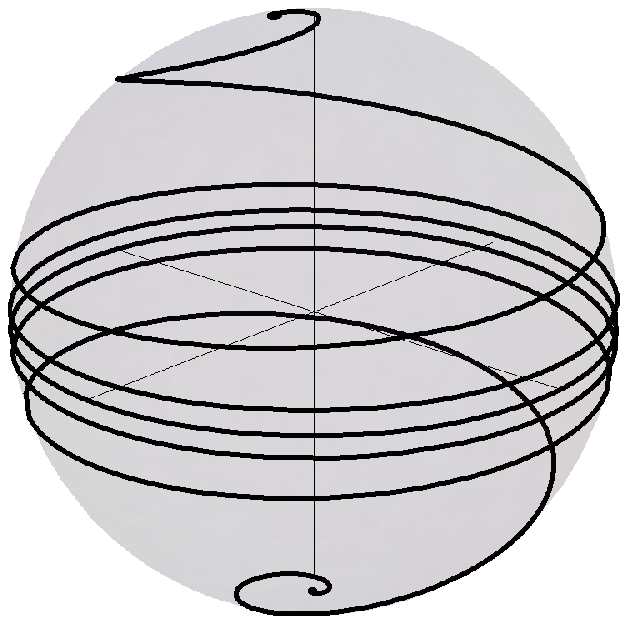}
\put(11,46){\tiny{$m_1$}}
\put(85,47){\tiny{$m_2$}}
\put(48,92){\tiny{$m_3$}}
\end{overpic}
}
\end{center}
\caption{Two profiles $\m_{c_0,0.4}(\cdot)$, with the same limit vector $\vec A_{c_0,0.4}^+$.}
\label{profiles-m}
\end{figure}
%
%
The rest of this section is devoted to give some numerical results on the behaviour of the limiting vector  $\vec{A}^{+}_{c_0,\alpha}$.
In particular, the results below aim to complement those established in Theorem~\ref{thm-c-0} on the behaviour of $\vec A^{+}_{c_0,\alpha}$ for small values of $c_0$, when  $\alpha$ is fixed.

We start recalling what it is known in the extremes cases $\alpha=0$ and $\alpha=1$. Precisely, if $\alpha=0$, the explicit formulae \eqref{A1-alpha-0}--\eqref{A3-alpha-0} for $\vec A^+_{c_0,0}$ allow us to prove that
\begin{equation}\label{Im-A3}
\lim_{c_0\to 0^+} A^+_{3,c_0,0}=0 \quad \text{and} \quad \lim_{c_0\to\infty} A^+_{3,c_0,1}=1,
\end{equation}
and also that $\{A^+_{3,c_0,0} : c_0\in (0,\infty)\}=(0,1)$. When $\alpha=1$ the picture is completely different.
 In fact $A^+_{3,c_0,1}=0$  for all $c_0> 0$, and the limit vectors remain in the equator plane $\S^1\times\{0\}$.
The natural question is what happens with $\vec A^+_{c_0,\alpha}$ when $\alpha\in (0,1)$ as a function of $c_0$.

Although we do not provide a rigorous  answer to this question, in Figure~\ref{fig-A-c} we show some numerical results. Precisely,
Figure~\ref{fig-A-c} depicts the curves $\vec A^+_{c_0,0.01}$, $\vec A^+_{c_0,0.4}$ and
$\vec A^+_{c_0,0.8}$ as functions of $c_0$, for $c_0\in[0,1000]$. We see that the behaviour of $\vec A^+_{c_0,\alpha}$ changes when $\alpha$ increases
in the sense that the first and second coordinates start oscillating more and more as $\alpha$ goes to 1. In all the cases the third component
remains monotonically increasing with $c_0$, but the value of $A^+_{3,1000,\alpha}$ seems to be decreasing with $\alpha$.
At this point it is not clear what the limit value of $A^+_{3,c_0,\alpha}$ as $c_0\to \infty$ is. For this reason, we perform
a more detailed analysis of $A^+_{3,c_0,\alpha}$ and we show the curves $A^+_{3,1,\alpha}$, $A^+_{3,10,\alpha}$, $A^+_{3,1000,\alpha}$ (for fixed $\alpha\in[0,1]$)
in Figure~\ref{fig-A-3-c}. From these results we conjecture that $\{A^+_{3,c_0,\cdot}\}_{c_0>0}$
is a pointwise nondecreasing sequence of functions that converges to $1$ for any $\alpha<1$  as $c_0\to \infty$.
This would imply that, for $\alpha\in(0,1)$ fixed,  $A_{1,c_0,\alpha}\to 0$ as $c_0\to \infty$,
and since  $A_{1,c_0,\alpha}\to 1$ as $c_0\to 0$ (see \eqref{A1}),
we could conclude by continuity (see Theorem~\ref{thm-c-0}) that for any angle $\theta\in (0,\pi)$  there exists $c_0>0$ such that
$\theta$ is the angle between $\vec A^+_{c_0,\alpha}$ and $-\vec A^+_{c_0,\alpha}$ (see \eqref{theta}). This provides an alternative way
to justify the surjectivity of the map $c_0\mapsto  \vec A^{+}_{c_0,\alpha}$ (in the sense explained above).

\begin{figure}[h]  %
\hspace*{-1cm}
\subfloat[$\vec  A^+_{c_0,0.01}$]{
\begin{overpic}[scale=0.7
]{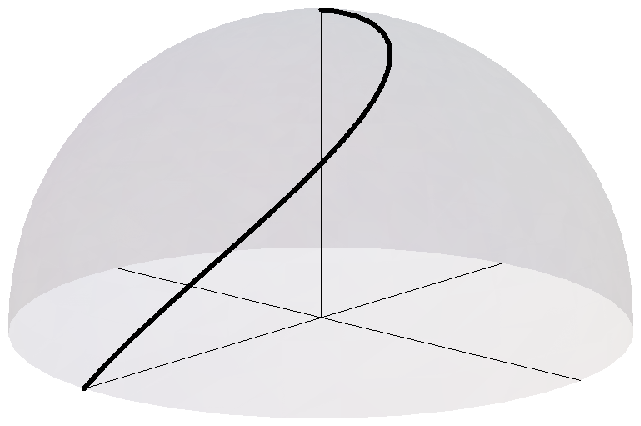}
\put(17,6){\tiny{$A^+_1$}}
\put(82,7){\tiny{$A^+_2$}}
\put(49,55){\tiny{$A^+_3$}}
\end{overpic}
}
\hspace{-1cm}
\subfloat[$\vec  A^+_{c_0,0.4}$]{
\begin{overpic}[scale=0.7]{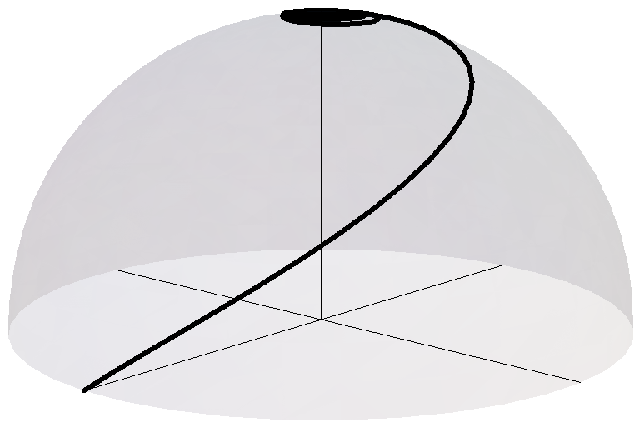}
\put(17,6){\tiny{$A^+_1$}}
\put(82,7){\tiny{$A^+_2$}}
\put(49,55){\tiny{$A^+_3$}}
\end{overpic}
}\hspace*{-1cm}\subfloat[$\vec A^+_{c_0,0.8}$]{
\begin{overpic}[scale=0.7]{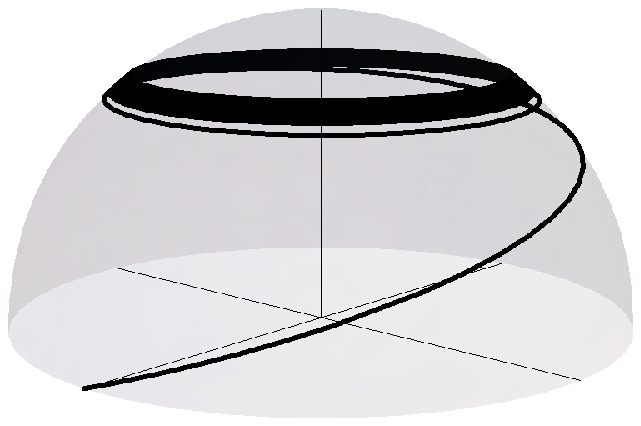}
\put(17,6){\tiny{$A^+_1$}}
\put(82,7){\tiny{$A^+_2$}}
\put(49,55){\tiny{$A^+_3$}}
\end{overpic}
}
\caption{The curves $\vec A^+_{c_0,0.01}$, $\vec A^+_{c_0,0.4}$ and
$\vec A^+_{c_0,0.8}$ as functions of $c_0$, for $c_0\in[0,1000]$.}
\label{fig-A-c}
\end{figure}

\begin{figure}[h]
\begin{center}
\begin{overpic}[scale=0.7]{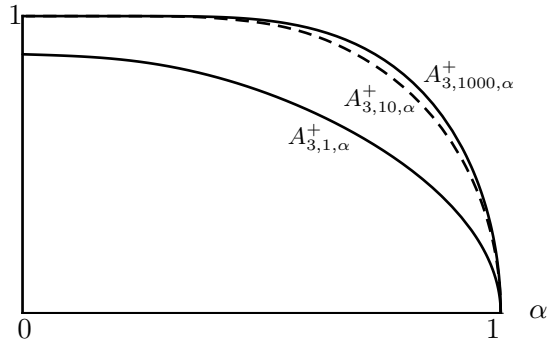}
\put(-1,-5){0}
\put(-3,60){1}
\put(96,-5){1}
\put(105,-1){$\alpha$}
\put(55,35){\footnotesize$A^+_{3,1,\alpha}$}
\put(66.5,43){\footnotesize$A^+_{3,10,\alpha}$}
\put(83,48){\footnotesize$A^+_{3,1000,\alpha}$}
\end{overpic}
\end{center}
\caption{The curves $A^+_{3,1,\alpha}$ ,$A^+_{3,10,\alpha}$, $A^+_{3,1000,\alpha}$ as functions of $\alpha$, for $\alpha\in[0,1]$.}
\label{fig-A-3-c}
\end{figure}

The curves in Figure~\ref{fig-A-3-c} also allow us to discuss further the results in Theorem~\ref{thm-alpha1-2}. In fact, when $\alpha$ is close to 1
the slope of the functions become unbounded and, roughly speaking, the behaviour of $A^+_{3,c_0,\alpha}$ is in agreement with the result in Theorem~\ref{thm-alpha1-2}, that is
$$A^+_{3,c_0,\alpha}\sim C(c_0)\sqrt{1-\alpha}, \quad \text{ as } \alpha\to 1^-.$$

Numerically, the analysis  is more difficult when $\alpha\sim 0$, because the number of computations needed to have an accurate profile of $A^+_{3,c_0,\alpha}$
increases drastically as $\alpha\to 0^+$. In any case, Figure~\ref{fig-A-3-c} suggests that $A^+_{3,c_0,\alpha}$ converges to $A^+_{3,c_0,0}$
 faster than $\sqrt{\alpha}\abs{\ln(\alpha)}$. We think that this rate of convergence can be improved to $\alpha\abs{\ln(\alpha)}$. In fact, in the proof of
 Lemma~\ref{lemma-z-alpha} we only used energy estimates. Probably,
taking into account the oscillations in equation \eqref{f-alpha} (as did in Proposition~\ref{prop-asymp}), it would be possible to establish the  necessary estimates
 to prove the following conjecture:
$$\abs{\vec A^+_{c_0,\alpha}-\vec A^+_{c_0,0}}\leq C(c_0)\alpha\abs{\ln(\alpha)},\quad \text{ for }\alpha\in(0,1/2].$$

\section{Appendix}\label{appendix}
%
\renewcommand\thethm{\thesection.\arabic{thm}}

In this appendix we show how to compute explicitly the solution $\mm_{c_0,\alpha}(s,t)$ of the LLG equation in the case $\alpha=1$.
As a consequence, we will obtain an explicit formula for the limiting vector $\vec A^+_{c_0,1}$ and the other constants appearing in the asymptotics of the associated profile established in Theorem~\ref{thm-conver} in terms of the parameter $c_0$ in the case when $\alpha=1$.

We start by recalling that if $\alpha=1$ then $\beta=0$. We need to find the solution $\{\m,\n,\b\}$ of the Serret--Frenet system \eqref{serret} with
$c(s)=c_0 e^{-s^2/4}$, $\tau\equiv 0$ and the initial conditions  \eqref{IC}.
Hence, it is immediate that
\bqq
m_3=n_3\equiv 0, \quad  b_1=b_2\equiv 0 \quad \text{ and } \quad b_3\equiv 1.
\eqq
 To compute the other components, we use the
Riccati equation \eqref{eq-eta} satisfied by the stereographic projection of $\{m_j, n_j, b_j   \}$
\bq\label{eta-j}
\eta_j=\frac{n_j+ib_j}{1+m_j}, \qquad {\hbox{for}}\qquad j\in \{ 1,2\},
\eq
found in the proof of Lemma~\ref{def-f}. For the values of curvature and torsion $c(s)= c_0e^{-s^2/4}$ and $\tau(s)=0$ the Riccati equation \eqref{eq-eta} reads
\begin{equation}\label{eq-eta-1}
\eta_j'+\frac{i\beta s}{2} \eta_j+\frac{c_0}2e^{-\alpha s^2/4}(\eta_j^2+1)=0.
\end{equation}
We see that when $\alpha=1$, and thus $\beta=0$, \eqref{eq-eta-1} is a separable equation that we write as:
\bqq
\frac{d\eta_j}{\eta_j^2+1}=-\frac{c_0}{2}e^{-\alpha s^2/4},
\eqq
so integrating, we get
\bq\label{eta-tan}
\eta_j(s)=\tan\left( \arctan(\eta_j(0))-\frac{c_0}{2} \Erf(s)\right),
\eq
where $\Erf(s)$ is the non-normalized error function
$$
 \Erf(s)=\int_0^s e^{-\sigma^2/4}\, d\sigma.
$$
Also, using \eqref{IC} and \eqref{eta-j} we get the initial conditions $\eta_1(0)=0$ and $\eta_2(0)=1$. In particular,  if $c_0$ is small
\eqref{eta-tan} is the global solution of the Riccati equation, but it blows-up in finite time if $c_0$ is large. As long as $\eta_j$ is well-defined,
by Lemma~\ref{def-f},
\begin{align*}
f_j(s)&=e^{\frac{c_0}2\int_0^s e^{-\alpha \sigma^2/4}\eta_j(\sigma)\,d\sigma}.
\end{align*}
The change of variables $$\mu=\arctan(\eta_j(0))-\frac{c_0}{2} \Erf(s)$$ yields
\bqq
\int_0^s e^{-\alpha \sigma^2/4}\eta_j(\sigma)\,d\sigma=
\frac{2}{c_0}\ln\left|\frac{\cos\left(\arctan(\eta_j(0))-\frac{c_0}{2}\Erf(s)\right)}{\cos(\arctan(\eta_j(0)))}
\right|,
\eqq
and after some simplifications, we obtain
\bqq
f_1(s)=\left|\cos\left(\frac{c_0}{2}\Erf(s)\right)\right| \ \text{ and }\ f_2(s)=\left|\cos\left(\frac{c_0}{2}\Erf(s)\right)+\sin\left(\frac{c_0}{2}\Erf(s)\right)\right|.
\eqq
In view of \eqref{m-1} and \eqref{m-2}, we conclude that
\bq\label{m-expl}
m_1(s)=2\abs{f_1(s)}^2-1=\cos\left(c_0\Erf(s)\right) \ \text{ and }\ m_2(s)=\abs{f_2(s)}^2-1=\sin\left(c_0\Erf(s)\right).
\eq
A priori, the formulae in \eqref{m-expl} are valid only as long as $\eta$ is well-defined, but a simple verification show that these are the global
solutions of \eqref{serret}, with
\bqq
n_1(s)=-\sin\left(c_0\Erf(s)\right) \quad \text{ and }\quad n_2(s)=\cos\left(c_0\Erf(s)\right).
\eqq
In conclusion, we have proved the following:

\begin{prop}\label{prop-alpha-1} Let $\alpha=1$, and thus $\beta=0$. Then, the trihedron
 $\{ \m_{c_0,1},\n_{c_0,1},\b_{c_0,1}\}$ solution of \eqref{serret}--\eqref{IC} is given by
\begin{align*}
\m_{c_0,1}(s)&=(\cos(c_0\Erf(s)), \sin(c_0\Erf(s)),0),\\
\n_{c_0,1}(s)&=-(\sin(c_0\Erf(s)), \cos(c_0\Erf(s)), 0),\\
\b_{c_0,1}(s)&=(0,0,1),
\end{align*}
for all $s\in \R$.
In particular, the limiting vectors $\vec A^+_{c_0,1}$ and $ \vec A^-_{c_0,1}$ in Theorem~\ref{thm-conver} are given in terms of $c_0$ as follows:
$$\vec A^\pm_{c_0,1}=(\cos(c_0\sqrt\pi),\pm \sin(c_0\sqrt\pi),0 ).$$
\end{prop}

Proposition~\ref{prop-alpha-1} allows us to give an alternative explicit proof of Theorem~\ref{thm-conver} when $\alpha=1$.
\begin{cor}\label{prop-alpha-1-bis}{{\it{[Explicit asymptotics when $\alpha=1$]}}} With the same notation as in Proposition~\ref{prop-alpha-1}, the following asymptotics for $\{ \m_{c_0,1}, \n_{c_0,1}, \b_{c_0,1}  \}$ holds true:
\begin{align*}
\m_{c_0,1}(s)&=\vec A^+_{c_0,1}-\frac{2c_0}{s}\vec B^+_{c_0,1}e^{- s^2/4}\sin(\vec a)-\frac{2c_0^2}{s^2} \vec A^+_{c_0,1}e^{-s^2/2} +O\left(\frac{e^{-s^2/4}}{s^3}\right),\\
\n_{c_0,1}(s)&=\vec B^+_{c_0,1}\sin(\vec a)+\frac{2c_0}{s}\vec  A^+_{c_0,1} e^{- s^2/4} - \frac{2c_0^2}{s^2}\vec  B^+_{c_0,1}e^{-s^2/2} \sin(\vec  a) + O\left(\frac{e^{-s^2/4}}{ s^3}\right),\\
\b_{c_0,1}(s)&=\vec B^+_{c_0,1}\cos(\vec  a),
\end{align*}
where the vectors $\vec{A}^{+}_{c_0,1}$, $\vec{B}^{+}_{c_0,1}$ and $\vec{a}=(a_j)_{j=1}^{3}$ are given explicitly in terms of $c_0$ by
$$
 \vec A^+_{c_0,1}=(\cos(c_0\sqrt\pi),\sin(c_0\sqrt\pi),0),\qquad \vec B^+_{c_0,1}=(\abs{\sin(c_0\sqrt\pi)},\abs{\cos(c_0\sqrt\pi)},1),
$$
\bqq
  a_1=\begin{cases}
\frac{3\pi}{2}, & \text{ if }\ \sin(c_0\sqrt{\pi})\geq0,\\
\frac{\pi}{2}, & \text{ if }\ \sin(c_0\sqrt{\pi})<0,
\end{cases}
\qquad
 a_2=\begin{cases}
\frac{\pi}{2}, & \text{ if }\ \cos(c_0\sqrt{\pi})\geq 0,\\
\frac{3\pi}{2}, & \text{ if }\ \cos(c_0\sqrt{\pi})<0,
\end{cases}
\qquad {\hbox{and}}\qquad a_3=0.
\eqq
Here, the bounds controlling the error terms depend on $c_0$.
\end{cor}
\begin{proof}
By Proposition~\ref{prop-alpha-1},
\begin{eqnarray}\label{mnb-alpha-1}
\left\{
\begin{array}{l}
\m_{c_0,1}(s)=(\cos(c_0\sqrt{\pi}-c_0\Erfc(s)), \sin(c_0\sqrt{\pi}-c_0\Erfc(s)),0),\\
\n_{c_0,1}(s)=-(\sin(c_0\sqrt{\pi}-c_0\Erfc(s)), \cos(c_0\sqrt{\pi}-c_0\Erfc(s)), 0),\\
\b_{c_0,1}(s)=(0,0,1),
\end{array}
\right.
\end{eqnarray}
where the complementary error function is given by
$$\Erfc(s)=\int_s^\infty e^{-\sigma^2/4}\, d\sigma=\sqrt{\pi}-\Erf(s).$$
It is simple to check that
\begin{align*}
\sin(c_0\Erfc(s))&=e^{-s^2/4}\left(\frac{2c_0}{s}-\frac{4c_0}{s^3}+\frac{24c_0}{s^5}+O\left(\frac{c_0}{s^7}\right)\right),\\
\cos(c_0\Erfc(s))&=1+e^{-s^2/2}\left(-\frac{2c_0^2}{s^2}+\frac{8c_0^2}{s^4}-\frac{56c_0^2}{s^6}+O\left(\frac{c_0^2}{s^8}\right)\right),
\end{align*}
so that, using \eqref{mnb-alpha-1}, we obtain that
\begin{align*}\
 m_1(s)=n_2(s)=\cos(c_0\sqrt{\pi})+\frac{2c_0 }{s}e^{-s^2/4}\sin(c_0\sqrt{\pi})-\frac{2c_0^2 }{s^2}e^{-s^2/2}\cos(c_0\sqrt{\pi})+O\left(\frac{e^{-s^2/4}}{s^3}\right),\\
 m_2(s)=-n_1(s)=\sin(c_0\sqrt{\pi})-\frac{2c_0 }{s}e^{-s^2/4}\cos(c_0\sqrt{\pi})-\frac{2c_0^2 }{s^2}e^{-s^2/2}\sin(c_0\sqrt{\pi})+O\left(\frac{e^{-s^2/4}}{s^3}\right).
\end{align*}
The conclusion follows from the definitions of $\vec A^+_{c_0,1}$, $\vec B^+_{c_0,1}$ and $\vec a$.
\end{proof}

\begin{remark}
Notice that $\vec a$ is not a continuous function of $c_0$, but the vectors $(B_j^+\sin(a_j))_{j=1}^3$  and  $(B_j^+\cos(a_j))_{j=1}^3$ are.
\end{remark}

\bibliography{ref}

\begin{thebibliography}{10}

\bibitem{abram}
M.~Abramowitz and I.~A. Stegun.
\newblock {\em Handbook of mathematical functions with formulas, graphs, and
  mathematical tables}, volume~55 of {\em National Bureau of Standards Applied
  Mathematics Series}.
\newblock For sale by the Superintendent of Documents, U.S. Government Printing
  Office, Washington, D.C., 1964.

\bibitem{banica-vega1}
V.~Banica and L.~Vega.
\newblock On the {D}irac delta as initial condition for nonlinear
  {S}chr\"odinger equations.
\newblock {\em Ann. Inst. H. Poincar\'e Anal. Non Lin\'eaire}, 25(4):697--711,
  2008.

\bibitem{brezis-mon}
H.~Br{\'e}zis.
\newblock {\em Op\'erateurs maximaux monotones et semi-groupes de contractions
  dans les espaces de {H}ilbert}.
\newblock North-Holland Publishing Co., Amsterdam, 1973.
\newblock North-Holland Mathematics Studies, No. 5. Notas de Matem{\'a}tica
  (50).

\bibitem{buttke}
T.~F. Buttke.
\newblock A numerical study of superfluid turbulence in the self-induction
  approximation.
\newblock {\em Journal of Computational Physics}, 76(2):301--326, 1988.

\bibitem{coddington}
E.~A. Coddington and N.~Levinson.
\newblock {\em Theory of ordinary differential equations}.
\newblock McGraw-Hill Book Company, Inc., New York-Toronto-London, 1955.

\bibitem{daniel-lak1}
M.~Daniel and M.~Lakshmanan.
\newblock Soliton damping and energy loss in the classical continuum heisenberg
  spin chain.
\newblock {\em Physical Review B}, 24(11):6751--6754, 1981.

\bibitem{daniel-lak}
M.~Daniel and M.~Lakshmanan.
\newblock Perturbation of solitons in the classical continuum isotropic
  {H}eisenberg spin system.
\newblock {\em Physica A: Statistical Mechanics and its Applications},
  120(1):125--152, 1983.

\bibitem{darboux}
G.~Darboux.
\newblock {\em Le\c cons sur la th\'eorie g\'en\'erale des surfaces. {I},
  {II}}.
\newblock Les Grands Classiques Gauthier-Villars. [Gauthier-Villars Great
  Classics]. \'Editions Jacques Gabay, Sceaux, 1993.
\newblock G{\'e}n{\'e}ralit{\'e}s. Coordonn{\'e}es curvilignes. Surfaces
  minima. [Generalities. Curvilinear coordinates. Minimum surfaces], Les
  congruences et les {\'e}quations lin{\'e}aires aux d{\'e}riv{\'e}es
  partielles. Les lignes trac{\'e}es sur les surfaces. [Congruences and linear
  partial differential equations. Lines traced on surfaces], Reprint of the
  second (1914) edition (I) and the second (1915) edition (II), Cours de
  G{\'e}om{\'e}trie de la Facult{\'e} des Sciences. [Course on Geometry of the
  Faculty of Science].

\bibitem{delahoz}
F.~de~la Hoz, C.~J. Garc{\'{\i}}a-Cervera, and L.~Vega.
\newblock A numerical study of the self-similar solutions of the
  {S}chr\"odinger map.
\newblock {\em SIAM J. Appl. Math.}, 70(4):1047--1077, 2009.

\bibitem{germain-rupflin}
P.~Germain and M.~Rupflin.
\newblock Selfsimilar expanders of the harmonic map flow.
\newblock {\em Ann. Inst. H. Poincar\'e Anal. Non Lin\'eaire}, 28(5):743--773,
  2011.

\bibitem{gilbert}
T.~L. Gilbert.
\newblock A lagrangian formulation of the gyromagnetic equation of the
  magnetization field.
\newblock {\em Phys. Rev.}, 100:1243, 1955.

\bibitem{grunrock}
A.~Gr{\"u}nrock.
\newblock Bi- and trilinear {S}chr\"odinger estimates in one space dimension
  with applications to cubic {NLS} and {DNLS}.
\newblock {\em Int. Math. Res. Not.}, (41):2525--2558, 2005.

\bibitem{guan}
M.~Guan, S.~Gustafson, K.~Kang, and T.-P. Tsai.
\newblock Global questions for map evolution equations.
\newblock In {\em Singularities in {PDE} and the calculus of variations},
  volume~44 of {\em CRM Proc. Lecture Notes}, pages 61--74. Amer. Math. Soc.,
  Providence, RI, 2008.

\bibitem{guo}
B.~Guo and S.~Ding.
\newblock {\em Landau-{L}ifshitz equations}, volume~1 of {\em Frontiers of
  Research with the Chinese Academy of Sciences}.
\newblock World Scientific Publishing Co. Pte. Ltd., Hackensack, NJ, 2008.

\bibitem{vega-gutierrez}
S.~Guti{\'e}rrez, J.~Rivas, and L.~Vega.
\newblock Formation of singularities and self-similar vortex motion under the
  localized induction approximation.
\newblock {\em Comm. Partial Differential Equations}, 28(5-6):927--968, 2003.

\bibitem{vega-gutierrez1}
S.~Guti{\'e}rrez and L.~Vega.
\newblock Self-similar solutions of the localized induction approximation:
  singularity formation.
\newblock {\em Nonlinearity}, 17:2091--2136, 2004.

\bibitem{hartman}
P.~Hartman.
\newblock {\em Ordinary differential equations}.
\newblock John Wiley \& Sons Inc., New York, 1964.

\bibitem{hasimoto}
H.~Hasimoto.
\newblock A soliton on a vortex filament.
\newblock {\em J. Fluid Mech}, 51(3):477--485, 1972.

\bibitem{hubert}
A.~Hubert and R.~Sch\"afer.
\newblock {\em Magnetic domains: the analysis of magnetic microstructures}.
\newblock Springer, 1998.

\bibitem{kosevich}
A.~M. Kosevich, B.~A. Ivanov, and A.~S. Kovalev.
\newblock Magnetic solitons.
\newblock {\em Physics Reports}, 194(3-4):117--238, 1990.

\bibitem{lakshmanan}
M.~Lakshmanan.
\newblock The fascinating world of the {L}andau-{L}ifshitz-{G}ilbert equation:
  an overview.
\newblock {\em Philos. Trans. R. Soc. Lond. Ser. A Math. Phys. Eng. Sci.},
  369(1939):1280--1300, 2011.

\bibitem{lakshmanan0}
M.~Lakshmanan, T.~W. Ruijgrok, and C.~Thompson.
\newblock On the dynamics of a continuum spin system.
\newblock {\em Physica A: Statistical Mechanics and its Applications},
  84(3):577--590, 1976.

\bibitem{lamb}
G.~L. Lamb, Jr.
\newblock {\em Elements of soliton theory}.
\newblock John Wiley \& Sons Inc., New York, 1980.
\newblock Pure and Applied Mathematics, A Wiley-Interscience Publication.

\bibitem{landaulifshitz}
L.~Landau and E.~Lifshitz.
\newblock On the theory of the dispersion of magnetic permeability in
  ferromagnetic bodies.
\newblock {\em Phys. Z. Sowjetunion}, 8:153--169, 1935.

\bibitem{lin}
F.~Lin and C.~Wang.
\newblock {\em The analysis of harmonic maps and their heat flows}.
\newblock World Scientific Publishing Co. Pte. Ltd., Hackensack, NJ, 2008.

\bibitem{lipniacki}
T.~Lipniacki.
\newblock Shape-preserving solutions for quantum vortex motion under localized
  induction approximation.
\newblock {\em Phys. Fluids}, 15(6):1381--1395, 2003.

\bibitem{steiner}
M.~Steiner, J.~Villain, and C.~Windsor.
\newblock Theoretical and experimental studies on one-dimensional magnetic
  systems.
\newblock {\em Advances in Physics}, 25(2):87--209, 1976.

\bibitem{struik}
D.~J. Struik.
\newblock {\em Lectures on {C}lassical {D}ifferential {G}eometry}.
\newblock Addison-Wesley Press, Inc., Cambridge, Mass., 1950.

\bibitem{vargas-vega}
A.~Vargas and L.~Vega.
\newblock Global wellposedness for 1{D} non-linear {S}chr\"odinger equation for
  data with an infinite {$L^2$} norm.
\newblock {\em J. Math. Pures Appl. (9)}, 80(10):1029--1044, 2001.

\end{thebibliography}
 \bibliographystyle{abbrv}

\end{document}